\documentclass[12pt,reqno]{amsart}

\usepackage[hmargin=2cm,bmargin=2cm,tmargin=2.5cm]{geometry}
\usepackage{graphicx}
\usepackage{subfigure}
\usepackage{color}
\usepackage{enumerate}
\usepackage[english]{babel}
\usepackage{mathtools}
\usepackage[utf8]{inputenc}
\usepackage[T1]{fontenc}
\usepackage[b]{esvect}
\usepackage{parskip}
\usepackage{bm}
\usepackage{array}
\usepackage{arydshln}
\usepackage{multirow}
\usepackage{bigstrut}
\usepackage{bbm}
\usepackage{esdiff}
\usepackage{relsize}
\usepackage{hyperref}
\usepackage{cleveref}

\newtheorem{thm}{Theorem}[section]
\newtheorem{lem}[thm]{Lemma}

\newtheorem{cor}[thm]{Corollary}

\theoremstyle{definition}
\newtheorem{defn}{Definition}
\newtheorem{exmp}{Example}
\newtheorem{remark}[thm]{Remark}

\newcommand*{\Scale}[2][4]{\scalebox{#1}{$#2$}}%
\newcommand{\cc}[1]{\overline{#1}}

\newcommand{\cB}{\mathcal{B}}

\newcommand{\cH}{\mathcal{H}}

\newcommand{\cK}{\mathcal{K}}

\newcommand{\cO}{\mathcal{O}}

\newcommand{\cR}{\mathcal{R}}

\newcommand{\cU}{\mathcal{U}}

\newcommand{\bC}{\mathbb{C}}

\newcommand{\bI}{\mathbb{I}}

\newcommand{\bR}{\mathbb{R}}

\def\up#1{^{(#1)}}

\DeclareMathOperator{\Dom}{Dom}
\DeclareMathOperator{\Hom}{Hom}

\newcommand{\vt}{{\mathfrak{v}}}
\newcommand{\Dmax}{{\widetilde{\mathfrak{H}}(\Gamma)}}
\newcommand{\dmax}{{\widetilde{\mathfrak{h}}(\Gamma)}}
\newcommand{\Ltwo}{{L^2(\Gamma)}}
\newcommand{\Prho}{\mathbf{P}_{\!\rho}}

\title[Elastic beam frames: variational and differential
formulation]{Three dimensional elastic beam frames: \\ rigid joint
  conditions in variational and differential formulation}

\author[G.~Berkolaiko]{Gregory Berkolaiko}
\address{Department of Mathematics, Texas A\&M University, College Station, TX 77843-3368, USA}

\author[M.~Ettehad]{Mahmood Ettehad}
\address{Institute for Mathematics and its Applications (IMA), University of Minnesota, Minneapolis, MN 55455, USA}

\begin{document}

\begin{abstract}
  We consider three-dimensional elastic frames constructed out of
  Euler--Bernoulli beams and describe a simple process of generating
  joint conditions out of the geometric description of the frame. The
  corresponding differential operator is shown to be self-adjoint. In
  the special case of planar frames, the operator decomposes into a
  direct sum of two operators, one coupling out-of-plane displacement
  to angular (torsional) displacement and the other coupling in-plane
  displacement with axial displacement (compression).

  Detailed analysis of two examples is presented. We actively exploit
  the symmetry present in the examples and decompose the operator by
  restricting it onto reducing subspaces corresponding to irreducible
  representations of the symmetry group. These ``quotient'' operators
  are shown to capture particular oscillation modes of the frame. 
\end{abstract}

\maketitle

\section{Introduction}

Mathematical modeling of vibration of structures made of joined
together beams is a topic of natural interest for engineering research
and, more recently, for mathematicians working on differential
equations defined on metric graphs. Each beam is described by an
Euler--Bernoulli energy functional\footnote{Or, allowing for more
  degrees of freedom, by a Timoshenko functional.  We will restrict
  our attention to Euler--Bernoulli theory here.} which involves four
degrees of freedom for every infinitesimal element along the beam:
lateral displacement (2 degrees of freedom), axial displacement, and angular displacement, see Figure \ref{fig:dofs} for schematic
representation of these degrees of freedom.  At a joint, these four
functions, supported on the beams involved, must be related via
matching conditions that take into account the physics of the joint.
The question of describing correct matching conditions is of central
importance.  So far, most of the engineering literature has been
dealing with planar structures (see
\cite{MM05,Mei_jva12,LR13,GLL17,TSOGM20} and references therein) and a
recent breakthrough into 3 dimensions has been restricted to
rectangular structures \cite{Mei19}. In mathematical literature, in
addition to the planarity assumption, a further restriction has been
imposed by considering out-of-plane displacements alone (see
\cite{DN00,BL04,MR08,KiiKurUsm_pla15,GM20}).  As we will see in this
paper (and as observed in the engineering literature from the start),
this is a significant departure from the mechanical model, since
out-of-plane displacements are coupled to angular displacement in all
non-trivial structures.

\begin{figure}
  \centering
  \includegraphics[width=0.925\textwidth]{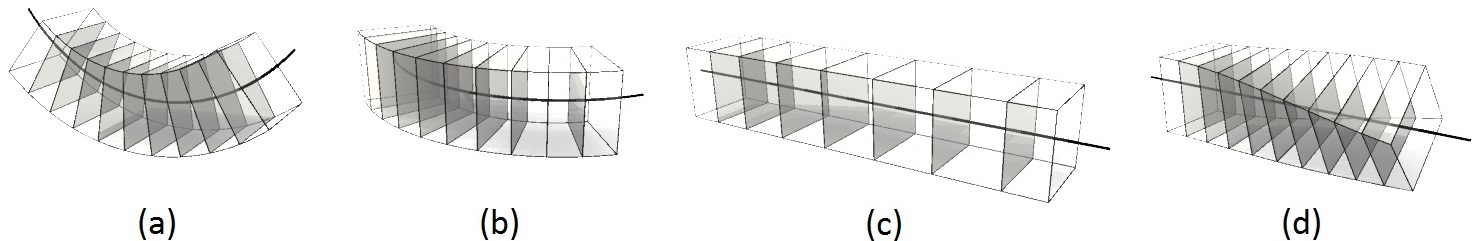}
  \caption{Degrees of freedom associated to Euler-Bernoulli beam,
    (a,b) lateral displacements, (c) axial displacement, and (d) angular displacement.}
  \label{fig:dofs}
\end{figure}

In this paper we focus mostly on studying the rigid joint, a joint
where relative angles of participating beams remain constant
throughout the motion. We put forward a simple, general and compact description of the matching conditions at the rigid joint (and some of its variants), both for the Lagrangian and Hamiltonian formulations. Mathematicaly speaking, we describe the closed symmetric sesquilinear
form associated with the model as well as the corresponding
self-adjoint operator. The results are formulated for a general
3-dimensional structure; in the planar case the Hamiltonian decomposes
into a direct sum involving two pairs of coupled degrees of freedom:
out-of-plane displacement and angular displacement on one side,
in-plane and axial displacements on the other.

To illustrate our results we present two examples, one planar structure and one 3-dimensional structure (``antenna tower''). In the latter example, a significant role is played by the symmetry.  We decompose the Hamiltonian operator by restricting it onto reducing subspaces corresponding to irreducible representations of the symmetry group. Numerical simulations of the entire structure are used to illustrate the corresponding decomposition of the vibrational spectrum of the tower.

While maintaining mathematical rigor, we made every effort to
highlight the practical aspects of every theorem. Some
technicalities are deferred to the Appendix.  When presenting results
of numerical simulations, we compared the results of direct finite
element computations (see, for instance, \cite{AB17} for analysis of
the method on network structures carrying a scalar second order
equation) with the results of the method based on characteristic (or
secular) equation \cite{WilWit_ijms70,WitWil_qjmam71,Bel_laa85} (also
known as ``dynamic stiffness matrix method'').  In
Section~\ref{sec:outlook} we present a partial list of directions for
further mathematical investigations together with additional
references.

%%%%%%%%%%%%%%%%%%%%%%%%%%%%%%%%%%%%%%%%%%%%%%%%%%%%%%%%%%%%%%%%%%
\section{Elasticity of beams and beam frames}

Throughout the manuscript, $\{\vec E_1, \vec E_2, \vec E_3\}$ will
denote the fixed orthonormal basis which spans the physical three
dimensional Euclidean space in which the beam frame is
embedded.

%%%%%%%%%%%%
\subsection{Deformation of an Euler--Bernoulli beam and rigid joint
  conditions}

According to the Euler--Bernoulli hypothesis, which states that
``plane sections remain plane'', the geometry of the spatial beam is
described by the centroid line\footnote{The term ``centroid''
  indicates that this line is the locus of the centers of mass of the
  cross-sections.} and a family of the cross-sections orthogonal to
it.  We assume the cross-section of the beam is arbitrary but unchanged
along the axis and that the beam is straight in its undeformed
configuration.  In this case, each undeformed beam is described by an
orthonormal basis $\{\vec i,\vec j, \vec k\}$, called the
cross-section basis or the material basis. The base vector $\vec i$ is
directed along the centroid line and the vectors $\vec j$ and $\vec k$
are directed along the principal axes of inertia of the
cross-section. The basis is assumed to be right-handed, i.e.
$\vec i = \vec j \times \vec k$.  The choice of orientation, implicit
in the choice of the overall sign, is arbitrary but fixed.
\begin{figure}
  \centering
  \includegraphics[width=0.6\textwidth]{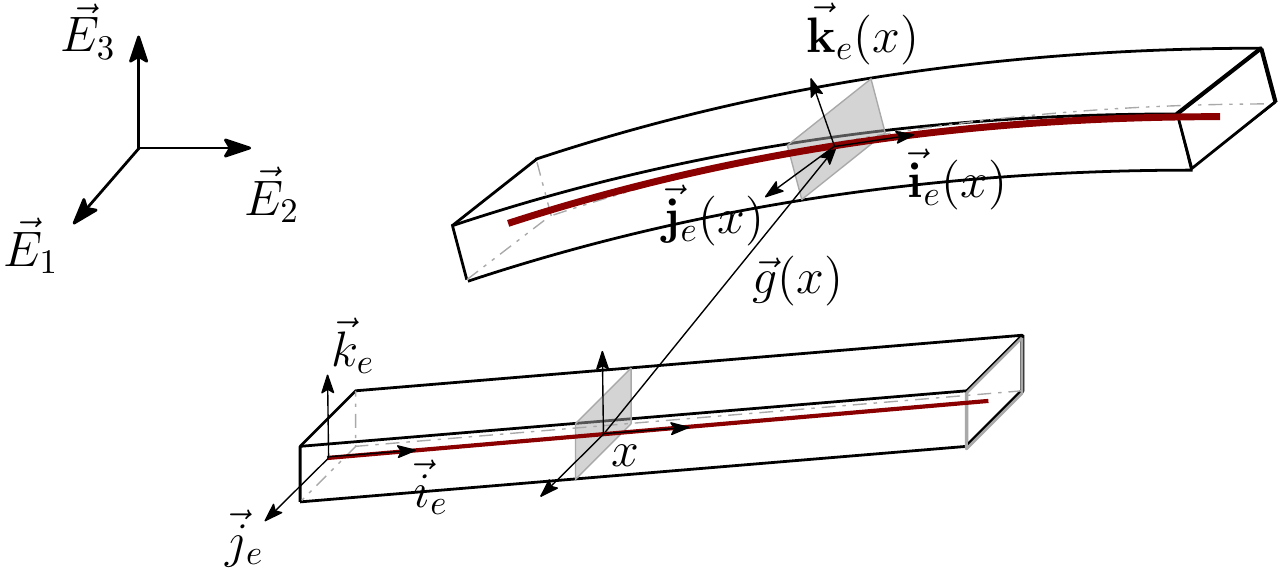}
  \caption{Local coordinates of a beam in its initial and deformed
    configurations.}
  \label{fig:positionRotVects}
\end{figure}
The deformed configuration of the beam can be fully described by the
displacement vector $\vec g(x)$ of the centroid line along with the
family of orthonormal bases
$\{\bm{\vec i}(x),\bm{\vec j}(x),\bm{\vec k}(x)\}$ which describe the
orientation of the cross sections.  Here $x$ represents the arc-length
coordinate along the centroid of the reference (undeformed)
configuration, see Figure~\ref{fig:positionRotVects}. 

Since the Euler--Bernoulli cross sections remain orthogonal to the
centroid line, the above parameterization has redundancy (assuming the
centroid remains smooth).  In what follows we will describe
alternative representations of the deformed configuration, finally
removing all redundancy in Theorem~\ref{primaryCond}.

The relationship between the cross-section basis in the reference and
the deformed configurations can be expressed as
\begin{equation}
  \label{eq:deformedRot}
  \bm{\vec i}(x) = \cR(x) \vec i,
  \quad \quad
  \bm{\vec j}(x) = \cR(x) \vec j,
  \quad \quad
  \bm{\vec k}(x) = \cR(x) \vec k
\end{equation}
where the \textbf{rotation transformation} $\cR(x)$ is an element of
$\text{SO}(3)$, the Lie group of proper orthogonal linear
transformations.

Consider now several beams, labeled $e_1, \ldots, e_n$ meeting at a
joint.  The joint is rigid if it
moves as a whole, preserving both the connection and the relative
angles of all incident beams\footnote{In engineering, this is called ``fully rigid joint'' to distinguish from other situations when beams may experience limited bending at the attachment points}.
Figure~\ref{fig:jointTypes3} shows an example of a rigid joint.
Mathematically, the deformed configuration of a joint is described by
its displacement and rotation which are obtained as the (coinciding!)
limits of the corresponding vectors along any beam attached to the
joint \cite{LLS92,Lagnese_multilink}.

\begin{figure}[h]
  \centering \includegraphics[width=0.575\textwidth]{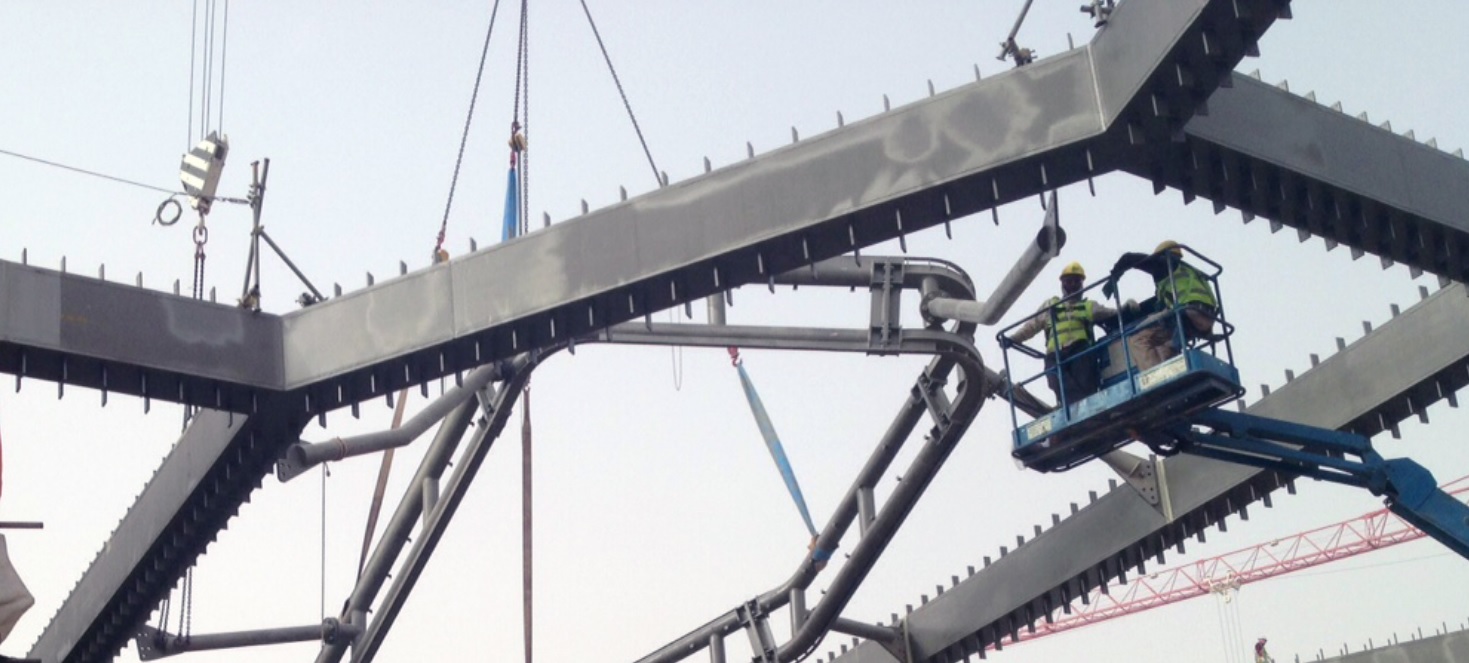}
  \caption{Connection of beams by rigid joints \href{https://www.studioavc.com/portfolio-items/king-abdullah-petroleum-studies-and-research-centre-2/?portfolioID=22}{(source)}.}
  \label{fig:jointTypes3}
\end{figure}

\begin{defn}
  \label{def:rigid}
  A joint $\vt$ with $n$ incident beams $\{e_j\}_{j=1}^n$ is called
  \textbf{rigid}, if the displacement and rotation on beams
  $e_j$ satisfy
  \begin{subequations}
    \label{rigidVertexDefn}
    \begin{equation}
      \label{eq:rigid_displacement}
      \vec g_{e_1}(\vt) = \cdots = \vec g_{e_n}(\vt),
    \end{equation}
    \begin{equation}
      \label{eq:rigid_rotation}
      \cR_{e_1}(\vt) = \cdots = \cR_{e_n}(\vt),
    \end{equation}
  \end{subequations}
  where
  \begin{equation}
    \label{eq:limits}
    \vec g_{e_j}(\vt) = \lim_{x\to\vt} \vec g_{e_j}(x),
    \qquad
    \cR_{e_j}(\vt) = \lim_{x\to\vt} \cR_{e_j}(x).
  \end{equation}
  We call a rigid joint \textbf{free} if it has no further
  restrictions imposed on it.
\end{defn}

Note that the definition assumes that the limits in \eqref{eq:limits}
exist; this will require a certain degree of smoothness of the
functions defined on the beams, see the discussion
preceding~\eqref{eq:vertex_limits} below.

%%%%%%%%%%%%%%%%%%%%%%%%%%%%%%%%%%%%%%%%%%%%%%%%%%%%%%%%%%%
\subsection{Euler--Bernoulli energy functional}

We now represent the centroid displacement vector $\vec{g}(x)$ in the reference basis of the undeformed beam,
\begin{equation}
  \label{dispComp}
  \vec{g}(x) =: u(x)\vec{i} + w(x)\vec{j} + v(x)\vec{k}.
\end{equation}
The component $u(x)$ is called the \textbf{axial displacement} and
$w(x)$ and $v(x)$ are \textbf{lateral displacements}.  We also
introduce the \textbf{in-axis angular
  displacement}\footnote{Alternative terminology includes names such
  as ``longitudinal deflection'' for $u(x)$, ``transverse deflection''
  for $(w(x), v(x))$ and ``torsional deflection'' for $\eta(x)$.}
\begin{equation}
  \label{eq:torsion_def}
  \eta(x) := \bm{\vec{j}}(x) \cdot \vec{k}
  = \cR(x) \vec{j} \cdot \vec{k}.
\end{equation}

In the context of the kinematic Bernoulli assumptions for the beam, no pre-stress or external force, the total strain energy of small
deformations of the beam can be expressed as (see, e.g. \cite[section
5.3.4]{GR15} or \cite{KarnovskiiLebed_FreeVib})
\begin{equation}
  \label{eq:EB_energy}
  \cU \up{e} := \int_e
  \Big(a(x) \big|v''(x)\big|^2 +
  b(x) \big|w''(x)\big|^2 + c(x)\big|u'(x)\big|^2 +
  d(x)\big|\eta'(x)\big|^2\Big) dx.
\end{equation}
The integration here is over the beam $e$, parameterized by the
arc-length $x\in[0,\ell]$. Interpretation of the parameters in the
energy functional are as follows: the parameter $a(x)$ represents the
bending stiffness about the local axis $\vec j$ at point $x$ along the
beam, $b(x)$ represents the bending stiffness about the local axis
$\vec k$, $c(x)$ represents the axial stiffness in the direction of
$\vec i$, and, finally, $d(x)$ represents the angular displacement
stiffness in the direction of $\vec i$. Throughout the rest of the
manuscript we assume each beam in our frame is homogeneous in the
axial direction and hence
\begin{equation}
  \label{eq:material_params}
  a(x) \equiv a_e,
  \qquad
  b(x) \equiv b_e,
  \qquad
  c(x) \equiv c_e,
  \qquad\text{and}\qquad
  d(x) \equiv d_e
\end{equation}
are real, positive constants. Extension of all results to variable
stiffness is straightforward.

Some remarks are in order.  First, the separation of lateral
displacements $v(x)$ and $w(x)$ in the energy functional is a
consequence of choosing the axes $\vec j$ and $\vec k$ to coincide
with the principal axes of inertia of the cross-section.  Second, as
a consequence of Euler--Bernoulli hypothesis, only one component of
$\cR(x)$ participates in the strain energy.  The other components are
derived quantities which will be expressed in terms of the centroid
displacement $\vec{g}(x)$ when we describe $\cR(x)$ via linearized
rotation vector.

%%%%%%%%%%%%%%%%%%%%%%%%%%%%%%%%%%%%%%%%%%%%%%%%%%%
\subsection{Beam frames with rigid joints}
\label{sec:frames_description}

A \textbf{beam frame} is a collection of beams connected at joints.
We will describe a beam frame as a geometric graph $\Gamma = (V,E)$,
where $V$ denotes the set of vertices and $E$ the set of edges.  The
vertices $\vt \in V$ correspond to joints and edges $e \in E$ are the
beams.  Each edge $e$ is a collection of the following information:
origin and terminus vertices $\vt_{e}^{o},\vt_{e}^{t}\in V$, length
$\ell_e$ and the local basis
$\{\vec{i}_e, \vec{j}_e, \vec{k}_e\}$.  Describing the
vertices $V$ as points in $\bR^3$ also fixes the length $\ell_e$ and
the axial direction $\vec{i}_e$ (from origin to terminus); the choice
of $\vec{j}_e$ in the plane orthogonal to $\vec{i}_e$ still needs to
be specified externally.

The distinction between origin and terminus, and thus the direction of
$\vec i_e$ is unimportant in analysis but should be fixed for
consistency.  It is important to use the same beam basis when writing
out joint conditions at both ends of the beam.  We will use the
\textbf{signed incidence indicator} $s_{\vt}^e$ which is defined to be
$1$ when $\vt$ is the origin of $e$, $-1$ if it is the terminus of $e$
and $0$ otherwise.  We will also write $e\sim \vt$ when $\vt$ is
either the origin or the terminus of $e$.

The precise meaning of the terms like $\vec{g}_{e}(\vt)$ and
$\cR_{e_1}(\vt)$ in equations~\eqref{rigidVertexDefn} depends on
whether $\vt$ is the origin or the terminus of $e$.  In the sequel,
$\vec{g}_e$ will be sufficiently smooth (at least Sobolev $H^1$)
vector-function defined on the open interval $(0,\ell_e)$ and
\begin{equation}
  \label{eq:vertex_limits}
  \vec{g}_e(\vt_e^o) := \lim_{x\to0} \vec{g}_{e}(x),
  \qquad
  \vec{g}_e(\vt_e^t) := \lim_{x\to\ell_e} \vec{g}_{e}(x),
\end{equation}
and similarly for $\eta_e$.

Finally, the strain energy of the beam frame (assuming no vertex
energy) is a sum of energies of the individual beams,
\begin{equation}
  \label{eq:frame_energy}
  \cU \up{\Gamma} := \sum_{e\in E} \cU \up{e}
  = \sum_{e\in E} \int_e
  \Big(a_e \big|v''_e(x)\big|^2 +
  b_e \big|w''_e(x)\big|^2 + c_e\big|u'_e(x)\big|^2 +
  d_e\big|\eta'_e(x)\big|^2\Big) dx.
\end{equation}
The interaction between different beams affects the energy via
constraints \eqref{rigidVertexDefn} which are imposed on the domain of
$\cU \up{\Gamma}$.  We make this mathematically precise in the next
section.

We stress that in the present paper we assume that joints are rigid
and we focus primarily on the free rigid joints.  Other common cases
include pinned (displacement is zero) and clamped (displacement and
rotation are zero) rigid joints.  These extensions are briefly
discussed in section~\ref{sec:VertexEnergy}.  Other flavor of joint
conditions can happen in real world application and mathematical
modeling, see, for example, \cite{RBR13,MM05}. Pathways for extending
the current framework to more general joint conditions are discussed
in Remark~\ref{rem:all_conditions}.

%%%%%%%%%%%%%%%%%%%%%%%%%%%%%%%%%%%%%%%%%%%%%%%%%%%%%%%%%%%%%%%%%%%%%%%%%%%%%
\section{Energy form and its corresponding self-adjoint
  differential operator}
\label{sec:main_results}

%%%%%%%%%%%%
\subsection{Primary matching conditions: the form domain}

We now give a formal mathematical description of the Euler--Bernoulli
strain energy form. 

\begin{thm}
  \label{primaryCond}
  The energy functional~\eqref{eq:frame_energy} of the beam frame with free rigid joints is the quadratic form corresponding to the
  positive closed sesquilinear form
  \begin{align}
    \label{varEnergyForm}
    h\left[\widetilde\Psi, \Psi\right] :=
    \sum_{e \in E} \int_e \big( a_e \cc{\widetilde v_e''}v_e''
    + b_e \cc{\widetilde w_e''} w_e''
    + c_e \cc{\widetilde u_e'} u_e'
    + d_e \cc{\widetilde \eta_e'} \eta_e'  \big) dx,
  \end{align}
  densely defined on the Hilbert space
%  \begin{equation}
%    \label{eq:Hilbert_space}
%    \Ltwo := \mathop{\mathsmaller\bigoplus}_{e \in E} L^2(e)
%    \times \mathop{\mathsmaller\bigoplus}_{e \in E} L^2(e)
%    \times \mathop{\mathsmaller\bigoplus}_{e \in E} L^2(e)
%    \times \mathop{\mathsmaller\bigoplus}_{e \in E} L^2(e),
%  \end{equation}
  \begin{equation}
  \label{eq:Hilbert_space}
  \Ltwo := \mathop{\mathsmaller\bigoplus}_{e \in E} L^2(e)
  \oplus \mathop{\mathsmaller\bigoplus}_{e \in E} L^2(e)
  \oplus \mathop{\mathsmaller\bigoplus}_{e \in E} L^2(e)
  \oplus \mathop{\mathsmaller\bigoplus}_{e \in E} L^2(e),
  \end{equation}
  with the domain of $h$ consisting of the vectors
%  \begin{equation}
%    \label{eq:Sobolev_form}
%    \Psi :=
%    \begin{pmatrix}
%      v, & w, & u, & \eta
%    \end{pmatrix}^T
%    \in \mathop{\mathsmaller\bigoplus}_{e \in E} H^2(e)
%    \times \mathop{\mathsmaller\bigoplus}_{e \in E} H^2(e)
%    \times \mathop{\mathsmaller\bigoplus}_{e \in E} H^1(e)
%    \times \mathop{\mathsmaller\bigoplus}_{e \in E} H^1(e)
%    =: \dmax
%  \end{equation}
  \begin{equation}
  \label{eq:Sobolev_form}
  \Psi :=
  \begin{pmatrix}
  v, & w, & u, & \eta
  \end{pmatrix}^T
  \in \mathop{\mathsmaller\bigoplus}_{e \in E} H^2(e)
  \oplus \mathop{\mathsmaller\bigoplus}_{e \in E} H^2(e)
  \oplus \mathop{\mathsmaller\bigoplus}_{e \in E} H^1(e)
  \oplus \mathop{\mathsmaller\bigoplus}_{e \in E} H^1(e)
  =: \dmax
  \end{equation}
  that satisfy, at every vertex $\vt \in V$, the \textbf{free rigid joint}
  conditions
  \begin{enumerate}
  \item [(i)] continuity of displacement: for all $e, e' \sim \vt$,
    \begin{align}
      \label{dispRigid}
      u_e \vec i_e + w_e \vec j_e + v_e \vec k_e =
      u_{e'} \vec i_{e'} + w_{e'} \vec j_{e'} + v_{e'} \vec k_{e'},
    \end{align}
  \item [(ii)] continuity of rotation: for all $e, e' \sim \vt$, 
    \begin{align}
      \label{rotationRigid}
      \eta_e \vec i_e - v_e' \vec j_e + w_e' \vec k_e =
      \eta_{e'} \vec i_{e'} - v_{e'}' \vec j_{e'} + w_{e'}' \vec k_{e'},
    \end{align} 
  \end{enumerate}
  where all functions are evaluated at the vertex $\vt$ and all
  derivatives are taken in direction $\vec i_e$ or $\vec i_{e'}$
  correspondingly.
\end{thm}

\begin{remark}
  Let us comment on the information contained in the theorem.
  Equation~\eqref{eq:Hilbert_space} specifies the underlying inner
  product (in the terminology of Gelfand triples, it is the ``pivot
  space''), while equation~\eqref{eq:Sobolev_form} prescribes the
  correct smoothness requirements on the individual fields (twice
  differentiable $v$, $w$ and once differentiable $u$ and $\eta$).
  The form being semibounded\footnote{In our case, it is semibounded
    below by 0, i.e.\ positive.  Also note that semibounded implies
    (by definition) symmetric.} and closed tells us that it
  corresponds to a self-adjoint differential operator
  (``Hamiltonian'') which we will describe in the next theorem.

  Perhaps the most practical consequence of this theorem is the set of
  matching conditions that are expressing the definition of joint
  rigidity in terms of the degrees of freedom of the beam --- the
  displacement and angular displacement fields --- written in the
  coordinate representation in which they enter the energy ($v$, $w$,
  $u$ and $\eta$).  Given that $\eta$ is the only coordinate of
  angular nature, it may mistakenly appear that continuity of rotation
  ought to be expressed as $\eta_e = \eta_{e'}$;
  equation~\eqref{rotationRigid} is a reminder that $\eta_e$ is just
  one component of rotation --- the one around the axis of the beam
  --- while the continuity must apply to the entire element of SO(3).
\end{remark}

To get to condition~\eqref{rotationRigid} we will parameterize small
rotations in terms of the \textbf{rotation vector} we denote $\omega$.

\begin{lem}
  \label{lem:RotationVector}
  The mapping $\vec\omega \in \bR^3 \mapsto \cR \in \mathrm{SO}(3)$
  given by
  \begin{equation}
    \label{expRota}
    \cR = \exp(\Omega),
  \end{equation}
  where $\Omega = \Omega(\vec \omega)$ is a $3\times3$ skew symmetric
  matrix acting as $\vec{a} \mapsto \vec\omega \times \vec{a}$, is a
  local $C^\infty$-diffeomorphism between a neighborhood of $\vec{0}$
  in $\bR^3$ and a neighborhood of $\bI$ in $\mathrm{SO}(3)$.  For
  small deformations of the beam characterized by
  $\Psi = (v,\ w,\ u,\ \eta)$, the
  corresponding $\vec \omega(x)$ is
  \begin{equation}
    \label{eq:omega_lin}
    \vec \omega(x) = \eta(x) \vec i - v'(x) \vec j + w'(x) \vec k
    + \cO(\|\Psi\|^2_\dmax),
  \end{equation}
  where $\{\vec i, \vec j, \vec k\}$ is the local basis of the beam
  and the norm is the Sobolev norm as in \eqref{eq:Sobolev_form}.
\end{lem}

\begin{proof}
  Exponential parameterization of $\mathrm{SO}(3)$ is standard; it
  is a local chart around $\omega = \vec{0}$ which can be extended to
  cover all of $\mathrm{SO}(3)$ except rotations by $\pi$ (around some
  axis). Expanding the exponential we get
  \begin{equation}
    \label{eq:linearVecRota}
    \cR \vec a = 
    \vec a + \vec \omega \times \vec a + \cO(|\vec \omega|^2),
  \end{equation}
  which are the leading terms in the so-called Euler--Rodrigues
  formula, more common in the applied literature (a discussion of this
  formula and its connection to exponential form (\ref{expRota}) can
  be found in the Appendix).  In particular, applying
  \eqref{eq:linearVecRota} to \eqref{eq:deformedRot}, we get
  \begin{align}
    \label{eq:iRot}
    \bm{\vec i}(x) &= \vec i + \vec \omega(x) \times \vec i
                     + \cO(|\vec \omega|^2),\\
    \label{eq:jRot}
    \bm{\vec j}(x) &= \vec j + \vec \omega(x) \times \vec j
                     + \cO(|\vec \omega|^2),\\
    \label{eq:kRot}
    \bm{\vec k}(x) &= \vec k + \vec \omega(x) \times \vec k
    + \cO(|\vec \omega|^2).
  \end{align}
  Substituting \eqref{eq:jRot} into \eqref{eq:torsion_def} we get
  \begin{equation}
    \label{eq:omega1}
    \eta(x) = \vec \omega(x) \cdot \vec{i} + \cO(|\vec \omega|^2).
  \end{equation}
  On the other hand, we can calculate $\bm{\vec{i}}(x)$ on an edge $e$
  from the corresponding $\vec{g}(x)$.
  The vector
  \begin{equation}
    \label{eq:tangent_to_beam}
    \vec \xi(x)
    := \lim_{h \rightarrow 0}
    \frac{h\vec{i} + \vec g(x + h) - \vec{g}(x)}{h}
    = \vec{i} + \vec{g}'(x)
  \end{equation}
  is a tangent vector to the beam $e$ at point $x$.  The extra term
  $h\vec{i}$ appears because $\vec{g}(x)$ is the displacement vector of
  the material point $\vt_{e}^{o} + x\vec{i}$, not a position vector,
  see figure \eqref{fig:defArrows}.
  \begin{figure}
  	\centering
  	\includegraphics[width=0.4\textwidth]{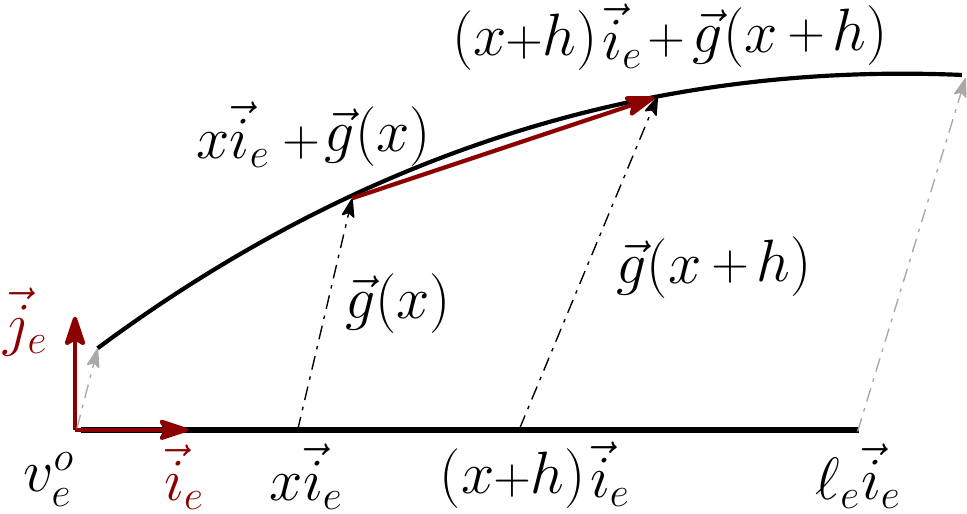}
  	\caption{Displacement of points on beam $e$}. 
  	\label{fig:defArrows}
  \end{figure}
  To normalize $\vec{\xi}(x)$, observe the following relation for general vectors
  $\vec a, \vec b \in \bR^3$,
  \begin{equation}
    \label{expansionVec}
    \frac{1}{|\vec a + \varepsilon\vec b|}
    = \frac{1}{|\vec a|}
    - \frac{\vec a \cdot \vec b}{|\vec a|^3} \varepsilon + \cO(\varepsilon^2).
  \end{equation}
  Applying this to \eqref{eq:tangent_to_beam} and using $|\vec i| = 1$
  we get
  \begin{align}
    \bm{\vec i}(x)
    &:= \frac{\vec \xi}{|\vec \xi|}
      = \vec i + \vec g'(x) - (\vec i \cdot \vec g'(x)) \vec i
      +\cO\left(\left|\vec{g}'\right|^2\right)\\
    \label{eq:i_from_g}
    &= \vec i + w'(x) \vec j + v'(x) \vec k
      +\cO\left(\left|\vec{g}'\right|^2\right).
  \end{align}
  Comparing equation~\eqref{eq:i_from_g} with \eqref{eq:iRot} and
  taking into account \eqref{eq:omega1} we arrive to
  \eqref{eq:omega_lin}.

  For future reference we also list here the expansions of other vectors in
  the deformed basis,
  \begin{equation}
    \label{eq:deformed_basisxpand}
    \bm{\vec j}(x) = -w'(x) \vec i + \vec j + \eta(x) \vec k + \cO\left(\left|\vec{g}'\right|^2\right), \quad
    \bm{\vec k}(x) = -v'(x) \vec i - \eta(x) \vec j + \vec k + \cO\left(\left|\vec{g}'\right|^2\right)
  \end{equation}
\end{proof}

\begin{proof}[Proof of Theorem~\ref{primaryCond}]
  We start by addressing the vertex matching condition.
  Equation~\eqref{dispRigid} is simply the coordinate expansion (in
  the local coordinates of each edge) of the displacement continuity
  condition in Definition~\ref{def:rigid},
  equation~\eqref{eq:rigid_displacement}.
  Equation~\eqref{rotationRigid} is a direct application of
  Lemma~\ref{lem:RotationVector} to
  condition~\eqref{eq:rigid_rotation}.

  The form $h$ defined by \eqref{varEnergyForm} is obviously positive
  and symmetric:
  \begin{equation}
    \label{eq:positive_symm_def}
    h\left[\Psi,\Psi\right] \geq 0
    \qquad\mbox{and}\qquad
    h\left[\widetilde\Psi, \Psi\right] =
    \overline{h\left[\Psi, \widetilde\Psi\right]}.
  \end{equation}
  To establish that $h$ is closed, i.e. that the domain $\Dom(h)$ --- the
  subspace of $\dmax$ consisting of functions satisfying
  \eqref{dispRigid}-\eqref{rotationRigid} --- is complete with respect
  to the norm
  \begin{equation}
    \label{eq:h-norm}
    \|\Psi\|_h := \|\Psi\|_{\Ltwo} + h\left[\Psi,\Psi\right],
  \end{equation}
  we follow the standard procedure (see, for example, \cite[Proof of
  Thm 1.4.11]{BerKuc_graphs} or \cite[Proof of Thm 4.3]{GM20}).  First
  we note that the space $\dmax$ is complete with respect to its norm
  (the sum of Sobolev norms of individual spaces).  Next, the
  operation of taking traces --- vertex values of components of $\Psi$
  and their derivatives --- is continuous in the Sobolev norm due to
  the standard inequalities (see \cite[Lemma 4.2]{GM20}),
  \begin{align}
    \label{eq:Sobolev1_trace}
    &|f(\vt_e)|^2
    \leq \frac2{|e|} \|f\|^2_{L^2(e)} + |e|\, \|f'\|^2_{L^2(e)},
    \qquad \hspace{2.8mm} \mbox{for any } f\in H^1(e),
    \\
    \label{eq:Sobolev2_trace}
    &|f'(\vt_e)|^2
    \leq \frac2{|e|} \|f'\|^2_{L^2(e)} + |e|\, \|f''\|^2_{L^2(e)},
    \qquad \mbox{for any } f\in H^2(e),
  \end{align}
  where $|e|$ is the length of the edge $e$ and $\vt_e$ is one of its
  endpoints.  Therefore $\Dom(h)$ is a closed subspace of $\dmax$ and
  thus also complete with respect to the norm of $\dmax$.  Finally, we
  observe that the $h$-norm \eqref{eq:h-norm} is equivalent to the
  Sobolev norm of $\dmax$, i.e.
  \begin{equation}
    \label{eq:form_equiv}
    c \|\Psi\|_{\dmax}
    \leq \|\Psi\|_{\Ltwo} + h\left[\Psi,\Psi\right]
    \leq C \|\Psi\|_{\dmax},
  \end{equation}
  for suitable constants $c,C>0$.  Here, for components $v$ and $w$ we
  use the inequality (see \cite[Cor 4.2.7]{Burenkov_sobolev})
  \begin{equation}
    \label{eq:itermid_deriv}
    \|f'\|_{L^2(e)}
    \leq \alpha \left( \frac1{|e|} \|f\|_{L^2(e)} + |e|
      \|f''\|_{L^2(e)}\right),
  \end{equation}
  for some $\alpha>0$ independent of $f$.
\end{proof}

%%%%%%%%%%%%
\subsection{Joint conditions for the self-adjoint operator}

The main result of this paper characterizes the Hamiltonian of the
frame as a self-adjoint differential operator on the metric graph.  
\begin{thm}
  \label{MainTheorem}
  Energy form~\eqref{varEnergyForm} on a beam frame with free rigid
  joints corresponds to the self-adjoint operator
  $H \colon \Ltwo \to \Ltwo$ acting as
  \begin{equation}
    \label{diffSystem}
    \Psi_e := \begin{pmatrix}
      v_e\\
      w_e\\
      u_e\\
      \eta_e
    \end{pmatrix}
    \mapsto
    \begin{pmatrix}
      a_e v_e''''\\
      b_e w_e''''\\
      -c_e u_e''\\
      -d_e \eta_e''
    \end{pmatrix}
  \end{equation}
  on every edge $e\in E$ of the graph.  The domain of the operator $H$
  consists of the functions
%  \begin{equation}
%    \label{eq:Sobolev_op}
%    \Psi
%    \in \mathop{\mathsmaller\bigoplus}_{e \in E} H^4(e)
%    \times \mathop{\mathsmaller\bigoplus}_{e \in E} H^4(e)
%    \times \mathop{\mathsmaller\bigoplus}_{e \in E} H^2(e)
%    \times \mathop{\mathsmaller\bigoplus}_{e \in E} H^2(e)
%    =: \Dmax,
%  \end{equation}
  \begin{equation}
  \label{eq:Sobolev_op}
  \Psi
  \in \mathop{\mathsmaller\bigoplus}_{e \in E} H^4(e)
  \oplus \mathop{\mathsmaller\bigoplus}_{e \in E} H^4(e)
  \oplus \mathop{\mathsmaller\bigoplus}_{e \in E} H^2(e)
  \oplus \mathop{\mathsmaller\bigoplus}_{e \in E} H^2(e)
  =: \Dmax,
  \end{equation}
  that satisfy, at each vertex $\vt \in V$,
  \begin{enumerate}
  \item continuity of displacement: for all $e, e' \sim \vt$, 
    \begin{equation}
      \label{primaryBcThmDisp}
      u_e \vec i_e + w_e \vec j_e + v_e \vec k_e =
      u_{e'} \vec i_{e'} + w_{e'} \vec j_{e'} + v_{e'} \vec k_{e'},
    \end{equation}
  \item continuity of rotation: for all $e, e' \sim \vt$,
    \begin{equation}
      \label{primaryBcThmRota}
      \eta_e \vec i_e - v_e' \vec j_e + w_e' \vec k_e =
      \eta_n \vec i_{e'} - v_{e'}' \vec j_{e'} + w_{e'}' \vec k_{e'},
    \end{equation}
  \item balance of forces 
    \begin{equation}
      \label{secondBcThmDisp}
      \sum_{e \sim \vt} s_\vt^e \left(c_e u_e' \vec i_e - b_e w_e''' \vec j_e
      - a_e v_e''' \vec k_e\right) = \vec0,
    \end{equation}
  \item balance of moments
    \begin{equation}
      \label{secondBcThmRota}
      \sum_{e \sim \vt} s_\vt^e \left(d_e \eta_e' \vec i_e - a_e v_e'' \vec j_e
      +  b_e w_e'' \vec k_e\right) = \vec0.
    \end{equation}
  \end{enumerate}
  Here $n_\vt$ is the degree of the vertex $\vt$, $s_\vt^e$ is the
  signed incidence indicator defined in
  section~\ref{sec:frames_description}, all functions are evaluated at
  the vertex $\vt$, and all derivatives on the edge $e$ are taken in
  the direction $\vec{i}_e$.
\end{thm}

\begin{remark}
  Knowing that the operator is self-adjoint opens up many established
  mathematical tools for understanding the time evolution of the
  corresponding wave equation,
  \begin{equation*}
    \frac{\partial^2 \Psi_e}{\partial t^2} = H \Psi_e, \qquad e \in E.
  \end{equation*}
  For example, from the Spectral Theorem we immediately get
  orthogonality of the eigenfunctions (studied, in the engineering
  literature, in \cite{ChaWil_sv00}).  With little extra work
  (establishing compactness of the resolvent on frames which contain
  finitely many beams, each of finite length), one obtains existence
  of an eigenfunction basis.
\end{remark}

\begin{remark}  
  \label{rem:conditions_meaning}
  Conditions (\ref{primaryBcThmDisp}) and (\ref{primaryBcThmRota}),
  inherited from the domain of the sesquilinear form, encapsulate the
  rigid vertex assumption. Two additional conditions
  (\ref{secondBcThmDisp}) and (\ref{secondBcThmRota}) are imposed to
  guarantee the operator is self-adjoint.  These two conditions,
  derived from operator-theoretic considerations, have important
  physical meaning.  Namely, they balance the forces and moments
  developed at the vertex as a result of relative displacement and
  angular displacement of edges adjacent to the vertex.  In more
  details:
  \begin{enumerate}
  \item [(i)] in condition (\ref{secondBcThmDisp}), $c_e u_e'$ is the
    force developed in direction $\vec i_e$ due to in-axis tension of
    the edge $e$, while $a_e v_e'''$ and $b_e w_e'''$ are shear forces
    developed inside the edge in the directions $\vec k_e$ and
    $\vec j_e$, respectively.
  	
  \item[(ii)] in condition (\ref{secondBcThmRota}), $d_e \eta_e'$ is
    the moment associated with angular displacement developed in direction $\vec i_e$ due to in-axis
    rotation of the edge $e$, while $a_e v_e''$ and $b_e w_e''$ represent
    bending moments of the edge $e$ in the directions $\vec j_e$ and
    $\vec k_e$, respectively.
  \end{enumerate}
  More details about the mechanics meaning of these quantities can be
  found in \cite{Ye_StructuralAndStress}, in particular in sections 4.5
  and 6.2 thereof.
\end{remark}

\begin{remark}
  \label{rem:edge_flip}
  It is instructive to see how the individual elements of conditions
  \eqref{primaryBcThmDisp}--\eqref{secondBcThmRota} change under the
  change of orientation of an edge $e$.  Consider, for example, the
  change of local basis (omitting subscript $e$ to avoid clutter)
  \begin{equation}
    \label{eq:local_basis_change}
    \{\vec i, \vec j, \vec k\} \mapsto \{-\vec i, -\vec j, \vec k\}.
  \end{equation}
  For convenience, we recall the definitions of the individual degrees
  of freedom,
  \begin{equation}
    \label{eq:field_comp_recall}
    v = \vec k \cdot \vec g, \quad
    w = \vec j \cdot \vec g, \quad
    u = \vec i \cdot \vec g, \quad
    \eta = \cR \vec j \cdot \vec k,
  \end{equation}
  where $\vec g$ and $\cR$ are basis-independent quantities.  In
  addition, all derivatives are taken in the direction $\vec i$,
  resulting in the sign change.  We get
  \begin{align*}
    &v \mapsto v, &&w\mapsto -w, &&u \mapsto -u, &&\eta\mapsto -\eta,\\
    &v'\mapsto -v', && w'\mapsto w', &&u'\mapsto u', &&\eta'\mapsto \eta',\\
    &v''\mapsto v'', && w''\mapsto -w'', \\
    &v'''\mapsto -v''', && w'''\mapsto w''', &&&&s_\vt^e \mapsto -s_\vt^e.
  \end{align*}
  Combining with \eqref{eq:local_basis_change}, we observe that all of the quantities
  \begin{equation*}
    u\vec i + w\vec j + v\vec k,\quad
    \eta\vec i - v'\vec j + w'\vec k,\quad
    s_\vt^e\big(cu'\vec i - bw'''\vec j - a v'''\vec k\big),
    \quad\mbox{and}\quad
    s_\vt^e\big(d\eta'\vec i - av''\vec j + bw''\vec k\big)   
  \end{equation*}
  remain invariant.
\end{remark}

\begin{remark}
  The formal proof of Theorem~\ref{MainTheorem} given below is built
  around ``knowing the right answer'' and is not particularly
  instructive.  In the present remark we \emph{derive the right answer} while
  sacrificing some of the mathematical rigor.

  We would like to represent the sesquilinear form as
  $h[\widetilde\Psi, \Psi] = \langle \widetilde\Psi, H\Psi\rangle$, so
  we integrate equation (\ref{varEnergyForm}) by parts.  This requires
  the field $\Psi$ to be sufficiently smooth, an issue we ignore for
  now.  We get
  \begin{equation}
    \label{beamDiff}
    h\left[\widetilde\Psi,\Psi\right]
    = \sum_{e \in E} \int_e \big(\cc{\widetilde v_e}v_e''''
    + b_e \cc{\widetilde w_e} w_e'''' - c_e \cc{\widetilde u_e} u_e''
    - d_e \cc{\widetilde \eta_e} \eta_e'' \big) dx +  \sum_{\vt \in V} \cB_v,
  \end{equation}
  where the boundary term $\cB_\vt$ at vertex $\vt$ has the form
  \begin{equation}
    \label{boundaryTerm}
    \cB_\vt := \sum_{e \sim \vt} s_\vt^e
    \left(a_e \cc{\widetilde v_e'} v''_e
      - a_e \cc{\widetilde v_e} v'''_e
      + b_s \cc{\widetilde w_e'} w''_e
      - b_s \cc{\widetilde w_e} w'''_e
      + c_e \cc{\widetilde u_e} u'_s
      + d_e \cc{\widetilde \eta_e} \eta'_e \right).
  \end{equation}
  Since the field $(\tilde v, \tilde w, \tilde u, \tilde \eta) $ is
  arbitrary, the form of the differential operator $H$ in
  $\langle \widetilde\Psi, H\Psi\rangle$ is clear, as long as the
  boundary terms disappear: $\cB_\vt = 0$ for all $\vt$.

  Applying the fact that $\{\vec i_e, \vec j_e, \vec k_e\}$ is an
  orthonormal basis, for each $e \sim \vt$ we can write
   \begin{equation}
  \label{proofSecondCond1}
  c_e \cc{\tilde u_e} u'_e - b_e \cc{\tilde w_e} w'''_e 
  - a_e \cc{\tilde v_e} v'''_e  \cc{\tilde v_e}
  = \underbrace{\cc{\big(\tilde u_e \vec i_e + \tilde w_e \vec j_e
  	+ \tilde v_e \, \vec k_e \big)}}_{= \overline{\vec{\tilde g}_e}(\vt)} \cdot \big(c_e u'_e \vec i_e - b_e w'''_e \vec j_e
  - a_e v'''_e  \vec k_e \big)
  \end{equation}
  and similarly 
  \begin{equation}
  \label{proofSecondCond2}
  d_e \cc{\tilde \eta_e} \eta'_e + b_e \cc{\tilde w_e'} w''_e
  + a_e \cc{\tilde v_e'} v''_e 
  = \underbrace{\cc{\big(\tilde \eta_e \vec i_e - \tilde v_e' \vec j_e
  	+ \tilde w_e' \vec k_e \big)}}_{= \overline{\vec{\tilde \omega}_e}(\vt)} \cdot \big(d_e \eta'_e \vec i_e - a_e v''_e \vec j_e
  + b_e w''_e \vec k_e \big)
  \end{equation}
  Summing over $e \sim \vt$ along with the fact that the displacement
  $\overline{\vec{\tilde g}_e}(\vt)$ and rotation vector
  $\overline{\vec{\tilde \omega}_e}(\vt)$ are independent of $e \sim \vt$, we
  arrive to the condition
  \begin{equation}
    \label{eq:B_final}
    \cB_\vt = \overline{\vec{\tilde g}_e}(\vt) \cdot
    \sum_{e \sim \vt} s_\vt^e \big(c_e u'_e \vec i_e - b_e w'''_e \vec j_e
    - a_e v'''_e  \vec k_e \big)
    + \overline{\vec{\tilde \omega}_e}(\vt) \cdot
    \sum_{e \sim \vt} s_\vt^e \big(d_e \eta'_e \vec i_e - a_e v''_e \vec j_e
    + b_e w''_e \vec k_e \big) = 0
  \end{equation}
  Since the vectors $\overline{\vec{\tilde g}_e}(\vt)$ and
  $\overline{\vec{\tilde \omega}_e}(\vt)$ are arbitrary, $\cB_\vt = 0$ only if the
  two sums in front of $\overline{\vec{\tilde g}_e}(\vt)$ and $\overline{\vec{\tilde \omega}_e}(\vt)$
  vanish. This results in conditions
  \eqref{secondBcThmDisp} and \eqref{secondBcThmRota} correspondingly.
\end{remark}

We now proceed to the formal proof of Theorem~\ref{MainTheorem}.

\begin{proof}[Proof of Theorem~\ref{MainTheorem}]
  By the Representation Theorem of sesquilinear forms (see
  \cite[section VI.2]{Kato_perturbation} or \cite[Thm.~10.7 and
  Cor.~10.8]{Schmudgen_unboundedSAO}), the form $h$ of
  Theorem~\ref{primaryCond} corresponds to a self-adjoint operator $H$
  such that $\Dom(H) \subset \Dom(h)$ and
  \begin{equation}
    \label{eq:form_op_correspondence}
    h[\widetilde\Psi, \Psi] = \langle \widetilde\Psi, H\Psi\rangle
    \quad\mbox{for all }\widetilde\Psi \in \Dom(h),\ \Psi\in\Dom(H).
  \end{equation}
  Moreover, the self-adjoint operator satisfying this condition is
  unique.

  Since $\Psi \in\Dom(H) \subset \Dmax$, the integration by parts
  performed in \eqref{beamDiff} is valid and equation
  \eqref{eq:B_final} shows that condition
  \eqref{eq:form_op_correspondence} is satisfied.  Therefore we only
  need to establish that the operator $H$ defined in the Theorem is
  indeed self-adjoint, i.e. $H^* = H$.

  The first step is to obtain a ``bound'' on the domain of $H^*$:
  prove that $\Dom(H^*) \subset \Dmax$ which guarantees enough
  smoothness of $\tilde{\Psi} \in \Dom(H^*)$ to integrate by parts
  freely.  The standard trick is as follows: consider the operator
  $H_{\mathrm{min}}$ which is given by the differential expression
  \eqref{diffSystem} with the domain consisting of functions from
  $\Dmax$ with additional conditions\footnote{These conditions
    are too restrictive for self-adjointness.} at every vertex $\vt$
  \begin{equation}
    \label{eq:all0}
    v_e = v_e' = v_e'' = v_e''' = 0,
    \quad w_e = w_e' = w_e'' = w_e''' = 0,
    \quad u_e = u_e' = 0,
    \quad \eta_e = \eta_e' = 0.
  \end{equation}
  The operator $H_{\mathrm{min}}$ decomposes into a direct sum of
  operators each acting on a single edge $e$ and a single component
  $u$, $v$, $w$ or $\eta$.  The standard calculation (see
  \cite[Example 1.4]{Schmudgen_unboundedSAO}) shows that the adjoint
  $H_{\mathrm{min}}^*$ is given by the same differential expression
  \eqref{diffSystem} with the domain equal to \emph{all} of
  $\Dmax$.  Now we obviously have
  $H_{\mathrm{min}} \subset H \subset H_{\mathrm{min}}^*$ (all
  operators act the same and the inclusion should be understood as
  inclusion of domains), therefore
  $H_{\mathrm{min}} \subset H^* \subset H_{\mathrm{min}}^*$, in
  particular $\Dom(H^*) \subset \Dmax$ and $H^*$ also acts as
  \eqref{diffSystem}.

  The second step is to observe that for \emph{any} operators $S_1$ and
  $S_2$ satisfying 
  $H_{\mathrm{min}} \subseteq S_{i} \subseteq H_{\mathrm{min}}^*$,
  $i=1,2$, we can integrate by parts to get
  \begin{equation}
    \label{eq:boundaryForm}
    \langle \widetilde\Psi, S_1 \Psi\rangle
    - \langle S_2 \widetilde\Psi, \Psi\rangle = \Omega[\widetilde\Psi,\Psi],
  \end{equation}
  where
  \begin{align}
    \label{eq:cB_def}
    \Omega[\widetilde\Psi,\Psi] = \sum_{\vt, e \sim \vt}
    &\vec{g}_e
    \cdot \big(c_e \tilde{u}'_e \vec i_e - b_e \tilde{w}'''_e \vec j_e
    - a_e \tilde{v}'''_e \vec k_e \big) s_\vt^e
    + \vec{\omega}_e
    \cdot \big(d_e \tilde{\eta}'_e \vec i_e - a_e \tilde{v}''_e \vec j_e
    + b_e \tilde{w}''_e \vec k_e \big) s_\vt^e \\
    &- \vec{\tilde g}_e
    \cdot \big(c_e u'_e \vec i_e - b_e w'''_e \vec j_e
    - a_e v'''_e  \vec k_e \big) s_\vt^e
    - \vec{\tilde \omega}_e
    \cdot \big(d_e \eta'_e \vec i_e - a_e v''_e \vec j_e
    + b_e w''_e \vec k_e \big) s_\vt^e,
  \end{align}
  with all functions evaluated at $\vt$.  This is known as the
  ``Green's identity''.
  
  The third step is to show that $H$ is symmetric, i.e.
  \begin{equation}
    \label{eq:symmetricH}
    \langle \widetilde\Psi, H\Psi\rangle
    - \langle H \widetilde\Psi, \Psi\rangle = 0
    \quad\mbox{for all }\Psi, \widetilde\Psi \in \Dom(H).
  \end{equation}
  We use \eqref{eq:boundaryForm} with $S_1=S_2=H$ and note that both
  $\Psi$ and $\tilde\Psi$ satisfy conditions
  \eqref{primaryBcThmDisp}--\eqref{secondBcThmRota}, yielding
  $\Omega[\widetilde\Psi,\Psi] = 0$ as required.  This establishes that our
  conditions are restrictive enough: $\Dom(H) \subset \Dom(H^*)$.

  Finally, we show that the conditions are not excessively
  restrictive, i.e. $\Dom(H^*) \subset \Dom(H)$.  We start with a
  $\widetilde\Psi \in \Dom(H^*)$  which means that
  \begin{equation}
    \label{eq:adjoint_def}
    \langle \widetilde\Psi, H\Psi\rangle
    - \langle H^* \widetilde\Psi, \Psi\rangle = 0
    \quad\mbox{for all }\Psi\in\Dom(H).
  \end{equation}
  Using \eqref{eq:boundaryForm} again,
  now with $S_1=H$ and $S_2=H^*$, we get
  $\Omega[\widetilde\Psi,\Psi] = 0$.  Using that $\Psi$ is arbitrary within
  $\Dom(H)$, a bit of linear algebra proves that $\widetilde\Psi$ must
  also satisfy \eqref{primaryBcThmDisp}--\eqref{secondBcThmRota}.
\end{proof}

%%%%%%%%%%%%%%%%%%%%%%%%%%%%%%%%%%%%%%%%%%%%%%%%%%%%%%%%%%%%%%%%%%%%%%%%
\section{Some extensions and special cases}

\subsection{Dummy vertices}

As an important check on the consistency of the conditions we
derived we show that any point in the interior of a beam can be considered a
free rigid joint of two co-linear beams (and vice versa).

\begin{lem}
  \label{coLinearEdges}
  Let $\vt$ be a free rigid joint of degree 2 and the beams incident
  to it have identical material properties and be co-linear, i.e.\
  $\{\vec i_1, \vec j_1, \vec k_1\} = \{\vec i_2, \vec j_2, \vec
  k_2\}$ up to a sign change of an even number of vectors. Then
  $\Psi_{e_1}$ and $\Psi_{e_2}$ agree at $\vt$ including derivatives
  (of order 1 for $u$ and $\eta$ and order 3 for $v$ and $w$) and
  therefore the union of $e_1$ and $e_2$ may be represented as a
  single edge $e$ with $\Psi_e \in \Dmax$.
\end{lem}

\begin{proof}
  Assume, for now, that
  $\{\vec i_1, \vec j_1, \vec k_1\} = \{\vec i_2, \vec j_2, \vec
  k_2\}$.  This means that the end $e_1$ meets the origin of $e_2$ (or
  vice versa) and therefore $s_\vt^{e_1} = -s_\vt^{e_2}$.  Evaluating
  components of conditions
  \eqref{primaryBcThmDisp}--\eqref{secondBcThmRota} with respect to
  the common basis $\{\vec i_1, \vec j_1, \vec k_1\}$, we get for the
  values of the corresponding functions at $\vt$
  \begin{equation}
    \label{eq:deg2agreement}
    u_{e_1} = u_{e_2},\ u_{e_1}' - u_{e_2}' = 0,
    \qquad
    v_{e_1} = v_{e_2},\ v_{e_1}' = v_{e_2}',\ 
    v_{e_1}'' - v_{e_2}''=0,\ v_{e_1}''' - v_{e_2}''' = 0,
  \end{equation}
  and so on.  To understand other cases of bases arrangements, for
  example, the case
  $\{\vec i_1, \vec j_1, \vec k_1\} = \{-\vec i_2, -\vec j_2, \vec
  k_2\}$, we refer to Remark~\ref{rem:edge_flip}.
\end{proof}

%%%%%%%%%%
\subsection{Planar frame}

A particularly well-studied case in the literature is the planar
frame, which in its un-deformed configuration is embedded in a two
dimensional plane.  Without loss of generality assume that this plane
has normal vector $\vec E_3$.  We also assume that for all edges $e$,
$\vec k_e$ is chosen to be the same, $\vec k_e = \vec k = \vec E_3$
(in particular this requires that the principal axes of inertia of
every beam can be aligned accordingly).

In this situation the Hamiltonian of the frame decouples into two
operators, one linking out-of-plane displacement with angular
displacement (cf.\ Chapter 1 in \cite{PDEOnMultiStructures01}) and the
other linking in-plane displacement with axial displacement.

\begin{cor}
  \label{decouplingPlane}
  Free planar network of beams with free rigid joints is described by
  Hamiltonian $H = H_1 \oplus H_2$ where $H_1$ is a differential operator acting as
  \begin{equation}
    \label{eq:operatorH_planar_inplane}
    \Psi_{e,1} :=
    \begin{pmatrix}
      v_e\\
      \eta_e
    \end{pmatrix}
    \mapsto
    \begin{pmatrix}
      a_e v_e''''\\
      -d_e \eta_e''
    \end{pmatrix}
  \end{equation}
  on functions
  $\Psi_{e,1} \in \bigoplus_{e\in E} H^4(e) \oplus \bigoplus_{e\in E}
  H^2(e)$ satisfying at each vertex $\vt$ the primary (form domain)
  conditions
  \begin{subequations}
    \label{vertexCondPlanarPrimaryH1}
    \begin{align}
      v_e
      &= v_{e'}, \qquad\qquad \forall e,e'\sim \vt,
        \label{primaryVertexCond1H1} \\
      \eta_1 \vec i_1 - v_1' \vec j_1
      &= \eta_{e'} \vec i_{e'} - v_{e'}' \vec j_{e'}, \qquad\qquad \forall e,e'\sim \vt,
        \label{primaryVertexCond2H1}
    \end{align}
  \end{subequations}
  and the secondary conditions
  \begin{subequations}
    \label{vertexCondPlanarSecondaryH1}
    \begin{gather}
    \label{secondVertexCond2H1}
    \sum_{e \sim \vt} s_\vt^e a_e v'''_e
    = 0,\\
    \label{secondVertexCond1H1}
    \sum_{e \sim \vt} s_\vt^e \big(d_e \eta'_e \vec i_e - a_e v''_e \vec j_e\big)
    = \vec 0.
    \end{gather}
  \end{subequations}

  The operator $H_2$ acts as 
  \begin{equation}
    \Psi_{e,2} :=
    \begin{pmatrix}
      w_e\\
      u_e
    \end{pmatrix}
    \mapsto
    \begin{pmatrix}
      b_e w_e''''\\
      -c_e u_e''
    \end{pmatrix}
  \end{equation}
  on functions
  $\Psi_{e,2} \in \bigoplus_{e\in E} H^4(e) \oplus \bigoplus_{e\in E}
  H^2(e)$ satisfying at each vertex $\vt$ the primary (form domain)
  conditions
  \begin{subequations}
    \label{vertexCondPlanarPrimaryH2}
    \begin{align}
      u_e \vec i_e + w_e \vec j_e
      &= u_{e'} \vec i_{e'} + w_{e'} \vec j_{e'}, \qquad\qquad \forall e,e'\sim \vt,
        \label{primaryVertex2H2} \\
      w_e' &= w_{e'}',  \qquad\qquad \forall e,e'\sim \vt,
             \label{primaryVertex1H2}
    \end{align}
  \end{subequations}
  and secondary conditions
  \begin{subequations}
    \label{vertexCondPlanarSecondaryH2}
    \begin{gather}
      \sum_{e \sim \vt} s_\vt^e \big(c_e u'_e \vec i_e - b_e w'''_e \vec j_e\big)
      = \vec 0,
      \label{secondVertex1H2} \\%
      \sum_{e \sim \vt} s_\vt^e b_e w''_e
      = \vec 0.
      \label{secondVertex2H2}
    \end{gather}
  \end{subequations}
\end{cor}

\begin{proof}
  The differential expression for the operator $H$ is already in the
  ``block-diagonal'' form, see \eqref{diffSystem}, so it remains to
  show that the vertex conditions decompose as described.  But that
  follows directly from projecting conditions
  \eqref{primaryBcThmDisp}--\eqref{secondBcThmRota} onto the common
  normal $\vec k$ and onto its orthogonal complement.
\end{proof}

One of the vertex conditions used in mathematical literature dealing
with planar beam networks is that the beams remain locally planar when
deformed, see e.g.\ \cite{BL04} and \cite{KiiKurUsm_pla15}.  This
property is now a part of vertex
condition~\eqref{primaryVertexCond2H1} as shown in the next Lemma.  In
fact, condition~\eqref{primaryVertexCond2H1} contains more information
due to the fact that we now have an extra degree of freedom (angular
displacement $\eta(x)$) in comparison to \cite{BL04,KiiKurUsm_pla15}.
We also note that condition~\eqref{primaryVertexCond2H1} does not
require special consideration for certain values of angles between
beams, cf.\ \cite[Theorem 1]{KiiKurUsm_pla15}.
\begin{lem}
  \label{equivConditions}
  Condition~\eqref{primaryVertexCond2H1} holds if and only if for any $e \sim \vt$
  \begin{align}
    \label{rigidPlane_1}
    (\vec j_2 \cdot \vec i_e) v_1'
    +  (\vec j_e \cdot \vec i_1) v_2'
    +  (\vec j_1 \cdot \vec i_2) v_e'
    &= 0,\\
    \label{rigidPlane_2}
    (\vec j_2 \cdot \vec j_e) v_1'
    - (\vec j_e \cdot \vec j_1) v_2'
    + (\vec j_1 \cdot \vec i_2) \eta_e
    &= 0.
  \end{align}
  Equation (\ref{rigidPlane_1}) is equivalent to the requirement that
  all vectors $\big\{\bm{\vec
  i}_e(\vt)\big\}_{e \sim \vt}$ lie in the same plane.
\end{lem}
The proof of the lemma is technical and is deferred to Appendix~\ref{sec:tangent_plane}.
\begin{remark}
  \label{degTwoTrivial}
  In the special case of degree $n=2$, condition (\ref{rigidPlane_1})
  becomes vacuous, reducing to $v_1' = v_1'$ and $v_2' = v_2'$.
\end{remark}

%%%%%%%%%%%
\subsection{Self-adjoint extension including vertex energy}
\label{sec:VertexEnergy}

Here we describe a range of vertex conditions achieved by introducing
vertex energy in~\eqref{varEnergyForm}.  Consider the energy form
\begin{equation}
  \label{modifiedEnergy}
  \cU^{(\Gamma)} = \sum_{e \in E} \cU^{(e)} +
  \sum_{\vt \in V} \cU^{(\vt)}
\end{equation}
with vertex energy defined as
\begin{align}
  \label{vertexEnergy}
  \cU^{(\vt)} := \frac{1}{2} \big(|\vec g_\vt|^2 \tan (\alpha_\vt)
  + |\vec \omega_\vt|^2 \tan (\beta_\vt) \big)
\end{align}
In (\ref{vertexEnergy}) the parameters
$\alpha_\vt, \beta_\vt \in (-\pi/2, \pi/2)$ and
$\vec g_\vt := \vec g_e(\vt)$ and
$\vec \omega_\vt := \vec \omega_e(\vt)$ for any $e$ adjacent to $\vt$.
We then arrive to the following modification of Theorem
\ref{MainTheorem}.

\begin{thm}
  \label{selfAdExtThm}
  Free network of beams with energy functional (\ref{modifiedEnergy})
  along with rigid vertex condition is described by the Hamiltonian $H$
  acting as (\ref{diffSystem}) that satisfies at each vertex $\vt \in V$
  of degree $n_\vt$, 
  \begin{enumerate}
  \item[(i)] continuity of displacement: for all $e,e'\sim \vt$,
    \begin{align}
      \label{primaryBcThmExtDisp}
      u_e \vec i_e + w_e \vec j_e + v_e \vec k_e
      = u_{e'} \vec i_{e'} + w_{e'} \vec j_{e'} + v_{e'} \vec k_{e'} =: \vec g_\vt
    \end{align}
  \item[(ii)] continuity of rotation: for all $e,e'\sim \vt$, 
    \begin{align}
      \label{primaryBcThmExtRota}
      \eta_e \vec i_e - v_e' \vec j_e + w_e' \vec k_e
      = \eta_{e'} \vec i_{e'} - v_{e'}' \vec j_{e'} + w_{e'}' \vec k_{e'} =: \vec \omega_\vt
    \end{align}
  \item[(iii)] balance of forces 
    \begin{align}
      \label{secondBcThmExtDisp}
      \vec g_\vt \sin(\alpha_\vt)
      + \cos(\alpha_\vt) \sum_{e\sim\vt } s_\vt^e\big(c_e u_e' \vec i_e
      - b_e w_e''' \vec j_e - a_e v_e''' \vec k_e \big) = 0,
    \end{align}
  \item[(iv)] balance of moments  
    \begin{align}
      \label{secondBcThmExtRota}
      \vec \omega_\vt \sin (\beta_\vt)
      + \cos(\beta_\vt) \sum_{e\sim\vt} s_\vt^e \big(d_e \eta_e' \vec i_e
      -a_e v_e'' \vec j_e + b_e w_e'' \vec k_e \big) = 0.
    \end{align}
  \end{enumerate}
\end{thm}

We omit the proof of this Theorem as it is analogous to the proof of
Theorem \ref{MainTheorem}. 

Setting $\alpha_\vt = \beta_\vt = 0$ in the above Theorem we recover
the free rigid joint conditions of Theorem \ref{MainTheorem}.  While
the value $\pi/2$ is excluded from the allowed range of $\alpha_\vt$
and $\beta_\vt$ in the energy form \eqref{vertexEnergy}, the
corresponding Hamiltonian remains self-adjoint, with conditions
becoming
\begin{equation}
  \label{eq:fixed_displacement}
  u_e \vec i_e + w_e \vec j_e + v_e \vec k_e = 0
\end{equation}
or
\begin{equation}
  \label{eq:fixed_rotation}
  \eta_e \vec i_e - v_e' \vec j_e + w_e' \vec k_e = 0
\end{equation}
with $\alpha_\vt=\pi/2$ or $\beta_\vt=\pi/2$ correspondingly.  As
usual with Dirichlet-type conditions, they have to be introduced into
the domain of the form directly, hence the divergence of
\eqref{vertexEnergy}.

The choices $\{\alpha_\vt, \beta_\vt\} \in \{0, \pi/2\}$ cover several
physically relevant conditions.  For a vertex of degree $n=1$
we summarize these conditions in Table~\ref{bCondTable}, listing the
names commonly used in applied literature.  Note that many more
variations of the conditions can be formulated (and are of great
practical relevance), where the displacement or rotation are not
completely free or completely zero, but are restricted to certain
subspaces.

\begin{table}
  \renewcommand{\arraystretch}{1.5}
  \begin{tabular}{|c||c|c|c|c|}
    \hline
    \multirow{2}{1.2cm}{\centering Type}
    & \multirow{2}{0.75cm}{\centering $\alpha_\vt$}
    & \multirow{2}{0.75cm}{\centering $\beta_\vt$}
    & \multicolumn{2}{c|}{\centering Boundary Types} \\
    \cline{4-5}
    &&& Primary & Secondary \\ 
    \hline
    Free & $0$& $0$ & \multirow{2}{1.2cm}{---}
    & $u'_\vt= w'''_\vt = v'''_\vt = 0$ \\
    &&&& $\eta'_\vt = v''_\vt = w''_\vt = 0$\\
    \hline
    Fixed (Clamped) & $\pi/2$ & $\pi/2$ &$u_\vt = w_\vt = v_\vt = 0$
    & \multirow{2}{1.2cm}{---} \\
    &&& $\eta_\vt = v'_\vt = w'_\vt = 0$ &\\
    \hline
    Pinned & $\pi/2$ & $0$ & $u_\vt = w_\vt = v_\vt = 0$
    & $\eta'_\vt = v''_\vt = w''_\vt = 0$\\
    \hline
    Guided & $0$ & $\pi/2$ & $\eta_\vt = v'_\vt = w'_\vt = 0$
    & $u'_\vt= w'''_\vt = v'''_\vt = 0$\\
    \hline
  \end{tabular}
  \caption{Some boundary conditions for a vertex of degree one.}
  \label{bCondTable}
\end{table} 

\begin{remark}
  \label{rem:all_conditions}
  It is also possible to classify \emph{all} conditions leading to a
  self-adjoint operator \eqref{diffSystem}, using, for example,
  symplectic linear algebra \cite[Theorem 3.1]{BerLatSuk_am19} or
  other methods \cite[Theorem 1.4.1]{BerKuc_graphs}.
\end{remark}

%%%%%%%%%%%%%%%%%%%%%%%%%%%%%%%%%%%%%%%%%%%%%%%%%%%%%%%%%%%%%%%%%%%%%%%%%%
\section{Spectral analysis by example}

%%%%%%%%%%%
\subsection{General strategy}
Calculating eigenvalues of the operator $H$ described in
Theorem~\ref{MainTheorem} is the main step towards understanding of
vibration of beam frames.  The most direct approach is to solve the
eigenvalue equation $H\Psi = \lambda \Psi$ component-wise on every
edge before applying conditions at vertices.  For example, the equation on
the out-of-plane displacement $v_e$ on an edge $e$ reads
\begin{equation}
  \label{eq:v_equation}
  a_e \frac{d^4v_e(x)}{dx^4} = \lambda v_e(x),
\end{equation}
with the general solution (assuming $\lambda\neq0$),
\begin{equation}
  \label{eq:v_solution_generale}
  v_e(x) = C_{e,1}^{v} \cosh(\mu_e x) + C_{e,2}^{v} \sinh(\mu_e x)
  + C_{e,3}^{v} \cos(\mu_e x) + C_{e,4}^{v} \sin(\mu_e x),
  \quad \mu_e := (\lambda/a_e)^{1/4}.
\end{equation}
The eigenvalue equation for the in-axis angular displacement is
\begin{equation}
  \label{eq:eta_equation}
  -d_e\frac{d^2\eta_e(x)}{dx^2} = \lambda \eta_e(x),
\end{equation}
with the general solution
\begin{equation}
  \label{eq:eta_solution_generale}
  \eta_e(x) = C_{e,1}^{\eta} \cos(\beta_e x) + C_{e,2}^{\eta} \sin(\beta_e x),
  \quad \beta_e := (\lambda/d_e)^{1/2},
\end{equation}
and so on.  Vertex conditions
\eqref{primaryBcThmDisp}--\eqref{secondBcThmRota} imposed at the ends
of the edges couple together different solutions and lead to a
homogeneous linear system on the constants
$\left\{C^{*}_{e,j}\right\}$.  The coefficients of this linear system
are trigonometric functions of $\lambda$ and edge lengths.
Eigenvalues $\lambda$ are characterized by the existence of a
non-trivial choice of the constants $\left\{C^{*}_{e,j}\right\}$,
which can be detected by equating the system's determinant to zero.
This condition is known as the dynamic stiffness method in engineering
literature \cite{Banerjee19} and is analogous to ``characteristic'' or ``secular'' equation used in
the study of quantum graphs \cite{Bel_laa85,KotSmi_ap99}.  

However, problems can quickly become computationally overwhelming
since there are 12 constants $C^{*}_{e,j}$ associated with every beam.
In section~\ref{sec:example_antenna} we simplify an example using
symmetry, which is frequently present in real-life applications.  This
technique, based on representation theory of finite groups, allows one
to decompose the original problem into a sum of operators each
corresponding to a particular class of vibrational modes.

We note that a particular example of such decomposition is the
decoupling of Corollary~\ref{decouplingPlane} for a planar frame.
Indeed, all planar graphs enjoy the symmetry of reflection across the $X-Y$
(or $E_1-E_2$) plane.  The operator $H_1$ of
Corollary~\ref{decouplingPlane} encapsulates the part of the operator
$H$ which is anti-symmetric with respect to this reflection.

%%%%%%%%%%%
\subsection{Example: a simple planar graph}
Consider the \emph{planar} beam frame depicted in Figure
\ref{fig:inPlane} consisting of three beams $e_1, e_2, e_3$ meeting at
the free rigid joint $\vt_c$.  The endpoints $\vt_1$, $\vt_2$ and
$\vt_3$ are assumed to be fixed, see Table~\ref{bCondTable}.  The
beams are oriented from the fixed ends to the joint; the local basis
of each beam is shown in the Figure.

\begin{figure}
  \centering
  \includegraphics[width=0.4\textwidth]{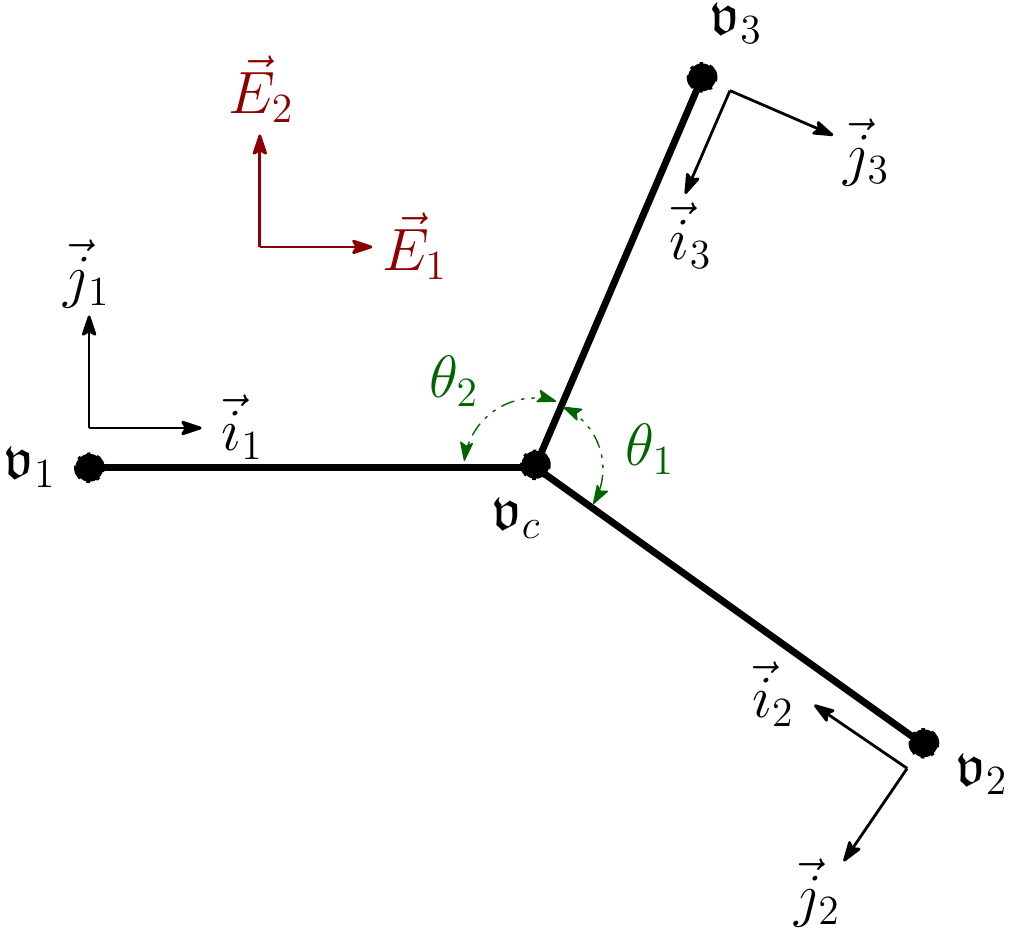}
  \caption{Geometry of planar star graph in its equilibrium
    sate.} 
  \label{fig:inPlane}
\end{figure}

We focus our attention on the out-of-plane displacement of the beam
frame, which, according to Corollary~\ref{decouplingPlane}, is coupled
to the angular displacement.  The eigenvalue problem for the operator
$H_1$ satisfies equations \eqref{eq:v_equation} and
\eqref{eq:eta_equation} on each edge.  Imposing the conditions at the
boundary vertices $\vt_1, \vt_2, \vt_3$, the general solutions reduce
to
\begin{equation}
  \label{eigenProblemPlanar}
  v_i(x) = A_i \big(\sinh(\mu_i x) - \sin(\mu_i x)\big)
  + B_i \big(\cosh(\mu_i x) - \cos(\mu_i x)\big), \qquad
  \eta_i(x) = C_i \sin(\beta_i x),
\end{equation}
with $\mu_i := (\lambda/a_i)^{1/4}$ and
$\beta_i = (\lambda/d_i)^{1/2}$.  Applying vertex conditions
(\ref{vertexCondPlanarPrimaryH1}) and
(\ref{vertexCondPlanarSecondaryH1}) to (\ref{eigenProblemPlanar}), the
eigenvalue problem reduces to finding the coefficient vector
$(A_1~B_1~A_2~B_2~A_3~B_3~C_1~C_2~C_3)^T$ in the kernel of the
$9\times9$ matrix $M = M(\lambda)$ given by
\begin{equation*}
\Scale[0.8]{
	M =
\begin{pmatrix}
S_{\mu_1 \ell_1}^{-} & C_{\mu_1 \ell_1}^{-} & 0 & -S_{\mu_2 \ell_2}^{-} & -C_{\mu_2 \ell_2}^{-} & 0 & 0 & 0 & 0 \\
S_{\mu_1 \ell_1}^{-} & C_{\mu_1 \ell_1}^{-} & 0 & 0 & 0 & 0 & -S_{\mu_3 \ell_3}^{-} & -C_{\mu_3 \ell_3}^{-} & 0 \\
\mu_1C_{\mu_1 \ell_1}^{+} & \mu_1S_{\mu_1 \ell_1}^{-} & 0 & \mu_2C_{\mu_2 \ell_2}^{+} & \mu_2S_{\mu_2 \ell_2}^{-} & 0 & \mu_3C_{\mu_3 \ell_3}^{+} & \mu_3S_{\mu_3 \ell_3}^{-} & 0 \\
0 & 0 & S_{\beta_1 \ell_1} & S_{\alpha_1}\mu_2 C_{\mu_2 \ell_2}^{-} &  S_{\alpha_1}\mu_2 S_{\mu_2 \ell_2}^{+} & -C_{\alpha_1}S_{\beta_2 \ell_2} & 0 & 0 & 0 \\
0 & 0 & S_{\beta_1\ell_1} & 0 & 0 & 0 & -S_{\alpha_2}\mu_3 C_{\mu_3 \ell_3}^{-} & -S_{\alpha_2}\mu_3 S_{\mu_3 \ell_3}^{+} & -C_{\alpha_2}S_{\beta_3\ell_3} \\
-\mu_1C_{\mu_1 \ell_1}^{-} & -\mu_1S_{\mu_1 \ell_1}^{+} & 0 & C_{\alpha_1}\mu_2C_{\mu_2 \ell_2}^{-} & C_{\alpha_1}\mu_2S_{\mu_2 \ell_2}^{+} & S_{\alpha_1}S_{\beta_2 \ell_2} & 0 & 0 & 0 \\
-\mu_1C_{\mu_1 \ell_1}^{-} & -\mu_1S_{\mu_1 \ell_1}^{+} & 0 & 0 & 0 & 0 & C_{\alpha_2}\mu_3C_{\mu_3 \ell_3}^{-} & C_{\alpha_2}\mu_3S_{\mu_3 \ell_3}^{+} & -S_{\alpha_2}S_{\beta_3 \ell_3} \\
\mu_1^2S_{\mu_1 \ell_1}^{+} & \mu_1^2C_{\mu_1 \ell_1}^{+} & 0 & C_{\alpha_1}\mu_2^2S_{\mu_2 \ell_2}^{+} & C_{\alpha_1}\mu_2^2C_{\mu_2 \ell_2}^{+} & S_{\alpha_1}\beta_2C_{\beta_2\ell_2} & C_{\alpha_2}\mu_3^2S_{\mu_3 \ell_3}^{+} & C_{\alpha_2}\mu_3^2C_{\mu_3 \ell_3}^{+} & -S_{\alpha_2}\beta_3S_{\beta_3\ell_3} \\
0 & 0 & \beta_1C_{\beta_1 \ell_1} & S_{\alpha_1}\mu_2^2S_{\mu_2 \ell_2}^{+} & -S_{\alpha_1}\mu_2^2C_{\mu_2 \ell_2}^{+} & C_{\alpha_1}\beta_2C_{\beta_2\ell_2} & S_{\alpha_2}\mu_3^2S_{\mu_3 \ell_3}^{+} & S_{\alpha_2}\mu_3^2C_{\mu_3 \ell_3}^{+} & C_{\alpha_2}\beta_3C_{\beta_3 \ell_3}
\end{pmatrix}
}  
\end{equation*}
with the corresponding eigenvalue $\lambda$ being a solution of
$\det(M_\lambda) = 0$.  We used the following abbreviations to make the matrix more compact
\begin{equation}
\label{eq:notationMatrix}
S_{\gamma} := \sin(\gamma), \hspace{5mm}  
C_{\gamma} := \cos(\gamma), \hspace{5mm}
S_{\gamma}^{\pm} := \sinh(\gamma) \pm \sin(\gamma), \hspace{5mm} 	C_{\gamma}^{\pm} := \cosh(\gamma) \pm \cos(\gamma),
\end{equation}

\begin{figure}
  \center
  \subfigure{
    \includegraphics[width=0.45\textwidth]{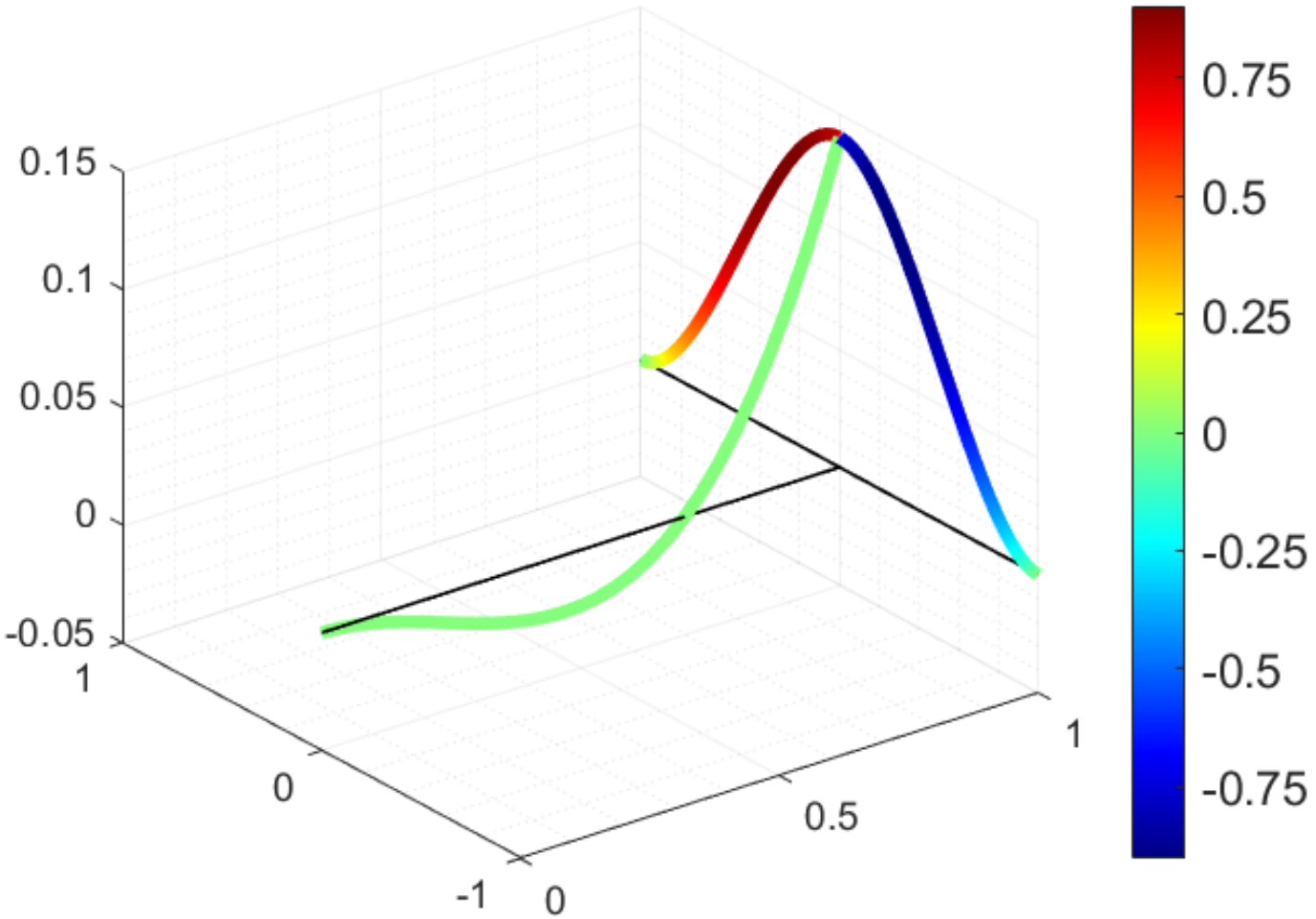}
  }
  \hspace{1mm}
  \subfigure{
    \includegraphics[width=0.45\textwidth]{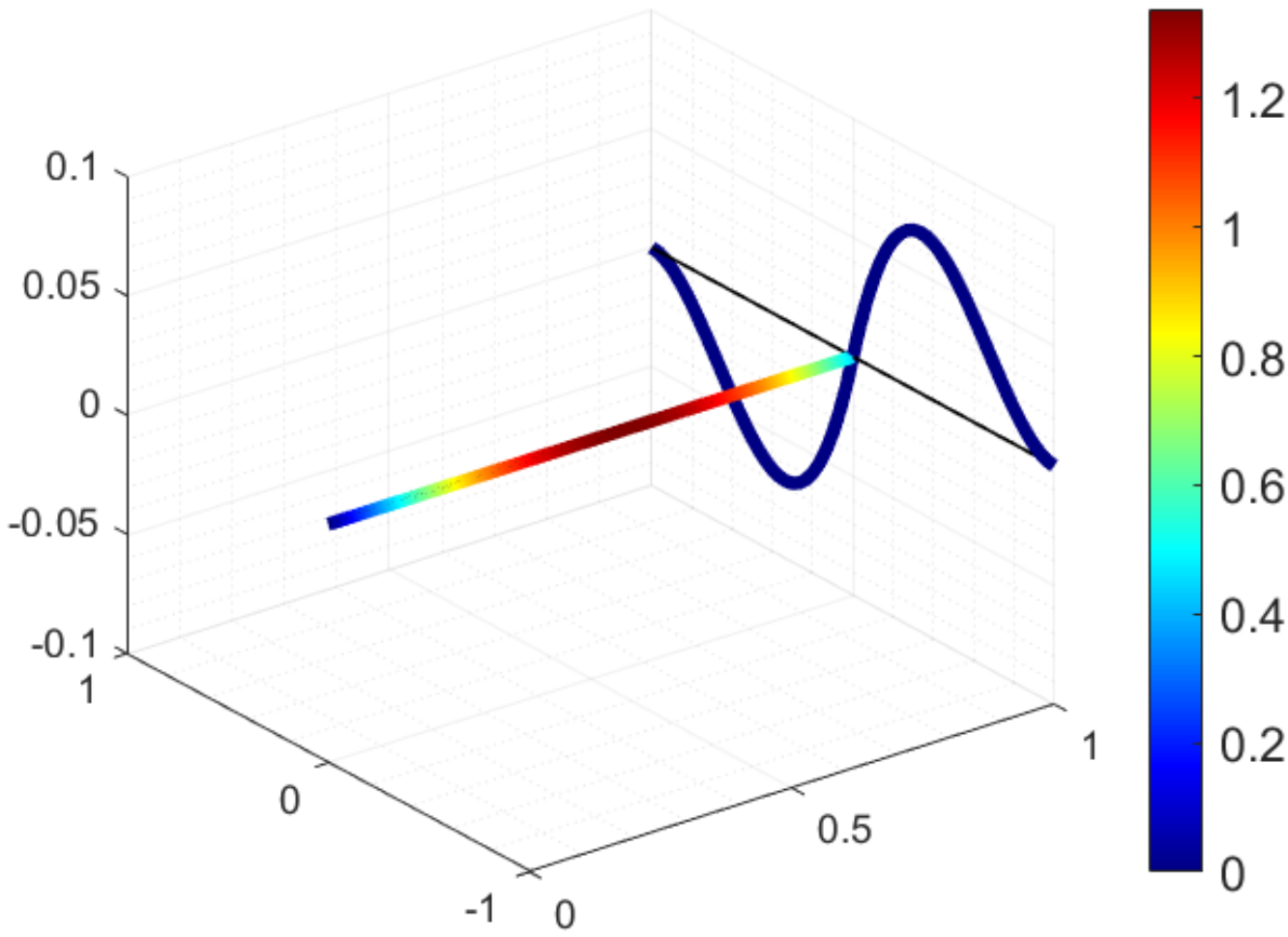}
  }
  \caption{Eigenfuctions corresponding to first and
    second eigenvalues for parameters $a_1=a_2=a_3=1$ and
    $d_1=d_2=d_3=1$. Color bar shows
    value of in-axis angular displacement of edges.}
  \label{dispEigenfunctionPlanar1}
\end{figure}

Results of numerical calculation are shown in
Figures~\ref{dispEigenfunctionPlanar1} and
\ref{dispEigenfunctionPlanar2} for $\theta_1 = \pi$ and $\theta_2= \pi/2$ (see Figure \ref{fig:inPlane}). Figure \ref{dispEigenfunctionPlanar1}
shows displacement of the first two eigenfunctions where all beams'
properties are selected identically. Unit value is selected for all
the parameters and the length of beams, i.e.\ for all beams we set $a_i=1$ and $d_i = 1$ (these parameters appear in definition of $\mu_i$'s and $\beta_i$'s above). The color represents
the in-axis angular displacement of the beams.

\begin{figure}
	\center
	\subfigure{
		\includegraphics[width=0.45\textwidth]{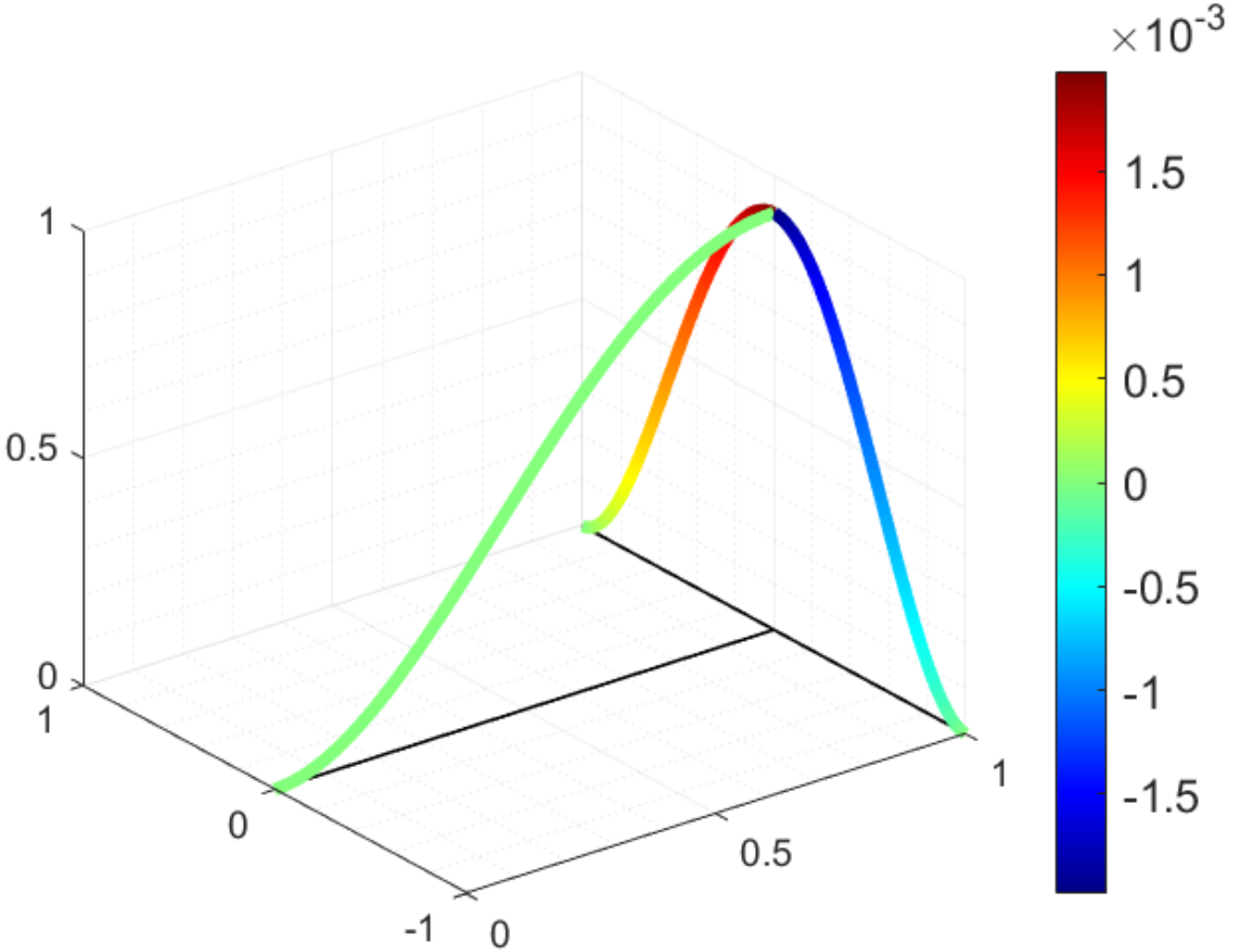}
	}
	\hspace{1mm}
	\subfigure{
		\includegraphics[width=0.45\textwidth]{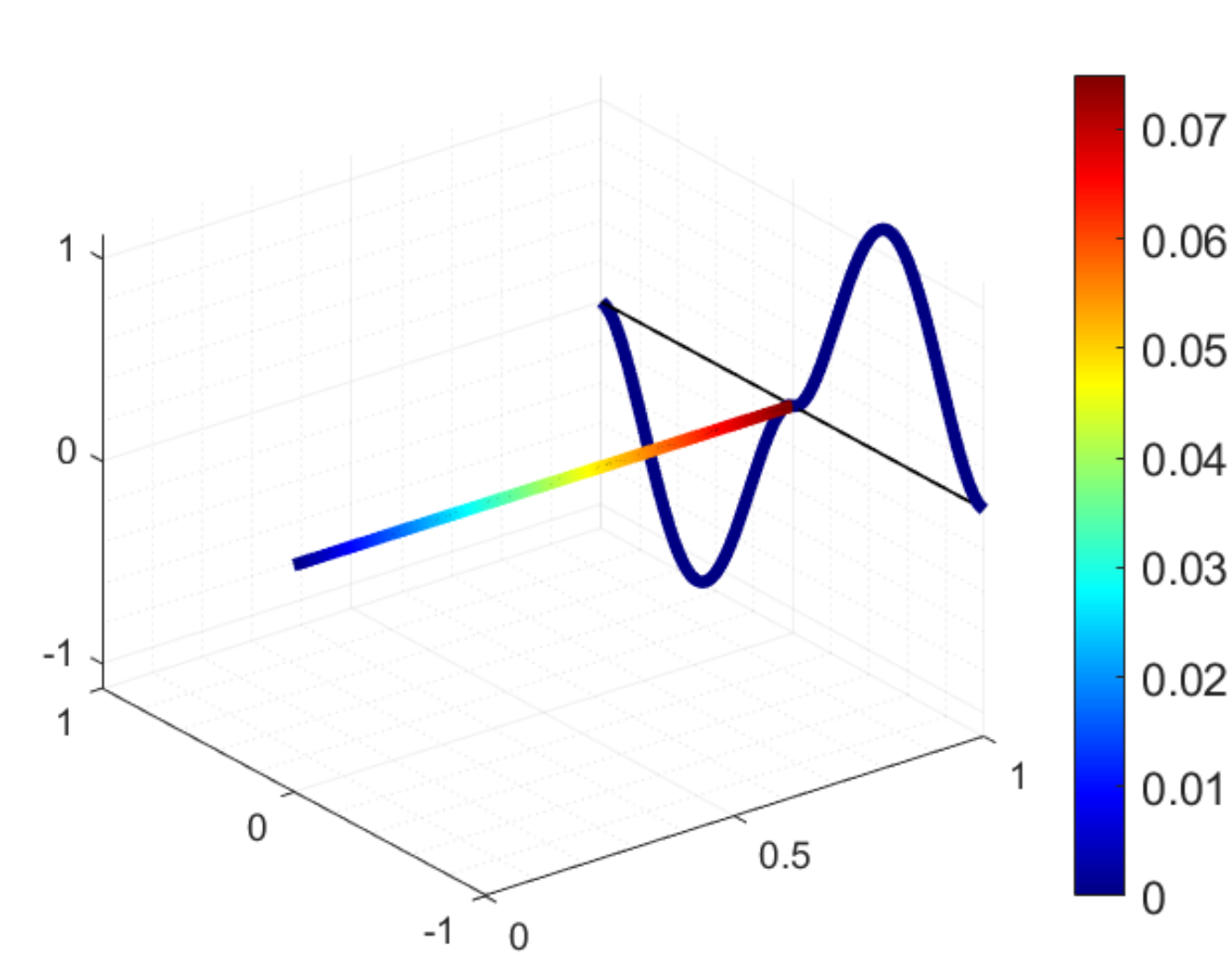}
	}
	\caption{Eigenfuctions corresponding to first and second
          eigenvalues for parameters $a_1=a_2=a_3=1$ and
          $d_1=d_2=d_3=10^3$. Color bar shows value of
          in-axis angular displacement of edges.}
	\label{dispEigenfunctionPlanar2}
\end{figure}

Increasing the edge's resistance to angular displacement strongly affects the form
of the eigenfunction. In the limiting
case where $d_i \rightarrow \infty$, the tangent plane at vertex $c$ does not change its angle (remains horizontal).  Figure \ref{dispEigenfunctionPlanar2} shows the
plots of displacement eigenfunction along with the value of  angular displacement for $a_i=1$ and $d_i = 10^3$ for $i=1,2,3$.

%%%%%%%%%%%
\subsection{Example of a three dimensional graph: antenna tower}
\label{sec:example_antenna}

Next we apply conditions of Theorem \ref{MainTheorem} to study
eigenmodes of a three dimensional star graph.  Our particular choice
is a graph with a high degree of symmetry, which is frequent in
applications.  This will allow us to demonstrate the use of group
representation theory to classify eigenmodes and simplify their
analysis.  Let $\Gamma_{AT}$ be the graph formed by three leg beams
$e_1, e_2, e_3$ and a vertical ``antenna'' beam $e_0$, all joining at
the internal vertex $\vt_c$, see Figure~\ref{threeDimGraphExample}.
The structure and its material parameters are assumed to be symmetric
with respect to rotation by $2\pi/3$ around the vertical axis and with
respect to reflection swapping a pair of the leg beams.  This implies
that $a_0=b_0$ and that the leg beams are identical with principal
axes of inertia that may be chosen to have $\vec{j}_n \perp \vec{E}_3$
for $n=1,2,3$, see Figure~\ref{threeDimGraphExample} and
equation~\eqref{eq:4star_bases} below.  The leg ends
$\vt_1,\vt_2,\vt_3$ are clamped (vanishing displacement and angular
displacement), vertex $\vt_c$ is a free rigid joint and the end
$\vt_0$ of the vertical beam is free.

Our main result of the section is a decomposition of the operator $H$
into a direct sum of four self-adjoint operators.  We will later
analyse each part separately.

\begin{thm}
  \label{thm:reducing_antenna_tower}
  The Hamiltonian operator $H$ of the beam frame $\Gamma_{AT}$ is
  reduced by the decomposition
  \begin{equation}
    \label{eq:reducing_AT}
    L^2(\Gamma) = \cH_{\mathrm{trv}} \oplus \cH_{\mathrm{alt}}
    \oplus \cH_\omega \oplus \cH_{\cc{\omega}},
  \end{equation}
  where
  \begin{align}
    \label{eq:Hid}
    \cH_{\mathrm{trv}}
    &:= \left\{\Psi \in L^2(\Gamma) \colon 
      v_0=w_0=\eta_0=0,\ w_s=\eta_s=0,\
      v_1=v_2=v_3,\ u_1=u_2=u_3 \right\}, \\
    \label{eq:Halt}
    \cH_{\mathrm{alt}}
    &:= \left\{\Psi \in L^2(\Gamma) \colon v_0=w_0=u_0=0,\ u_s=v_s=0,\ 
      w_1=w_2=w_3,\ \eta_1=\eta_2=\eta_3 \right\}, \\
    \label{eq:Homega}
    \cH_\omega
    &:= \left\{\Psi \in L^2(\Gamma) \colon u_0=\eta_0=0,\ w_0=\mathrm{i} v_0,\
      \Psi_3 = \omega \Psi_2
      = \omega^2 \Psi_1 \right\} = \cc{\cH_{\cc{\omega}}},
  \end{align}
  where $s\in\{1,2,3\}$ labels the legs, $\Psi_s:=(v_s,w_s,u_s,\eta_s)^T$,
  and $\omega = e^{2\pi\mathrm{i}/3}$.
\end{thm}

\begin{remark}
  A decomposition $\cH = \bigoplus_{\rho} \cH_\rho$ is
  \textbf{reducing} for an operator $H$ if $H$ is invariant on each of
  the subspaces and the operator domain $\Dom(H)$ is ``aligned''
  with respect to the decomposition, namely
  \begin{equation*}
    \Dom(H) = \bigoplus_{\rho} \big(\cH_\rho\cap\Dom(H)\big).
  \end{equation*}
  More precisely, the operator $H$ is required to commute with the
  orthogonal projector $\Prho$ onto $\cH_\rho$.
  This means that we can restrict $H$ to each subspace in turn and
  every aspect of the spectral data of the operator $H$ is the sum (or
  union) of the spectral data of the restricted parts.  In particular,
  since $\cH_\omega = \cc{\cH_{\cc{\omega}}}$, the eigenvalues of the
  corresponding restrictions are equal and thus each eigenvalue of the
  restriction $H_\omega = H\big|_{\cH_\omega}$ enters the spectrum of
  $H$ with multiplicity two.  Kinematically, these eigenvalues
  correspond to the rotational wobbles of the antenna beam.
\end{remark}

\begin{figure}[h]
	\centering
	\includegraphics[width=0.375\textwidth]{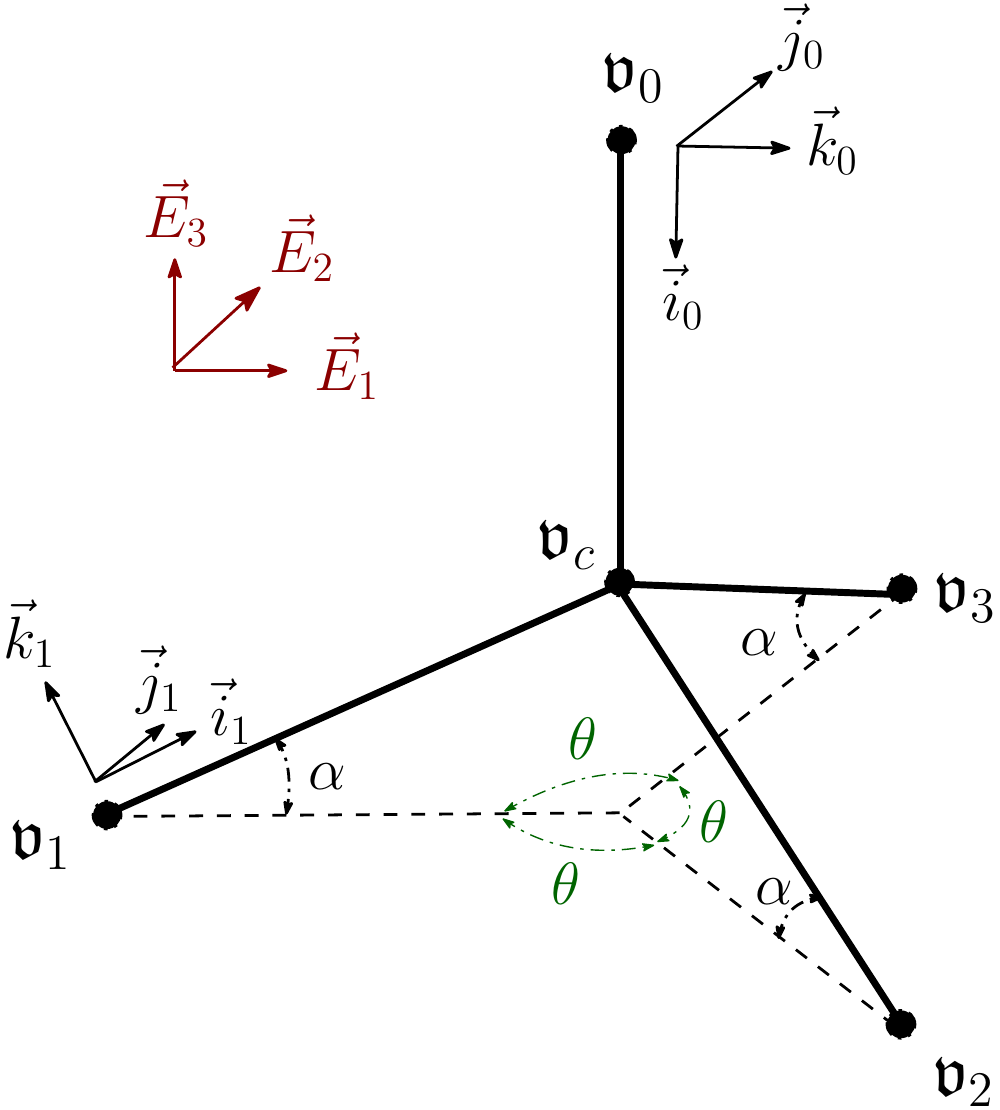}
	\caption{Geometry of three dimensional star graph in its equilibrium
		state.}
	\label{threeDimGraphExample}
\end{figure}

\begin{proof}[Proof of Theorem~\ref{thm:reducing_antenna_tower}]
  The graph $\Gamma$ is invariant under the following geometric
  transformations and their products: $R$ acting as the rotation
  counterclockwise by $\theta = 2\pi/3$ around the axis $-\vec{i}_0$,
  and $F$ acting as the reflection with respect to the plane spanned
  by $\vec{i}_1$ and $\vec{i}_0$.  These transformations generate the
  group $G = D_3$, the dihedral group\footnote{It is isomorphic to the
    symmetric group $S_3$ of permutations of 3 elements.} of degree
  $3$, according to the presentation
\begin{equation}
  \label{eq:group_presentation}
  G = \left\langle R,F ~| ~R^3 = I, ~F^2 = I, ~FRFR = I \right\rangle.
\end{equation}

Let us now fix the following local bases: take $\vec{j}_1$ to be
orthogonal to the plane spanned by $\vec{i}_0$ and $\vec{i}_1$; let
$\vec{j}_0=\vec{j}_1$, see Figure~\ref{threeDimGraphExample}; this
determines $\vec{k}_0$ and $\vec{k}_1$.  More specifically, 
\begin{equation}
  \label{eq:4star_bases}
  \vec{i}_0 =
  \begin{pmatrix}
    0\\0\\-1
  \end{pmatrix},
  \ 
  \vec{j}_0 =
  \begin{pmatrix}
    0\\1\\0
  \end{pmatrix},
  \ 
  \vec{k}_0 = 
  \begin{pmatrix}
    1\\0\\0
  \end{pmatrix},
  \quad
  \vec{i}_1 =
  \begin{pmatrix}
    \cos(\alpha)\\0\\\sin(\alpha)
  \end{pmatrix},
  \ 
  \vec{j}_0 =
  \begin{pmatrix}
    0\\1\\0
  \end{pmatrix},
  \ 
  \vec{k}_1 =
  \begin{pmatrix}
    -\sin(\alpha)\\0\\\cos(\alpha)
  \end{pmatrix}.
\end{equation}
We also assume
\begin{equation}
  \label{eq:rotation_fo_basis}
  \vec{i}_2 = R\vec{i}_1,\quad
  \vec{j}_2 = R\vec{j}_1,\quad 
  \vec{k}_2 = R\vec{k}_1,
  \quad\mbox{and}\quad
  \vec{i}_3 = R\vec{i}_2,\quad
  \vec{j}_3 = R\vec{j}_2,\quad 
  \vec{k}_3 = R\vec{k}_2.
\end{equation}
With these choices, we obtain the following geometric
representation\footnote{With a slight abuse of notation we use the
  same letters for the group elements and the matrices realizing their
  action on the corresponding linear space} of $G$,
\begin{equation}
  \label{eq:matrices_R_F}
  R =
  \begin{pmatrix}
    \cos(\theta) & -\sin(\theta) & 0\\
    \sin(\theta) & \cos(\theta) & 0\\
    0 & 0 & 1
  \end{pmatrix},
  \qquad
  F = 
  \begin{pmatrix}
    1 & 0 & 0 \\
    0 & -1 & 0 \\
    0 & 0 & 1
  \end{pmatrix}.
\end{equation}

We now have to describe the action of $G$ on the functions
$\{v_e,w_e,u_e,\eta_e\}$ defined along the beams $e$ of the frame.  We
note that the space transformation $T$ maps a vector $\vec{g}$ based
at a material point $\vec{x}$ to the vector $T\vec{g}$ at the point
$T\vec{x}$.  Therefore
\begin{equation}
  \label{eq:displacement_mapped}
  R :
  \begin{pmatrix}
    \vec{g}_0 \\ \vec{g}_1 \\ \vec{g}_2 \\ \vec{g}_3
  \end{pmatrix}
  \mapsto
  \begin{pmatrix}
    R\vec{g}_0 \\ R\vec{g}_3 \\ R\vec{g}_1 \\ R\vec{g}_2    
  \end{pmatrix}
  \quad\mbox{and}\quad
  F :
  \begin{pmatrix}
    \vec{g}_0 \\ \vec{g}_1 \\ \vec{g}_2 \\ \vec{g}_3
  \end{pmatrix}
  \mapsto
  \begin{pmatrix}
    F\vec{g}_0 \\ F\vec{g}_1 \\ F\vec{g}_3 \\ F\vec{g}_2    
  \end{pmatrix}
\end{equation}
On the other hand a transformation $\cR$ based at a point $\vec{x}$
gets mapped to $T\cR T^{-1}$ at the point $T\vec{x}$.  By direct
calculation from definition \eqref{eq:torsion_def} (or by using
geometric intuition), we obtain that
\begin{equation}
  \label{eq:torsion_mapped}
  R :
  \begin{pmatrix}
    \eta_0 \\ \eta_1 \\ \eta_2 \\ \eta_3
  \end{pmatrix}
  \mapsto
  \begin{pmatrix}
    \eta_0 \\ \eta_3 \\ \eta_1 \\ \eta_2    
  \end{pmatrix}
  \quad\mbox{and}\quad
  F :
  \begin{pmatrix}
    \eta_0 \\ \eta_1 \\ \eta_2 \\ \eta_3
  \end{pmatrix}
  \mapsto
  \begin{pmatrix}
    -\eta_0 \\ -\eta_1 \\ -\eta_3 \\ -\eta_2    
  \end{pmatrix}.
\end{equation}
Since, by definition,
$\vec{g}_n = u_n \vec{i}_n + w_n \vec{j}_n + v_n \vec{k}_n$, mapping
rules \eqref{eq:displacement_mapped} and \eqref{eq:torsion_mapped}
induce the linear action of $R$ and $F$ on the space $L^2(\Gamma)$.
These linear actions commute with the operator $H$ (in particular,
preserving its domain) and form a representation of the group $G$.

It is well-known
\cite{Mackey_representations_in_physics,FultonHarris_RT} that the
Hilbert space $L^2(\Gamma)$ with an action of a finite group $G$ on it
can be decomposed into a sum of \textbf{isotypic components}, i.e.\
the subspaces consisting of all copies of a given irreducible
representation (\textbf{irrep}) contained in $L^2(\Gamma)$.
Concretely, if $\rho$ is an irrep with character
$\chi_\rho : G \to \bC$, then
\begin{equation}
  \label{eq:projector_rho}
  \Prho = \frac{\dim \rho}{|G|} \sum_{g\in G} \cc{\chi_\rho(g)} g
\end{equation}
is the orthogonal projector onto the isotypic component of $\rho$, see
\cite[Equation (2.32)]{FultonHarris_RT}.  Here $\dim\rho$ is the
dimension of the representation and the last $g$ should be understood
as the representative of the group element acting on the Hilbert space
in question.  It is a classical result of representation theory that
the projectors $\Prho$ corresponding to different irreps are mutually
orthogonal and their sum over all irreps of $G$ is identity.  Since
$H$ commutes with every $g\in G$, the projector $\Prho$ commutes with
$H$ and therefore the subspace $\Prho(L^2(\Gamma))$ is reducing.

Our group $G=D_3$ has three irreducible representations, known as
trivial (identity), alternating and standard.  The first two are
one-dimensional and the last one is two-dimensional.  We will denote
the corresponding isotypic components by $\cH_\mathrm{trv}$,
$\cH_\mathrm{alt}$ and $\cH_\mathrm{std}$.  The group $D_3$ contains
the cyclic group $C_3$ as a subgroup; it is generated by the rotations
$R$,
\begin{equation}
  \label{eq:C3_presentation}
  C_3 = \left\langle R ~|~ R^3 = I\right\rangle \subset D_3.  
\end{equation}
The group $C_3$ also have three repesentations: identity,
$\omega=e^{2\pi i/3}$, and $\cc{\omega}$, all one-dimensional.

The character of the standard representation of $D_3$ is non-zero only
on the elements of $C_3\subset D_3$, more precisely,
\begin{equation}
  \label{eq:std_character}
  \chi_\mathrm{std}(g) =
  \begin{cases}
    \chi_\omega(g) + \chi_{\cc\omega}(g), & g\in C_3, \\
    0, & g\not\in C_3,
  \end{cases}
\end{equation}
which gives
\begin{equation}
  \label{eq:std_proj_decomp}
  \mathbf{P}_\mathrm{std}
  = \frac{2}{|D_3|} \sum_{g\in D_3} \cc{\chi_\mathrm{std}(g)} g
  = \frac{1}{|C_3|} \sum_{g\in C_3} \left(\cc{\chi_\omega(g)}
    + \cc{\chi_{\cc\omega}(g)}\right) g
  = \mathbf{P}_\omega + \mathbf{P}_{\cc\omega}.
\end{equation}
This, together with observation that
$\chi_{\cc\omega} = \cc{\chi_\omega(g)}$ and therefore
$\mathbf{P}_{\cc\omega} = \cc{\mathbf{P}_\omega}$, proves
decomposition \eqref{eq:reducing_AT} in principle.  To specify details
of the individual spaces as in
equations~\eqref{eq:Hid}--\eqref{eq:Homega}, we can proceed by
definition, computing the ranges of the corresponding projectors
$\mathbf{P}_\mathrm{trv}$, $\mathbf{P}_\mathrm{alt}$ and
$\mathbf{P}_\omega$ using equation~\eqref{eq:projector_rho}.  This
requires summation over $|D_3|=6$ or $|C_3|=3$ group elements and is
rather tedious. 

Instead we will use the fact that all representations are
one-dimensional and therefore the isotypic components are isomorphic
to the corresponding spaces of \textbf{intertwiners}.\footnote{The
  action of an operator on the space of intertwiners is somewhat
  different from its action on the isotypic component and results in a
  decomposition with some favorable spectral properties
  \cite{BanBerJoyLiu_prep17}.  However the two notions coincide for a
  one-dimensional representation.}  Recall the following definition:
let $G$ be a finite group with two representations
$\rho : G \rightarrow \text{GL}(V_\rho)$ and
$\sigma : G \rightarrow \text{GL}(V_\sigma)$. The vector space
homomorphism $\phi :V_{\rho} \rightarrow V_{\sigma}$ is said to be an
\textbf{intertwiner} if it satisfies
\begin{equation}
  \label{eq:intertwiner}
  \forall g \in G, \quad \sigma(g)\phi = \phi \rho(g).
\end{equation}
The vector space of all intertwiners is denoted by
$\Hom_G(V_\rho,V_\sigma)$.

In our setting $V_\sigma$ is $L^2(\Gamma)$ with the action of $G$
defined by \eqref{eq:displacement_mapped}--\eqref{eq:torsion_mapped};
$\rho$ will be one of the irreps of $G$.  Since we only deal with
1-dimensional representations (i.e.\ $\rho(g)$ is a scalar), $\phi$
are simply the eigevectors of $R$ and $F$ on $L^2(\Gamma)$.  The
detailed calculations are performed in the next three sections.
\end{proof}

%%%%%%%%%%%%%%
\subsubsection{Trivial irrep}
We first consider the trivial (identity) irreducible representation
$\rho = \mathrm{trv}$ of $D_3$, given by $V_\rho = \bC^1$, $\rho(R)=1$,
and $\rho(F)=1$.  We will now use intertwiner
condition~\eqref{eq:intertwiner} to calculate the intertwiners which
are just vectors from $L^2(\Gamma)$.  Naturally, it is enough to check
condition~\eqref{eq:intertwiner} only on the generators of the group.

Applying the intertwiner condition to $R$ from
\eqref{eq:displacement_mapped}, we have
\begin{equation*}
  R
  \begin{pmatrix}
    \vec{g}_0 \\ \vec{g}_1 \\ \vec{g}_2 \\ \vec{g}_3
  \end{pmatrix}
  = 
  \begin{pmatrix}
    R\vec{g}_0 \\ R\vec{g}_3 \\ R\vec{g}_1 \\ R\vec{g}_2    
  \end{pmatrix}
  = \begin{pmatrix}
    \vec{g}_0 \\ \vec{g}_1 \\ \vec{g}_2 \\ \vec{g}_3
  \end{pmatrix} \cdot 1,
\end{equation*}
in particular $R \vec{g}_0 = \vec{g}_0$,
which means that for any $x$, $\vec{g}_0(x)$ is proportional to the
eigenvector of $R$ with eigenvalue 1, see \eqref{eq:matrices_R_F} and therefore
\begin{equation}
  \label{eq:id_v0w0_zero}
  v_0 \equiv 0
  \qquad\mbox{and}\qquad
  w_0\equiv 0.  
\end{equation}
We also have $R \vec{g}_1 = \vec{g}_2$ and $R \vec{g}_2 = \vec{g}_3$
yielding, since $R$ is orthogonal,
\begin{equation*}
  u_2 := \vec{i}_2 \cdot \vec{g}_2 
  = \big(R \vec{i}_1\big) \cdot \big( R\vec{g}_1 \big)
  = \vec{i}_1 \cdot \vec{g}_1
  = u_1,
\end{equation*}
and so on,
\begin{equation}
  \label{eq:all_around_equal}
  v_1=v_2=v_3,\qquad w_1=w_2=w_3
  \qquad\mbox{and}\qquad
  u_1=u_2=u_3.
\end{equation}

Switching our attention to $F$, we have from equation
\eqref{eq:displacement_mapped}
\begin{equation*}
  w_1 := \vec{j}_1 \cdot \vec{g}_1
  = \vec{j}_1 \cdot \big( F\vec{g}_1 \big)
  = \big( F\vec{j}_1 \big) \cdot \vec{g}_1
  = -\vec{j}_1 \cdot \vec{g}_1
  = -w_1,
\end{equation*}
and thus
\begin{equation}
  \label{eq:w_are_0}
  w_1=w_2=w_3 \equiv 0.  
\end{equation}
Finally,
equation~\eqref{eq:torsion_mapped} gives
\begin{equation*}
  F
  \begin{pmatrix}
    \eta_0 \\ \eta_1 \\ \eta_2 \\ \eta_3
  \end{pmatrix}
  =
  \begin{pmatrix}
    -\eta_0 \\ -\eta_1 \\ -\eta_3 \\ -\eta_2    
  \end{pmatrix}
  =
  \begin{pmatrix}
    \eta_0 \\ \eta_1 \\ \eta_2 \\ \eta_3
  \end{pmatrix},
\end{equation*}
therefore all torsion fields are identically zero,
\begin{equation}
  \label{eq:torsion_zero_id}
  \eta_1=\eta_2=\eta_3 \equiv 0.  
\end{equation}
It is an easy check that the
remaining conditions result in no additional restrictions on the
undetermined $u_0$, $u$ and $v$, where we abbreviated $u := u_s$ and
$v := v_s$ for $s=1,2,3$.  The conditions we obtained describe the space
$\cH_{\mathrm{trv}}$.

We now describe the \emph{domain} of the operator
$H_\mathrm{trv} := H \big|_{\cH_\mathrm{trv}}$.  First of all, the
functions $u_0$, $u$ and $v$ need to come from Sobolev spaces
$H^2(0,\ell_0)$, $H^2(0,\ell)$ and $H^4(0,\ell)$ correspondingly.
Continuity of displacement (\ref{primaryBcThmDisp}) at the central
joint becomes
\begin{equation}
u(\ell) \vec i_1 + v(\ell) \vec k_1 =  u_0(\ell_0) \vec i_0.
\end{equation}
Expressing $\vec i_0$ in the local coordinate system of $e_1$, i.e.
\begin{equation}
\label{i0Expansion}
\vec i_0 = -\sin(\alpha) \vec i_1 - \cos(\alpha) \vec k_1
\end{equation}
implies that at vertex $c$ the relations 
\begin{equation}
\label{trivialB1}
u(\ell) + u_0(\ell_0) \sin(\alpha) = 0, \hspace{1cm} v(\ell) + u_0(\ell_0) \cos(\alpha) = 0
\end{equation}
should be satisfied. Using \eqref{eq:id_v0w0_zero},
\eqref{eq:all_around_equal} and \eqref{eq:w_are_0} in (\ref{primaryBcThmRota}), we get
\begin{equation}
  \label{trivialB3}
    v'(\ell) = 0.
\end{equation}
Similarly, the balance of forces, equation~(\ref{secondBcThmDisp}),
reduces to
\begin{equation*}
  av'''(\ell)(\vec k_1 + \vec k_2 + \vec k_3)
  + cu'(\ell)(\vec i_1 + \vec i_2 + \vec i_3)
  + cu'_0(\ell_0) \vec i_0 = 0
\end{equation*}
Applying identities
$\vec i_1 + \vec i_2 + \vec i_3 = -3 \vec i_0 \sin(\alpha)$ and
$\vec k_1 + \vec k_2 + \vec k_3 = -3 \vec i_0 \cos(\alpha)$, it takes
the form
\begin{equation}
  \label{trivialB2}
  3a v'''(\ell) \cos(\alpha) + 3c u'(\ell) \sin(\alpha) - c_0u'_0(\ell_0) = 0
\end{equation}
Finally, due to \eqref{eq:w_are_0}, \eqref{eq:torsion_zero_id} and
$\vec j_1 + \vec j_2 + \vec j_3 = 0$, the balance of moments condition
(\ref{secondBcThmRota}) will be automatically satisfied.

The fixed end at $\vt_1$ results in conditions
\begin{equation}
  \label{eq:fixed_end_x}
  u(0) = v(0) = v'(0) = 0,
\end{equation}
while free end at $\vt_0$ gives
\begin{equation}
  \label{eq:free_end_x0}
  u_0'(0)=0.
\end{equation}

The eigenproblem for the operator $H_\mathrm{trv}$ becomes
\begin{equation}
  \label{eienProblemTrivial}
  a \frac{d^4v(x)}{dx^4} = \lambda v(x), \hspace{1cm}
  c\frac{d^2u(x)}{dx^2} = \lambda u(x), \hspace{1cm}
  c_0\frac{d^2u_0(x)}{dx^2} = \lambda u_0(x)
\end{equation}
together with conditions \eqref{trivialB1}, \eqref{trivialB3},
\eqref{trivialB2}, \eqref{eq:fixed_end_x} and
\eqref{eq:free_end_x0}.

Imposing conditions \eqref{eq:fixed_end_x} and
\eqref{eq:free_end_x0} first, we get
\begin{equation}
  v(x) = A \big(\sinh(\mu_a x) - \sin(\mu_a x)\big)
  + B \big(\cosh(\mu_a x) - \cos(\mu_a x)\big)
\end{equation}
and
\begin{equation}
  u(x) = C \sin(\beta_c x),
  \qquad
  u_0(x) = D \cos(\beta_{c_0} x)
\end{equation}
where $\mu_a := (\lambda/a)^{1/4}$, $\beta_d := (\lambda/d)^{1/2}$
and $\beta_{c_0} := (\lambda/c_0)^{1/2}$. 
Application of conditions (\ref{trivialB1}), (\ref{trivialB3}) and
(\ref{trivialB2}) brings us to the equivalent problem of finding
non-trivial coefficient vector $(A~B~C~D)^T$ in the kernel of the matrix
$M_\mathrm{trv} = M_\mathrm{trv}(\lambda)$ defined as (see \eqref{eq:notationMatrix} for the definition of notations)
\begin{equation*}
{
    M_\mathrm{trv}(\lambda) =
    \begin{pmatrix}
    C_{\mu_a \ell}^{-} & S_{\mu_a \ell}^{+} & 0 & 0\\
     0 & 0 & S_{\beta_c \ell} & S_{\alpha} C_{\beta_{c_0} \ell_0}  \\
      S_{\mu_a \ell}^{-} & C_{\mu_a \ell}^{-} & 0 & C_{\alpha} C_{\beta_{c_0} \ell_0}\\
      3a\mu_a^3 C_{\alpha} C_{\mu_a \ell}^{+} & 3a\mu_a^3 C_{\alpha} S_{\mu_a \ell}^{-} & 3c\beta_c S_{\alpha} C_{\beta_c \ell} &   c_0\beta_{c_0} S_{\beta_{c_0} \ell_0} 
    \end{pmatrix}.
  }    
\end{equation*}
The eigenvalues $\lambda$ are determined as the points in which
$\det(M_\mathrm{trv}) = 0$.

%%%%%%%%%%%% 
\begin{exmp}
  \label{ex:antenna_numerical}
  We set all material constants and beam lengths to 1.  Figure
  \ref{detVariation}(left) plots the value of $\det(M_\mathrm{trv})$.
  Figure \ref{trivialEigenPlots} shows the components of the first
  eigenfunction corresponding to the trivial representation.  For
  verification purposes, the latter plot was obtained from
  finite-element calculation for the \emph{entire structure}, without
  symmetry reduction, i.e.\ without enforcing conditions
  \eqref{eq:id_v0w0_zero}--\eqref{eq:torsion_zero_id}.  Thus we can
  confirm by a different method that the relevant fields are zero (up
  to a numerical error).
  \begin{figure}
    \center
    \subfigure{
      \includegraphics[width=0.3\textwidth]{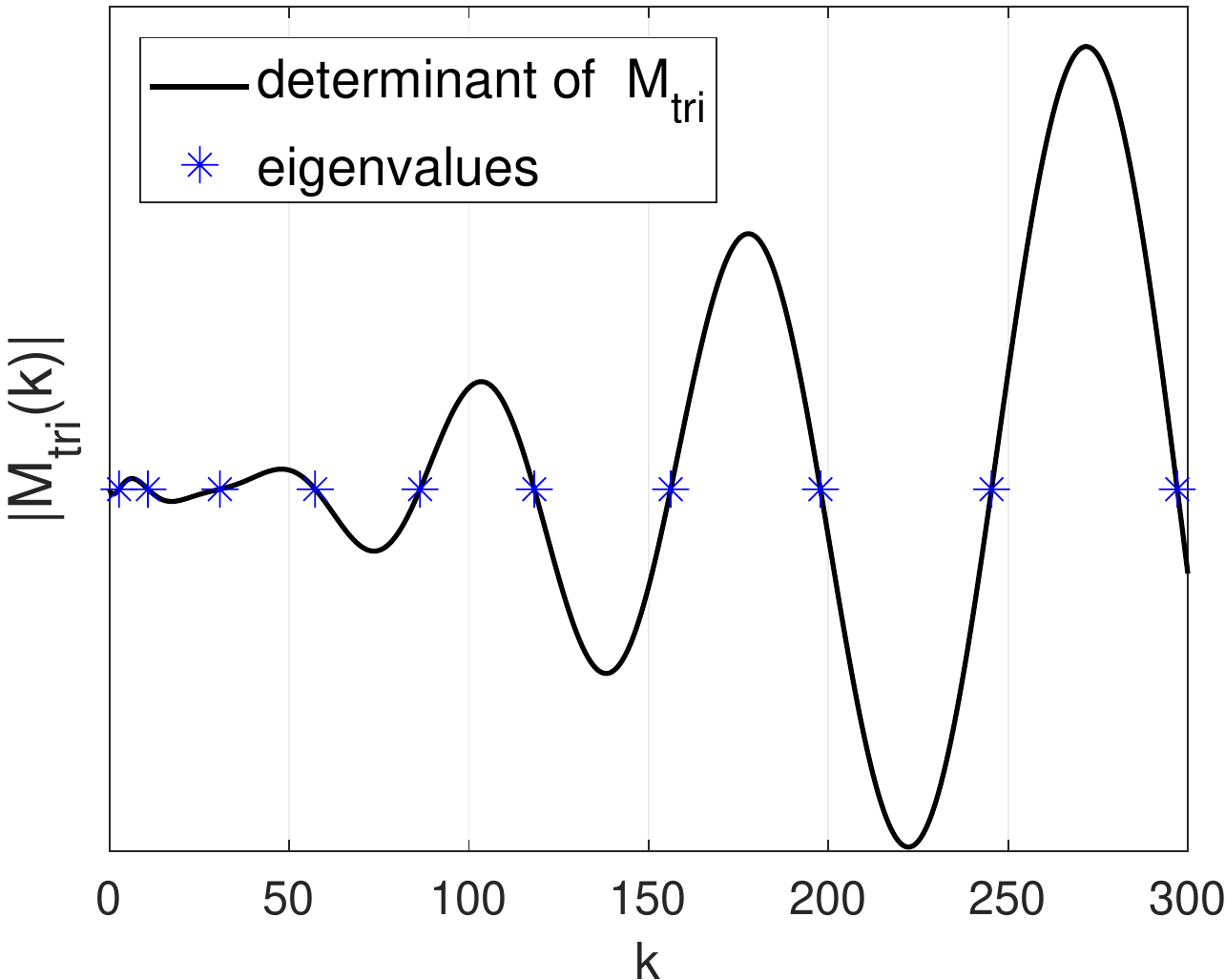}
    }
    \hspace{0.1mm}
    \subfigure{
      \includegraphics[width=0.3\textwidth]{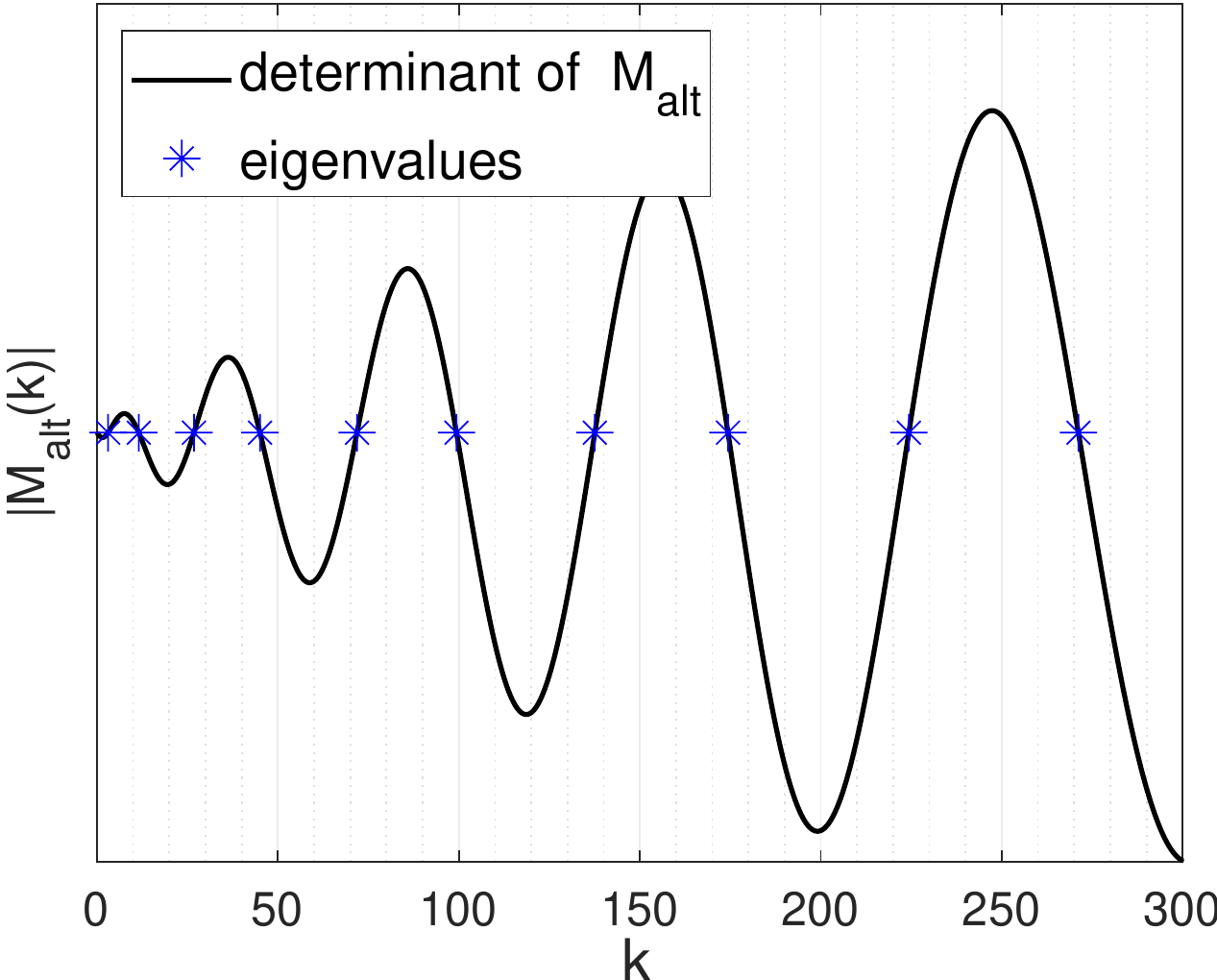}
    }
    \hspace{0.1mm}
    \subfigure{
      \includegraphics[width=0.3\textwidth]{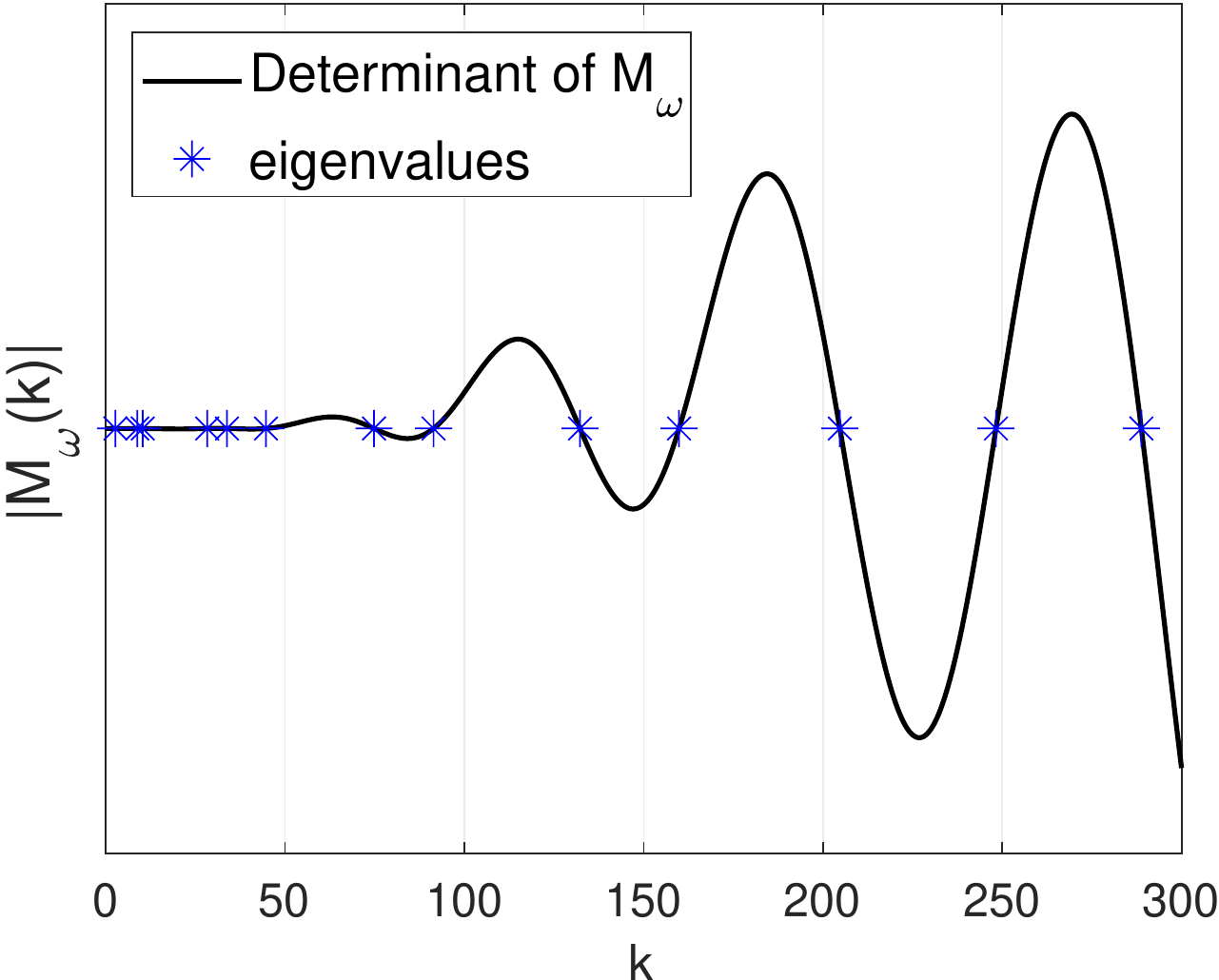}
    }
    \caption{Variation of determinant of matrices corresponding
      irreducible representations and their corresponding eigenvalues:
      (left) trivial $M_\mathrm{trv}$, (middle) alternating
      $M_\mathrm{alt}$, and (right) standard
      $M_\omega$. All the results are based on unit materials parameters and beams lengths.}
    \label{detVariation}
  \end{figure}

  \begin{figure}
    \center
    \subfigure{
      \includegraphics[width=0.45\textwidth]{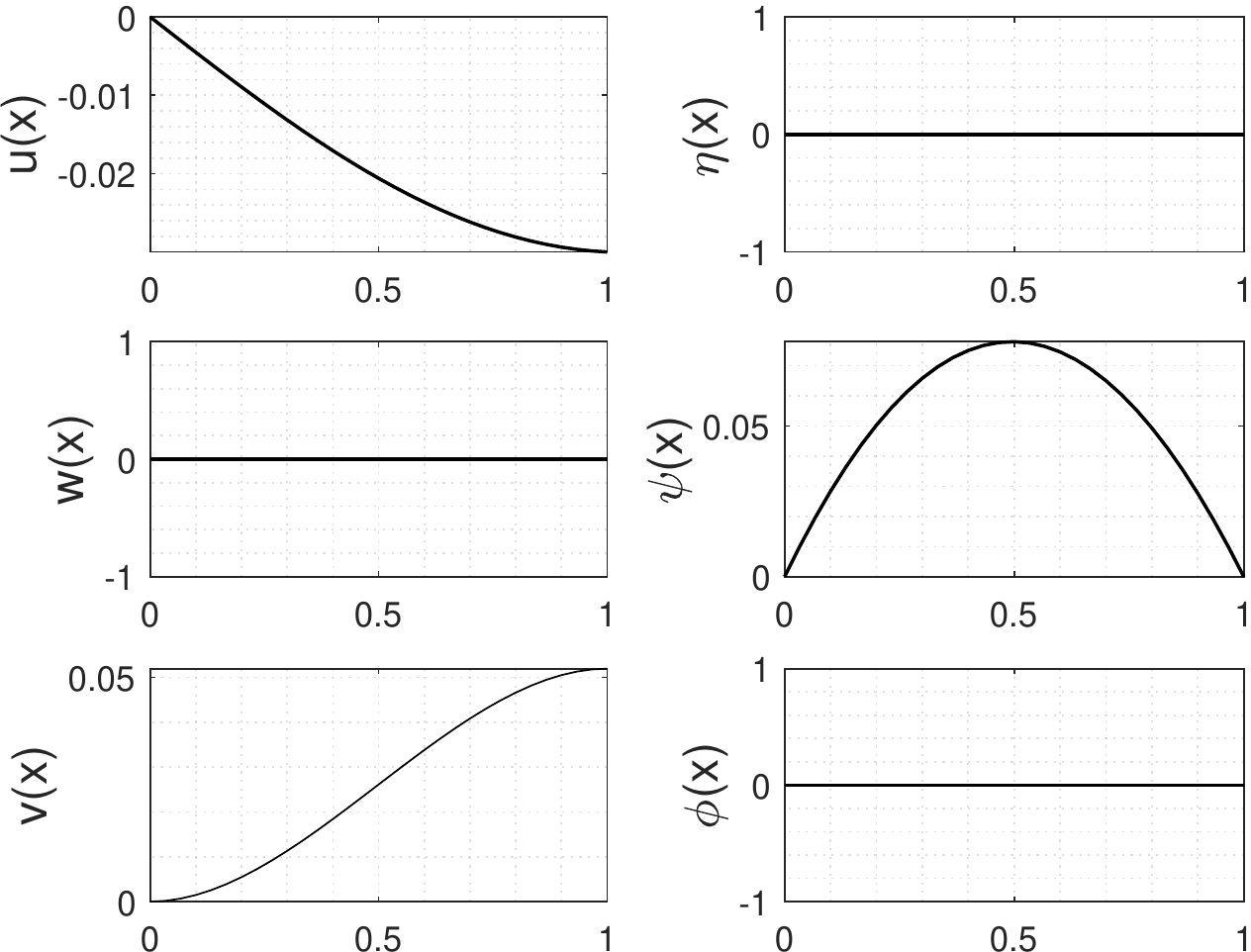}
    }
    \hspace{5mm}
    \subfigure{
      \includegraphics[width=0.45\textwidth]{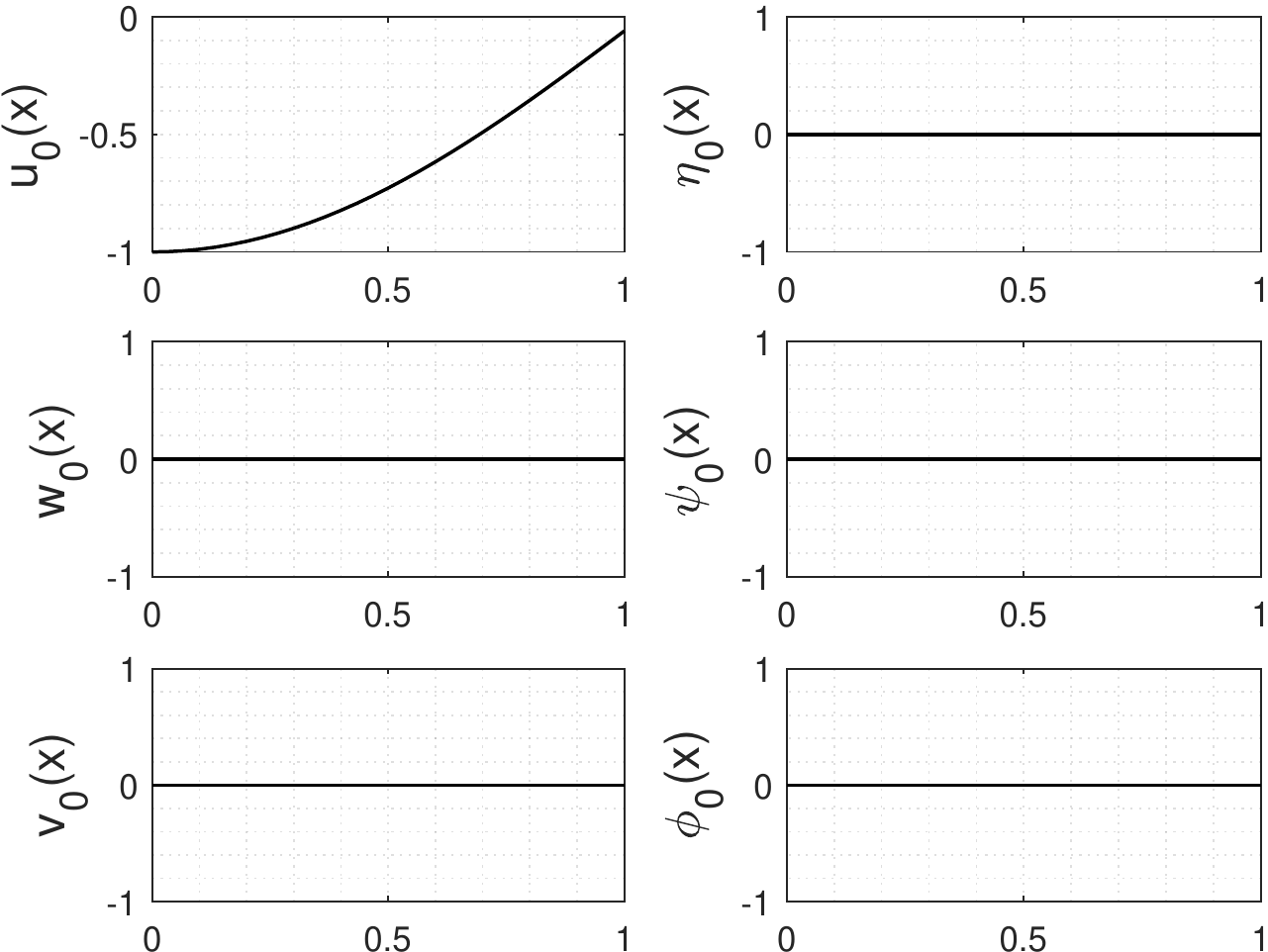}
    }
    \caption{Plot of the components of the first eigenfuction from
      $\cH_\mathrm{trv}$. Plots are obtained from a finite elements
      numerical computation and are displayed in the local coordinate
      system of the corresponding edge. All the results are based on unit materials parameters and beams lengths.}
    \label{trivialEigenPlots}
  \end{figure}
\end{exmp}

%%%%%%%%%%%%%%
\subsubsection{Alternating irrep} 
For the alternating representation, $\rho(R)=1$ and $\rho(F)=-1$.
Indentities~\eqref{eq:id_v0w0_zero} and \eqref{eq:all_around_equal}
still apply in this case, and we also have from
equation~\eqref{eq:torsion_mapped}(left)
\begin{equation}
  \label{eq:alt_eta_equal}
  \eta_1=\eta_2=\eta_3 =: \eta.
\end{equation}
Turning to $F$, condition $F\vec{g}_0 = -\vec{g}_0$ yields
\begin{equation}
  \label{eq:alt_u0_zero}
  u_0 \equiv 0
\end{equation}
(in addition to $v_0=w_0\equiv0$ we already have), while $F\vec{g}_1 =
-\vec{g}_1$ gives
\begin{equation*}
  u_1 := \vec{i}_1 \cdot \vec{g}_1
  = - \vec{i}_1 \cdot \big( F\vec{g}_1 \big)
  = -\big(F\vec{i}_1\big) \cdot \vec{g}_1 
  = -\vec{i}_1 \cdot \vec{g}_1
  = -u_1,
\end{equation*}
and, similarly, $v_1=-v_1$, therefore
\begin{equation}
  \label{eq:alt_uv_zero}
  u_1=u_2=u_3 \equiv 0\qquad\mbox{and}\qquad  
  v_1=v_2=v_3 \equiv 0.
\end{equation}
Condition \eqref{eq:torsion_mapped}(right) results in no further
restrictions.  We have arrived to \eqref{eq:Halt} and the problem is
fully described in terms of remaining degrees of freeedom $\eta_0$,
$w$ and $\eta$.

We will now describe the domain of the operator $H_\mathrm{alt} :=
H\big|_{\cH_\mathrm{alt}}$.  Identities \eqref{eq:id_v0w0_zero},
\eqref{eq:alt_u0_zero} and \eqref{eq:alt_uv_zero} reduce the
displacement continuity to 
\begin{equation}
  \label{AlternatB1}
  w(\ell) = 0.
\end{equation}
Using (\ref{i0Expansion}) and the continuity of rotation reduces to
\begin{equation}
  \label{AlternatB2}
  \eta(\ell) + \eta_0(\ell_0)\sin(\alpha) = 0, \hspace{1cm} 
  w'(\ell) + \eta_0(\ell_0) \cos(\alpha) = 0.
\end{equation}
Balance of forces, equation~\eqref{secondBcThmDisp} becomes
$bw'''(\ell)(\vec j_1 + \vec j_2 + \vec j_3) = 0$, which is satisfied
automatically due to $\vec j_1 + \vec j_2 + \vec j_3 = 0$.  Finally,
balance of moments, equation~(\ref{secondBcThmRota}), takes the form
\begin{equation*}
  -bw''(\ell)(\vec k_1 + \vec k_2 + \vec k_3)
  + d \eta'(\ell)(\vec i_1 + \vec i_2 + \vec i_3)
  + d \eta_0'(\ell_0) = 0,
\end{equation*}
and, via identities $\vec i_1 + \vec i_2 + \vec i_3 = -3 \vec i_0
\sin(\alpha)$ and $\vec k_1 + \vec k_2 + \vec k_3 = -3 \vec i_0
\cos(\alpha)$, reduces to 
\begin{equation}
  \label{AlternatB3}
  3bw''(\ell)\cos(\alpha) - 3d\eta'(\ell) \sin(\alpha) + d_0 \eta'_0(\ell_0) = 0.
\end{equation}
The conditions at the degree 1 endpoints result in
\begin{equation}
  \label{eq:AlternatB4}
  \eta(0)=w(0)=w'(0)=0
  \qquad\mbox{and}\qquad
  \eta_0'(0)=0.
\end{equation}

Solving the eigenvalue equations and imposing endpoint
conditions~\eqref{eq:AlternatB4} results in
\begin{equation*}
  w(x) = A \big(\sinh(\mu_b x) - \sin(\mu_b x)\big)
  + B \big(\cosh(\mu_b x) - \cos(\mu_b x)\big)
\end{equation*}
and
\begin{equation*}
  \eta(x) = C \sin(\beta_d x), \hspace{1cm} \eta_0(x) = D \cos(\beta_{d_0} x)
\end{equation*}
where $\mu_b = (\lambda/b)^{1/4}$, $\beta_d = (\lambda/d)^{1/2}$ and
$\beta_{d_0} = (\lambda/d_0)^{1/2}$.  Applying the vertex
conditions (\ref{AlternatB1}), (\ref{AlternatB2}) and
(\ref{AlternatB3}), the eigenvalue problem reduces to condition
$\det(M_\mathrm{alt}(\lambda)) = 0$, where
\[
  { M_\mathrm{alt}(\lambda) =
    \begin{pmatrix}
      S_{\mu_b \ell}^{-} & C_{\mu_b \ell}^{-} & 0 & 0 \\
      0 & 0 & S_{\beta_d \ell} & S_{\alpha} C_{\beta_{d_0} \ell_0}\\
      \mu_b C_{\mu_b \ell}^{-} & \mu_b S_{\mu_b \ell}^{+} &0 & C_{\alpha} C_{\beta_{d_0} \ell_0} \\
      3b\mu_b^2 C_{\alpha} S_{\mu_b \ell}^{+} & 3b\mu_b^2 C_{\alpha} C_{\mu_b \ell}^{+} & -3\mu d S_{\alpha}C_{\beta_d \ell} & -d_0\beta_{d_0} S_{\beta_{d_0} \ell_0}
    \end{pmatrix}.
  }
\]

%%%%%%%%%%%%
\begin{exmp}
  Continuing Example~\ref{ex:antenna_numerical}, Figure
  \ref{detVariation}(middle) plots the determinant of the matrix
  $M_\mathrm{alt}$ with the roots highlighted.  Figure \ref{altEigenPlots}
  shows the components of the eigenfunction corresponding to the first
  eigenvalue of alternating representation.
  
  \begin{figure}[htbp]
    \center
    \subfigure{
      \includegraphics[width=0.45\textwidth]{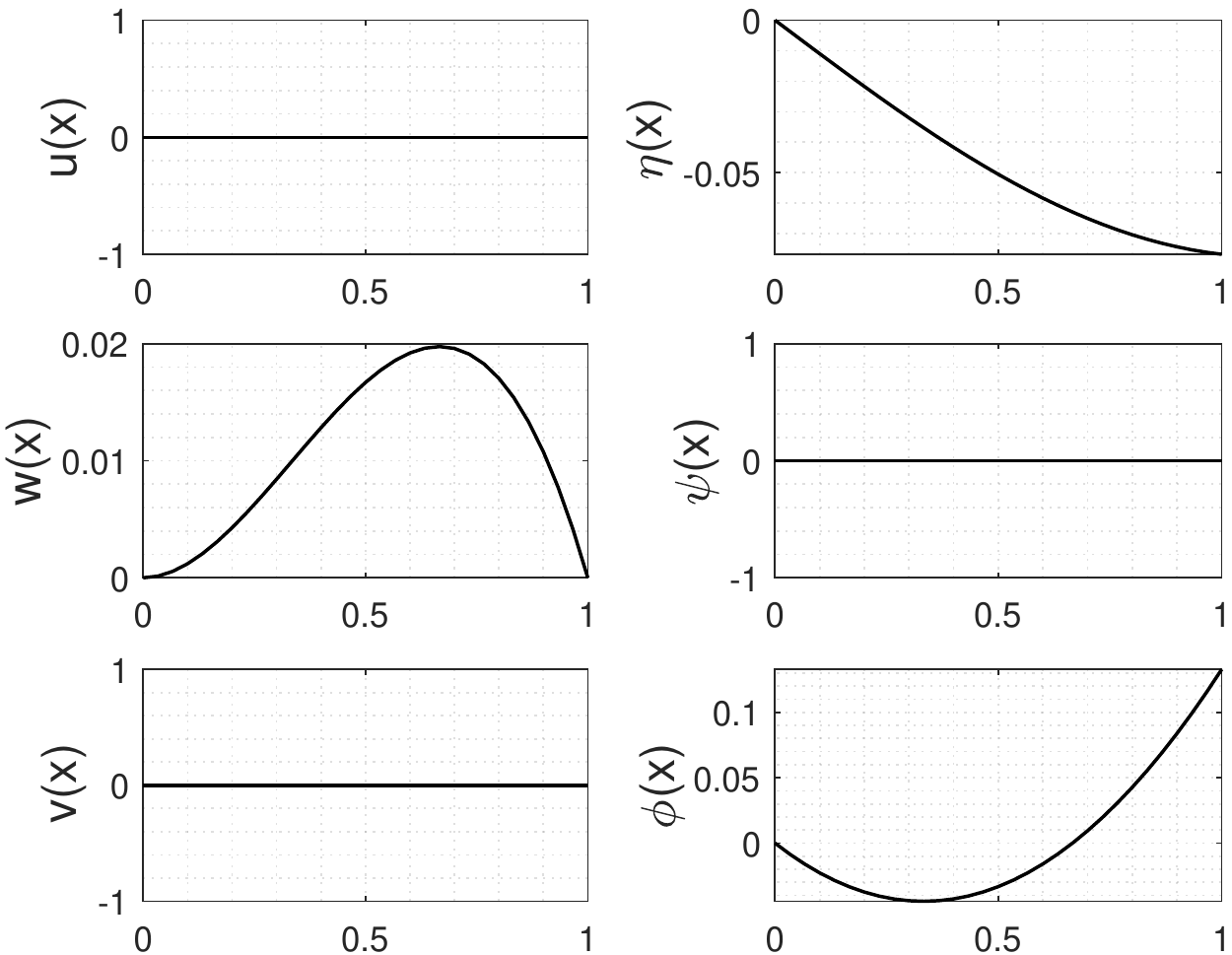}
    }
    \hspace{5mm}
    \subfigure{
      \includegraphics[width=0.45\textwidth]{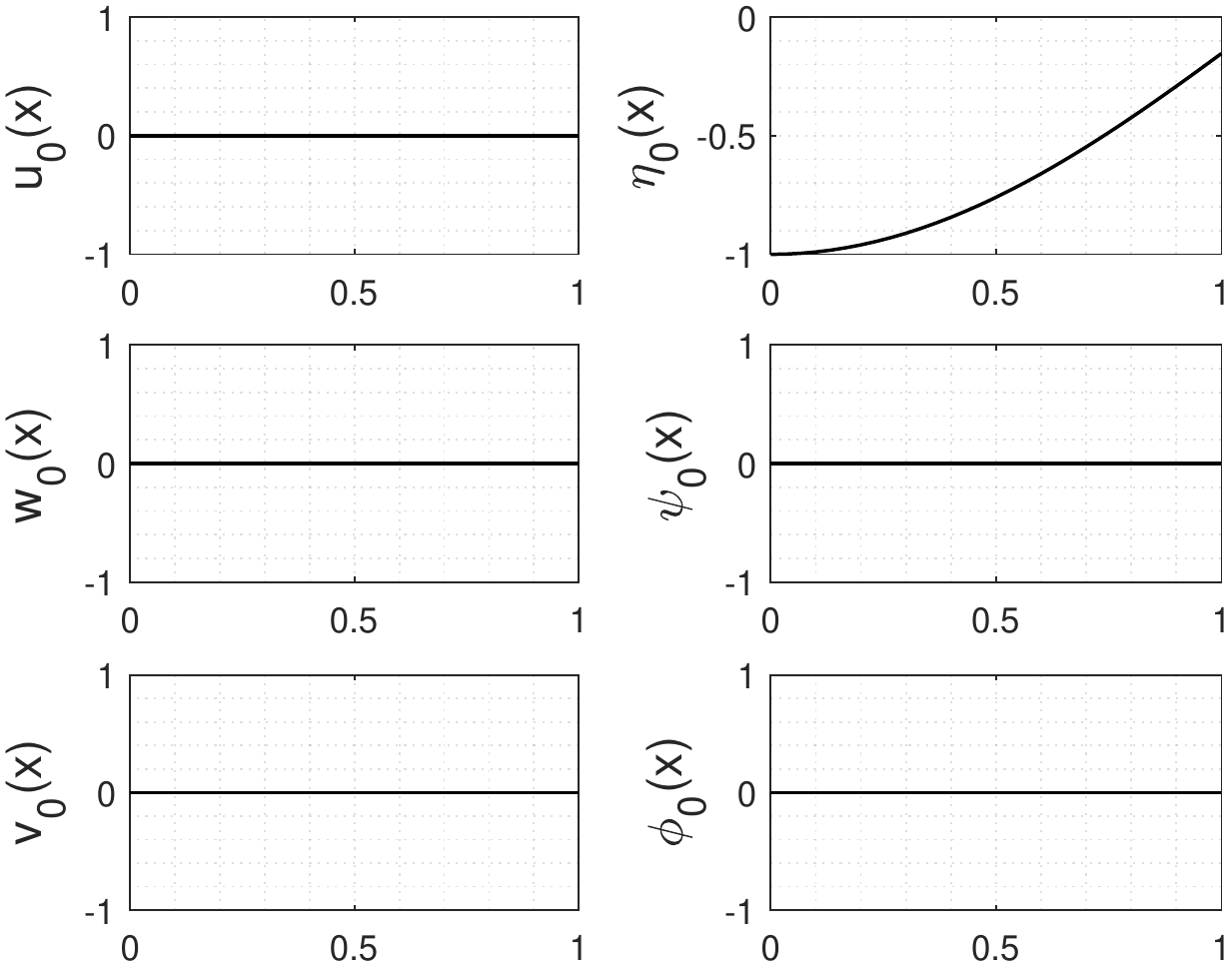}
    }
    \caption{Plot of the components of the first eigenfuction from
      $\cH_\mathrm{alt}$. Plots are obtained from a finite elements
      numerical computation and are displayed in the local coordinate
      system of the corresponding edge. All the results are based on unit materials parameters and beams lengths.}
    \label{altEigenPlots}
  \end{figure}
\end{exmp}

%%%%%%%%%%%%%%
\subsubsection{Irrep $\omega$ of $C_3$}
To derive the description of $\cH_\omega$, we only have the conditions
\begin{equation}
  \begin{pmatrix}
    R\vec{g}_0 \\ R\vec{g}_3 \\ R\vec{g}_1 \\ R\vec{g}_2    
  \end{pmatrix}
  = \omega
  \begin{pmatrix}
    \vec{g}_0 \\ \vec{g}_1 \\ \vec{g}_2 \\ \vec{g}_3
  \end{pmatrix}
  \qquad\mbox{and}\qquad
  \begin{pmatrix}
    \eta_0 \\ \eta_3 \\ \eta_1 \\ \eta_2    
  \end{pmatrix}
  = \omega
  \begin{pmatrix}
    \eta_0 \\ \eta_1 \\ \eta_2 \\ \eta_3
  \end{pmatrix},
\end{equation}
obtained by applying $\rho(R)=\omega:=e^{2\pi i/3}$ to
\eqref{eq:intertwiner}, \eqref{eq:displacement_mapped} and
\eqref{eq:torsion_mapped}.

For beam $e_0$ we have
\begin{equation}
  \label{eq:w_beam0}
  u_0 = 0, \qquad \eta_0 = 0, \qquad
  w_0 = i v_0,
\end{equation}
while on the other beams,
\begin{align}
  \label{eq:w_legs}
  v_2 &= \omega v_1,
  &w_2 &= \omega w_1,
  &u_2 &= \omega u_1,
  &\eta_2 &= \omega \eta_1,\\
  v_3 &= \cc\omega v_1
  &w_3 &= \cc\omega w_1
  &u_3 &= \cc\omega u_1
  &\eta_3 &= \cc\omega \eta_1.
\end{align}

Displacement continuity, equation (\ref{primaryBcThmDisp}), for the
edges $e_1$ and $e_0$ has the form
\begin{equation*}
  u(\ell) \vec i_1 + w(\ell) \vec j_1 + v(\ell) \vec k_1
  = w_0(\ell_0) \vec j_0 + v_0(\ell_0) \vec k_0,
\end{equation*}
where, as before, we abbreviated
$(v,w,u,\eta) := (v_1,w_1,u_1,\eta_1)$.  Applying
$\vec j_0 = \vec j_1$ and
$\vec k_0 = \cos(\alpha) \vec i_1 -\sin(\alpha) \vec k_1$ (see
Figure~\ref{threeDimGraphExample}) and
equation~\eqref{eq:w_beam0}, the vertex condition reduces to
\begin{equation}
  \label{standardB1}
  u(\ell) - v_0(\ell_0) \cos(\alpha) = 0,
  \hspace{1cm} v(\ell) + v_0(\ell_0) \sin(\alpha) = 0,
  \hspace{1cm} w(\ell) - iv_0(\ell_0) = 0
\end{equation}
Simple calculations show that same relations as (\ref{standardB1})
holds by replacing $e_1$ with $e_2$ or $e_3$.

Rotation continuity, equation (\ref{primaryBcThmRota}), becomes
\begin{equation*}
  \eta(\ell) \vec i_1 - v'(\ell)\vec j_1 + w'(\ell) \vec k_1
  =
  -v_0'(\ell_0) \vec j_0 + w_0'(\ell_0) \vec k_0,  
\end{equation*}
similarly leading to
\begin{equation}
  \label{standardB3}
  \eta(\ell) - iv_0'(\ell_0)\cos(\alpha) = 0, \hspace{1cm}
  w'(\ell)  + iv_0'(\ell_0)\sin(\alpha) = 0, \hspace{1cm}
  v'(\ell) - v_0'(\ell_0) = 0.
\end{equation}

The vertex condition (\ref{secondBcThmDisp}) reduces to
\begin{multline}
  \label{eq:antena_balance_forces}
    cu'(\ell)(\vec i_1 + \omega \vec i_2 + \bar \omega \vec i_3)
    - bw'''(\ell)(\vec j_1 + \omega \vec j_2 + \bar \omega \vec j_3) \\
    -av'''(\ell)(\vec k_1 + \omega \vec k_2 + \bar \omega \vec k_3)
    -b_0w_0'''(\ell_0) \vec j_0 - a_0v_0'''\vec k_0= 0.
\end{multline}
Identities
\begin{equation*}
  %\label{sumUnitVector}
  \vec i_1 + \omega \vec i_2 + \bar \omega \vec i_3 = \frac{3}{2}
  \begin{pmatrix}
    \cos(\alpha) \\
    i\cos(\alpha)\\
    0 \\
  \end{pmatrix},\hspace{3mm}
  \vec j_1 + \omega \vec j_2 + \bar \omega \vec j_3 = \frac{3}{2}
  \begin{pmatrix}
    -i \\
    1\\
    0 \\
  \end{pmatrix},\hspace{3mm}
  \vec k_1 + \omega \vec k_2 + \bar \omega \vec k_3 = -\frac{3}{2}
  \begin{pmatrix}
    \sin(\alpha) \\
    i \sin(\alpha)\\
    0 \\
  \end{pmatrix}
\end{equation*}
and $\vec j_0 = (0~1~0)^T$, $\vec k_0 = (1~0~0)^T$, $a_0=b_0$ reduce
equation \eqref{eq:antena_balance_forces} to a single linearly
independent condition, namely
\begin{equation}
  \label{standardB2}
  \frac{3}{2}cu'(\ell)\cos(\alpha) + \frac{3}{2}biw'''(\ell)
  + \frac{3}{2}av'''(\ell)\sin(\alpha) - a_0v_0'''(\ell_0) = 0.
\end{equation}

The vertex condition (\ref{secondBcThmRota}) reduces to
\begin{align*}
d\eta'(\ell)(\vec i_1 + \omega \vec i_2 + \bar \omega \vec i_3) 
&-av''(\ell)(\vec j_1 + \omega \vec j_2 + \bar \omega \vec j_3) \\
&+bw''(\ell)(\vec k_1 + \omega \vec k_2 + \bar \omega \vec k_3) 
-a_0v_0''(\ell_0)\vec j_0 + b_0w_0''(\ell_0) \vec k_0 = 0,
\end{align*}
which is equivalent to single equation
\begin{equation}
  \label{standardB4}
  \frac{3}{2}d i\eta'(\ell) \cos(\alpha)
  - \frac{3}{2}biw''(\ell)\sin(\alpha)
  - \frac{3}{2}av''(\ell)- a_0v_0''(\ell_0) = 0.
\end{equation}

With the restrictions given by \eqref{eq:w_beam0} and
\eqref{eq:w_legs}, the eigenvalue problem is
fully determined by functions $v, w, u$ and $\eta$ on the base edges
and the function $v_0$ defined on vertical edge. Solving the corresponding differential quations and imposing free vertex condition at $\vt_0$ and clamped vertex conditions at $\vt_1$, $\vt_2$, and $\vt_3$, the general solutions can be written as
\begin{align*}
    &v(x) = A \big(\sinh(\mu x) - \sin(\mu x)\big)
    + B \big(\cosh(\mu x) - \cos(\mu x)\big) \\
    &iw(x) = C \big(\sinh(\mu x) - \sin(\mu x)\big)
    + D \big(\cosh(\mu x) - \cos(\mu x)\big) \\
    &u(x) = E \sin(\beta x), \\
    &i\eta(x) = F \sin(\beta x)\\
    &v_0(x) = A_0 \big(\sinh(\mu_0 x) + \sin(\mu_0 x)\big)
    + B_0 \big(\cosh(\mu_0 x) + \cos(\mu_0 x)\big) 
\end{align*}
where $\mu = \lambda^{1/4}$, $\beta = \lambda^{1/2}$, $\mu_0 = (\lambda/a_0)^{1/4}$, and $\beta_0 = (\lambda/a_0)^{1/2}$.   

Applying vertex conditions
(\ref{standardB1}), (\ref{standardB3}), (\ref{standardB2}), and
(\ref{standardB4}), we reduce the eigenvalue problem to 
an equation of the form $\det(M_\omega(\lambda)) = 0$.  If we make additional simplifying assumptions that $a_0=a=b=c=1$ the matrix $M_\omega(\lambda)$ becomes 
\[
  {M_\omega(\lambda) =
    \begin{pmatrix}
      C_{\mu \ell}^{-} & S_{\mu \ell}^{-} & 0 & 0 & 0 & 0 & -C_{\mu \ell_0}^{+} & -S_{\mu \ell_0}^{-} \\
      0 & 0 & S_{\mu \ell}^{-} & C_{\mu \ell}^{-} & 0 & 0 & S_{\mu \ell_0}^{+} & C_{\mu \ell_0}^{+} \\
      0 & 0 & 0 & 0 & S_{\beta \ell} & 0 & -C_{\alpha}S_{\mu \ell_0}^{+} & -C_{\alpha}C_{\mu \ell_0}^{+} \\
      0 & 0 & 0 & 0 & 0 & S_{\beta \ell} & \mu C_{\alpha}C_{\mu \ell_0}^{+} & \mu C_{\alpha}S_{\mu \ell_0}^{-} \\
      S_{\mu \ell}^{-} & C_{\mu \ell}^{-} & 0 & 0 & 0 & 0 & S_{\alpha}S_{\mu \ell_0}^{+} & S_{\alpha}C_{\mu \ell_0}^{+} \\
      0 & 0 & C_{\mu \ell}^{-} & S_{\mu \ell}^{+} & 0 & 0 & -S_{\alpha}C_{\mu \ell_0}^{+} & -S_{\alpha}S_{\mu \ell_0}^{-} \\
      \frac{3\mu}{2}S_{\alpha}C_{\mu \ell}^{+} & \frac{3\mu}{2}S_{\alpha}S_{\mu \ell}^{-} & \frac{3\mu}{2}C_{\mu \ell}^{+} & \frac{3\mu}{2}S_{\mu \ell}^{-} & \frac{3}{2}C_{\alpha}C_{\beta \ell} & 0 & -\mu C_{\mu \ell_0}^{-}  & -\mu S_{\mu \ell_0}^{+} \\
      -\frac{3}{2}S_{\mu \ell}^{+} & -\frac{3}{2}C_{\mu \ell}^{+} & -\frac{3}{2}S_{\alpha}S_{\mu \ell}^{+} & -\frac{3}{2}S_{\alpha}C_{\mu \ell}^{+} & 0 & \frac{3}{2}C_{\alpha}C_{\beta \ell} & -S_{\mu \ell_0}^{-} & -C_{\mu \ell_0}^{-}
    \end{pmatrix},
  }
\]
\begin{exmp}
  Continuing Example~\ref{ex:antenna_numerical},
  Figure~\ref{detVariation}(right) plots the $\det(M_\omega)$ with the
  roots highlighted.  Figure \ref{stdEigenPlots} shows the components
  of the \emph{real part} of the eigenfunction corresponding to the
  first eigenvalue of $\omega$ representation of $C_3$.  The real part
  is an eigenfunction of the original problem because if $\Psi$
  is an eigenfunction in $\cH_\omega$, then $\cc\Psi$ is an
  eigenfunction in $\cH_{\cc\omega}$ with the same eigenvalue.
  
  \begin{figure}[htbp]
    \center
    \subfigure{
      \includegraphics[width=0.45\textwidth]{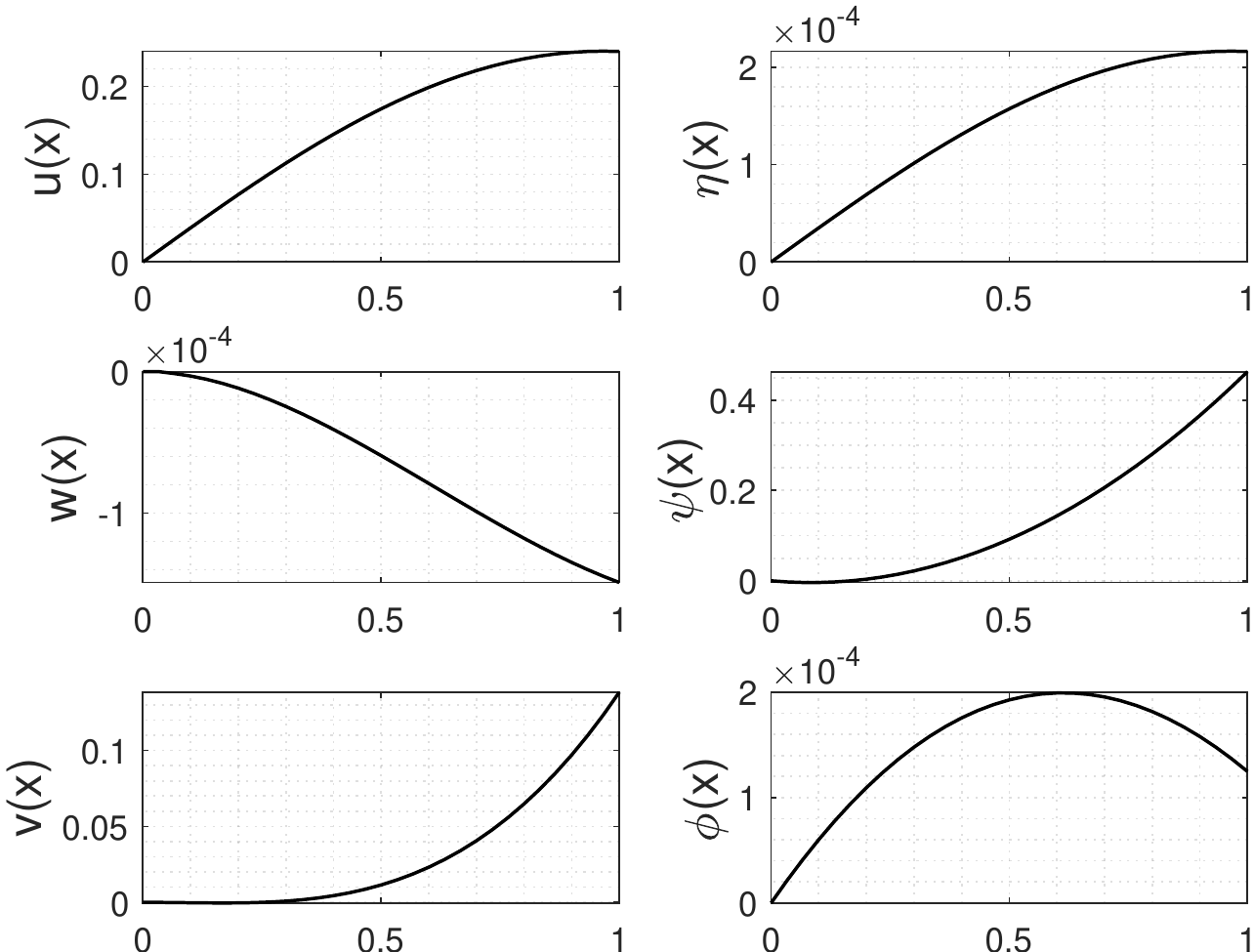}
    }
    \hspace{5mm}
    \subfigure{
      \includegraphics[width=0.45\textwidth]{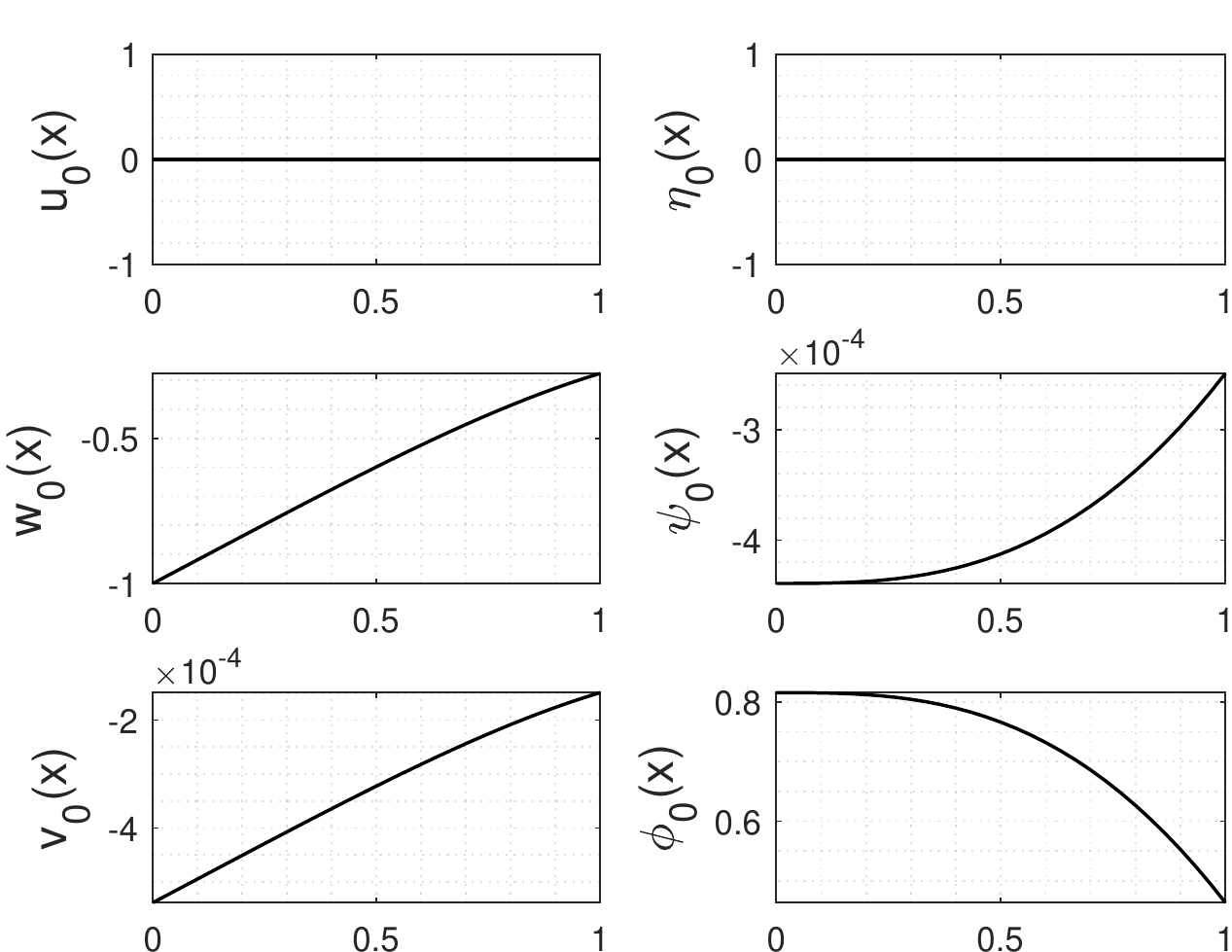}
    }
    \caption{Plot of (first) eigenfunction fields corresponding to the eigenvalue of multiplicity two in edge's local coordinate system by finite element approximation. All the results are based on unit value of materials properties and unit beam's lengths.}
    \label{stdEigenPlots}
  \end{figure}

  Additionally, we display in Figure~\ref{f_ResultsNumWindows} a
  visualization of the two independent eigenfunctions from $\cH_\omega
  \oplus \cH_{\cc\omega}$.
  
  \begin{figure}[htbp]
    \center
    \subfigure{
      \includegraphics[width=0.45\textwidth]{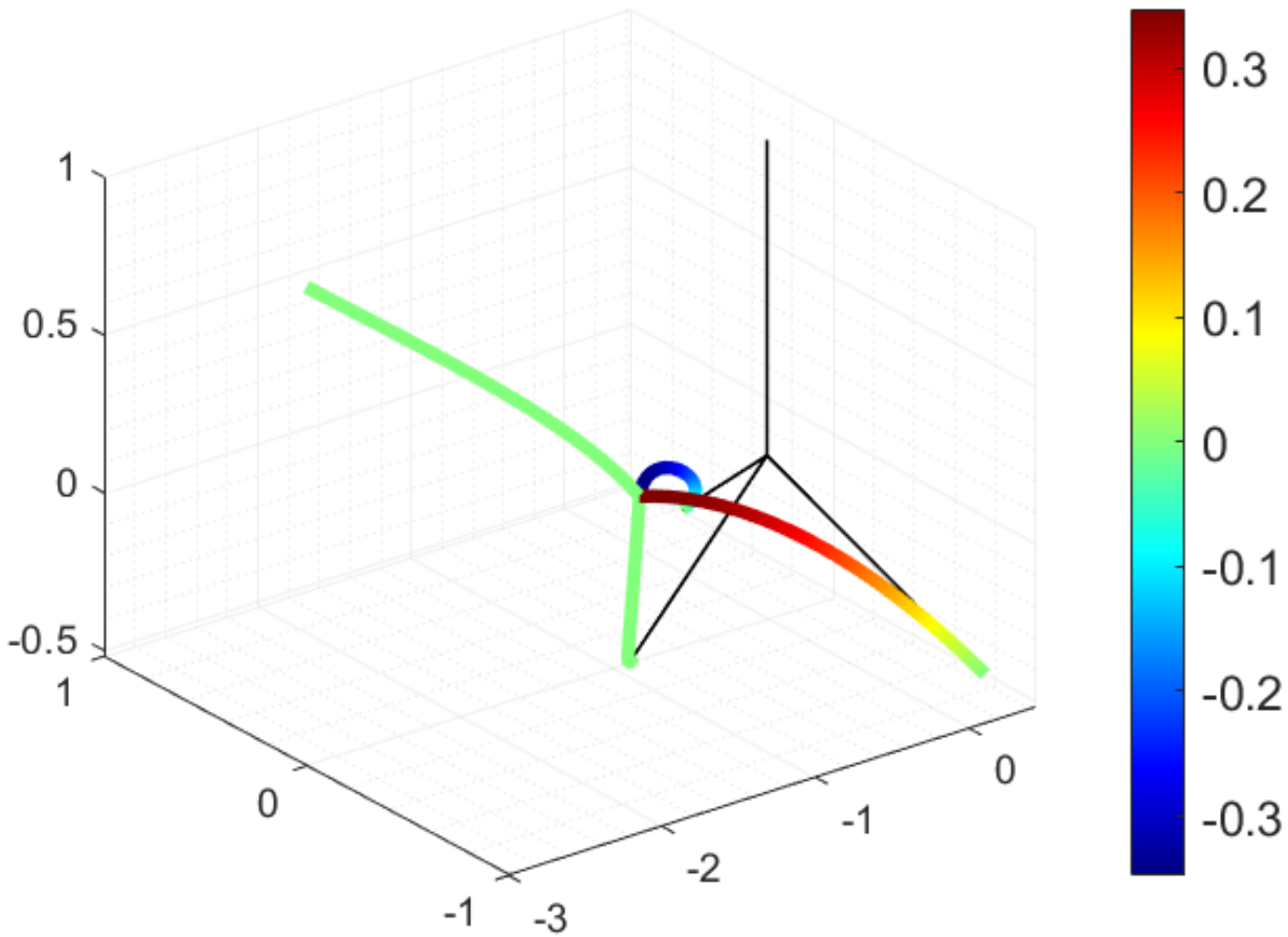}
    }
    \hspace{1mm}
    \subfigure{
      \includegraphics[width=0.45\textwidth]{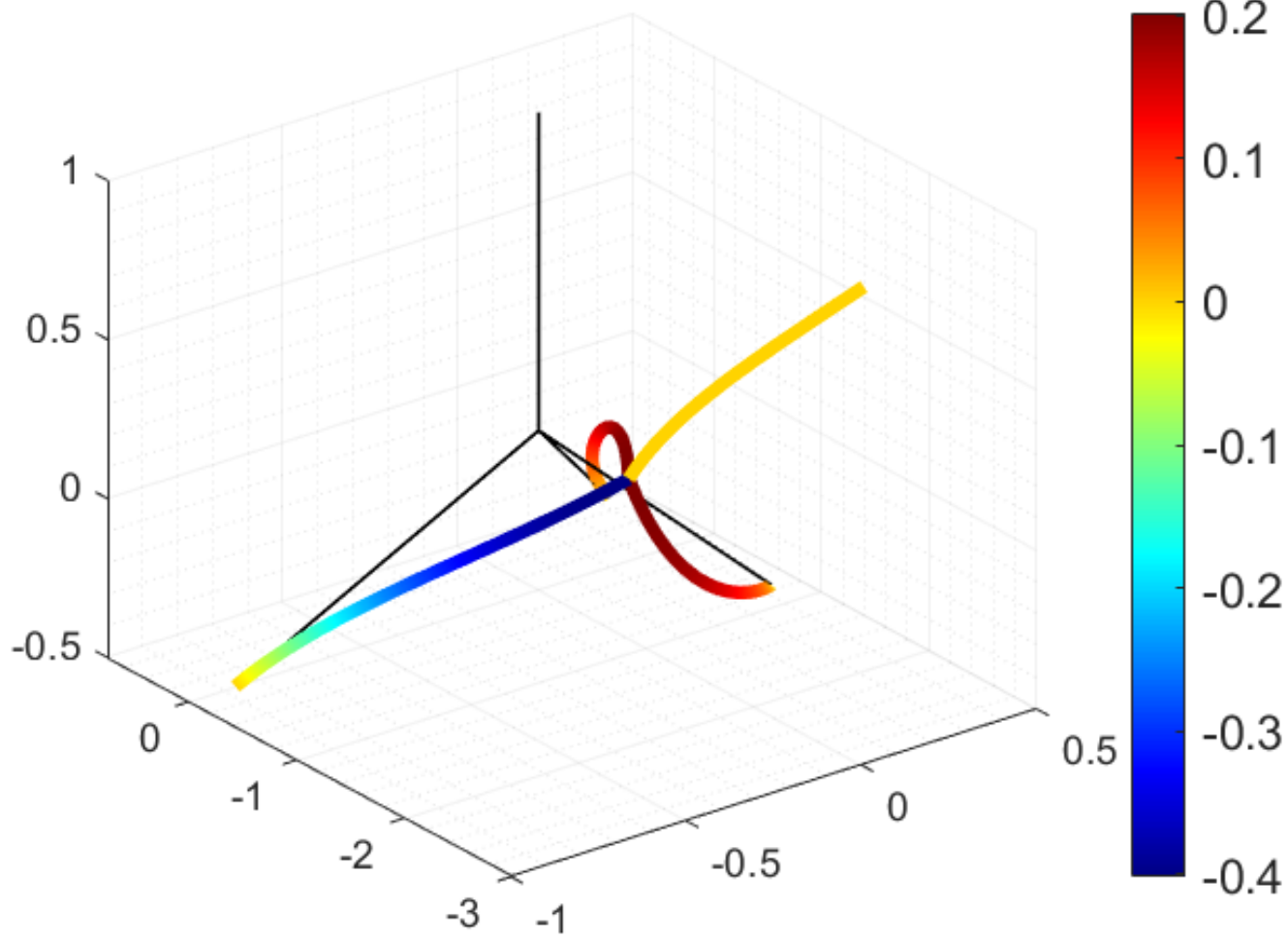}
    }
    \caption{Plot of the first and
      second eigenfunctions corresponding to the standard
      representation in global coordinate system. The results are obtained by a finite element
      approximation of the entire structure with 20 number of element for each beam. The angle $\alpha = \pi/6$, and unit values for all material's parameters and beam's lengths are selected. The  eigenvalues agree up to order $10^{-8}$ due to finite 
      number of discretization elements.}
    \label{f_ResultsNumWindows}
  \end{figure}
\end{exmp}

%%%%%%%%%%%%%%%%%%%%%%%%%%%%%%%%%%%%%%%%%%%%%%%%%%%%%%%%%%%%%%%%%%%%%%
\section{Outlook}
\label{sec:outlook}

In this section we give a partial list of topics that would benefit
from the spectral-theoretic approach.  The most obvious is to extend
our results to Timoshenko beams. Timoshenko model no longer assumes
the cross-sections remain orthogonal to the deformed axis and
therefore incorporates more degrees of freedom
\cite{Lagnese_multilink,MPZ02,GR15,Mei19}.  

Of similar applied interest is the analysis of pre-curved beams which
play an important role in application of highly flexible thin
structures.  Such beams may exhibit unusual and unexpected behavior,
and their spectral analysis does not seem to be explored fully.  Along
this applied path, extension of the current framework to the case of
semi-rigid joints, as well as the case of concentrated masses carried
by them may be interesting to the community (see, for example,
\cite{LLS93,CZ98,CZ00,GA13,TSOGM20,GM20} and references therein).

The viability of the frame model as a structure composed of
one-dimensional segments also needs to be verified mathematically, as
a limit of a three-dimensional structure as the beam widths are going
to zero.  This approach will also be of interest to structural
engineers working with FEM-packages for 3-D structural elements such
as Kichhoff--Love plates.  There is a significant mathematical
literature on this question for second-order operators (see, for
example, \cite{KucZen_jmaa01,Z02,Gri_incol08,Post_book12}), with some
recent advances for the fourth-order operators \cite{D89, LNS19,
  LNST19}.  Of particular interest may be the approach via general
scales of quadratic forms \cite{MNP13,PS19} which can be more readily
adapted to the present setting.  A variety of operators may be
expected to arise in the limit, including the operators with masses
concentrating at joints and other cases of applied interest.

Going in the other direction, it is interesting to investigate the
limit of the operator we discussed when some material parameters are
either very small or very large.  For example, consider taking the
angular displacement stiffness $d$ to 0 or to $\infty$ in the planar
case.  In both limits the out-of-plane displacement $v$ can be
expected to decouple from the rest of the degrees of freedom, but
with different vertex conditions at the joint (see \cite{GM20}
for a classification of self-adjoint conditions applicable to this
case).

One of the methods to ease the computation of eigenvalues for beam
frames is the Williams--Wittrick method
\cite{WilWit_ijms70,WitWil_qjmam71}, which, in mathematical terms,
uses a finite-rank perturbation to obtain a simpler system (decoupled
beams) whose eigenvalues bound the eigenvalues of the original frame.
Extending similar ``graph surgery'' methods summarized in
\cite{BerKenKurMug_tams19} to the present setting would result in
improved spectral estimates and answers to natural spectral
optimization questions.

Many of the real-life beam structures are periodic; for those,
Floquet--Bloch analysis \cite{Kuc_bams16,ZP16,KM21} is a natural tool to understand the spectrum and to eventually design structures that block certain ranges of frequencies. Along the same line, spectral analysis of a fourth-order operator on periodic hexagonal lattices was recently discussed in \cite{EH21}.

Finally, the positivity of the ground state is a feature that is
almost taken for granted in studies of the Laplacian and which is
broken in the example shown in Figure~\ref{dispEigenfunctionPlanar1}.
That forth-order operators may have non-positive ground states is a
well-know fact \cite{S01}.  The positivity is restored when the
angular displacement stiffness $d$ is large
(Figure~\ref{dispEigenfunctionPlanar2}), amplifying the second order
part of the operator compared to the fourth order, see
\eqref{eq:operatorH_planar_inplane}.  A careful extension of the
second-order theory on graphs \cite{Kur_lmp19} to the present case is
thus sorely needed.

\section*{Acknowledgment}

This work was partially supported by the NSF grant DMS-1815075 and BSF
grant No.~2016281.  The authors gratefully acknowledge enlightening
discussions with and improving suggestions by Delio Mugnolo and Maryam
Shakiba.  Numerous insightful suggestions by the referees helped us
improve the manuscript and give a more complete picture of the
existing literature.

%%%%%%%%%%%%%%%%%%%%%%%%%%%%%%%%%%%%%%%%%%%%%%%%%%%%%%%%%%%%%%%%%%%%%%
%%%%%%%%%%%%%%%%%%%%%%% APPENDIX %%%%%%%%%%%%%%%%%%%%%%%%%%%%%%%%%%%%%
%%%%%%%%%%%%%%%%%%%%%%%%%%%%%%%%%%%%%%%%%%%%%%%%%%%%%%%%%%%%%%%%%%%%%%
\appendix

\section{Rodrigues' rotation formula.}
\label{sec:Rodrigues}
Starting with the fact that $\cR \in \mathrm{SO}(3)$, then $\cR$ in \eqref{expRota} can be fully characterized by three free parameters, namely, axis of rotation represented by unit vector $\vec \vartheta(x)$ and angle of rotation $\alpha \in [0,\pi]$. Let define skew-symmetric matrix $\cK$ with components
\begin{equation}
\label{KmatrixLeviCivita}
\cK_{pq} = -\sum_{r=1}^{3}\epsilon_{pqr} \vartheta_r,
\qquad p,q \in \{1,2,3\},
\end{equation}
where $\vartheta_r$ is the $r$-the coordinate of
the vector $\vec \vartheta(x)$ in the global basis
$\{\vec E_1, \vec E_2, \vec E_3\}$, and $\epsilon_{pqr}$ is the
Levi-Civita alternating symbol defined as the
scalar triple product of orthonormal vectors in a right-handed system
$\{\vec e_1, \vec e_2, \vec e_3\}$ as
$\epsilon_{pqr} = \vec e_p \cdot (\vec e_q \times \vec e_r)$. The operator $\cR$ corresponding to $\vec \vartheta$ and $\alpha$ then has a form 
\begin{align}
\label{rotatMatrix}
\cR(x) = \bm{\mathbb{I}} + (\sin \alpha) \cK(x) + (1 - \cos \alpha) \cK^2(x)
\end{align}
with $\bm{\mathbb{I}}$ is the $3 \times 3$ identity matrix. Application of $\cR$ in \eqref{rotatMatrix} on vector $\vec a \in \bR^3$ will transform the vector to $\bm{\vec a}$ of the form 
\begin{align}
\label{rotVecActionGen}
\bm{\vec a} = \vec a \cos \alpha + \big(\vec \vartheta \times \vec a \big) \sin \alpha +
\vec \vartheta \big(\vec \vartheta \cdot \vec a\big) (1 - \cos \alpha)
\end{align}
which is the so-called Euler-Rodrigues formula. Replacing $\vec a$ with $\vec i$ in (\ref{rotVecActionGen}), then 
\begin{align}
\label{rotVecAction}
\bm{\vec i}(x) = \vec i \cos \alpha + \big(\vec \vartheta(x) \times \vec i \big) \sin \alpha +
\vec \vartheta(x) \big(\vec \vartheta(x) \cdot \vec i\big) (1 - \cos \alpha)
\end{align}
For small deformation, let introduce (linearized)-rotation vector
\begin{equation}
	\vec \omega(x) = \alpha \vec \vartheta(x)
\end{equation}
Then the leading terms of \eqref{rotVecAction}, i.e., by neglecting terms of order $\cO(\alpha^2)$, recovers \eqref{eq:iRot}. Similar discussion holds on recovering the leading terms in \eqref{eq:jRot} and \eqref{eq:kRot}. Following Lemma establishes connection of \eqref{rotatMatrix} with the exponential form in \eqref{expRota}. 
\begin{lem}
	\label{lem:rotationExp}
	Rotational matrix $\cR$ in (\ref{rotatMatrix}) with $\cK$ as (\ref{KmatrixLeviCivita}) can be represented in exponential form $\cR = \exp(\Omega)$, see \eqref{expRota}, with $\Omega = \alpha \cK$. 
\end{lem}
\begin{proof}[\normalfont \textbf{Proof of Lemma~\ref{lem:rotationExp}}] 
	Simple calculation on the components in (\ref{KmatrixLeviCivita}) implies that matrix $\cK$ has the form 
	\begin{align}
	\label{cKdefn}
	\cK = 
	\begin{pmatrix}
	\hspace{2mm}0& -\vartheta_3&\hspace{3mm}\vartheta_2\\
	\hspace{3mm}\vartheta_3 &\hspace{2mm}0&-\vartheta_1\\
	-\vartheta_2 & \hspace{3mm}\vartheta_1& \hspace{2mm}0
	\end{pmatrix}
	\end{align}
	Applying the property that $\vartheta$ is unit vector along with simple calculations confirms equality $\cK^3 = -\cK$. This in sufficient to derive higher powers of $\cK$, e.g. observe that $\cK^4 = \cK\cK^3 = -\cK^2$ and so on. To proceed, let $\kappa := (+1,-1,+1,\cdots)$ be an alternating sign sequence, then for $i = 1,2,\ldots$ 
	\begin{align}
	\label{characKi}
	\cK^{2i-1} = \kappa_i \cK, \quad \text{ and } \quad \cK^{2i} = \kappa_i \cK^2
	\end{align}
	Expanding $\sin \alpha$ and $\cos \alpha$ using Taylor series, then (\ref{rotatMatrix}) turns to
	\begin{align}
	\label{Rexpanded}
	\cR = \bm{\mathbb{I}}  + (\alpha - \frac{\alpha^3}{3!} + \cdots)\cK + (\frac{\alpha^2}{2!} - \frac{\alpha^4}{4!} + \cdots) \cK^2 = 
	\bm{\mathbb{I}} + \alpha \cK + \frac{\alpha^2}{2!}\cK^2 - \frac{\alpha^3}{3!}\cK - \frac{\alpha^4}{4!} \cK^2 + \cdots
	\end{align}
	Finally, replacing $\cK$ and $\cK^2$ by characterizations in terms of $\cK^{2i-1}$ and $\cK^{2i}$ stated in (\ref{characKi}) in (\ref{Rexpanded}), then 
	\begin{align}
	\cR = \bm{\mathbb{I}} + \alpha \cK + \frac{\alpha^2}{2!}\cK^2 + \frac{\alpha^3}{3!} \underbrace{(-\cK)}_{\cK^3} + \frac{\alpha^4}{4!} \underbrace{(-\cK^2)}_{\cK^4} + \cdots = \sum_{n \geq 0} \frac{(\alpha \cK)^n}{n!} = \exp(\alpha \cK)
	\end{align}
	The proof is complete by setting $\Omega = \alpha \cK$. 
\end{proof}

%%%%%%%%%%%%%%%%%%%%%%%%%%%%%%%%%%%%%%%%%%%%%%%%%%%%%%%%%%%%%%%%%%%%%%%%
\section{Tangent plane for deformed planar frame}
\label{sec:tangent_plane}
  
\begin{proof}[\normalfont \textbf{Proof of Lemma~\ref{equivConditions}}] 
	Following remark \ref{degTwoTrivial}, we assume degree of internal vertex satisfies $n \ge 3$. We first show the claim on equivalency of \eqref{rigidPlane_1} and the condition that all vectors $\{\bm{\vec i}_e(\vt)\}_{e \sim \vt}$ lie in the same plane, i.e. $\left( \bm{\vec i}_1(\vt) \times \bm{\vec i}_2(\vt) \right) \cdot \bm{\vec i}_e(\vt) = 0$ for all $e \sim \vt$. Setting $w'_e=0$ in the representation of $\bm{\vec i}_e(\vt)$ in \eqref{eq:i_from_g}, then 
	\begin{equation}
	\begin{split}
	\left( \bm{\vec i}_1(\vt) \times \bm{\vec i}_2(\vt) \right) \cdot \bm{\vec i}_e(\vt) = \left((\vec i_1 + v_1' \vec k) \times (\vec i_2 + v_2' \vec k) \right) \cdot (\vec i_e + v_e' \vec k)
	\end{split}
	\end{equation}
	Applying properties $(\vec i_1 \times \vec i_2)\cdot \vec i_e = 0$, $(\vec i_1 \times \vec k)\cdot \vec k = 0$, and $(\vec k \times \vec i_2)\cdot \vec k = 0$ implies that 
	\begin{equation}
	\begin{split}
	\left( \bm{\vec i}_1(\vt) \times \bm{\vec i}_2(\vt) \right) \cdot \bm{\vec i}_e(\vt) = v_1' \underbrace{(\vec k \times \vec i_2)}_{\vec j_2} \cdot \vec i_e +
	v_2'\underbrace{(\vec i_1 \times \vec k) \cdot \vec i_e}_{-\vec j_1 \cdot \vec i_e ~ = ~\vec j_e \cdot \vec i_1} + v_e'\underbrace{(\vec i_1 \times \vec i_2) \cdot \vec k}_{\vec i_2 \cdot (\vec k \times \vec i_1) ~=~ \vec i_2 \cdot \vec j_1}
 	\end{split}
	\end{equation}
	which is condition \eqref{rigidPlane_1} as it has been claimed. Next we process with the proof of the Lemma.  
	Firstly, in order to show $\eqref{primaryVertexCond2H1} \Rightarrow \eqref{rigidPlane_1}$, observe that from \eqref{primaryVertexCond2H1} and for $e = 1,2,\ldots,n$ adjacent to the internal vertex
	\begin{align*}
	\vec j_e \cdot (\eta_1 \vec i_1  - v_1' \vec j_1) = \vec j_e \cdot (\eta_e \vec i_e - v_e' \vec i_e) \quad  \Rightarrow \quad  (\vec j_e \cdot \vec j_1) v_1' -(\vec j_e \cdot \vec i_1) \eta_1 - v_e' = 0
	\end{align*}
	This then implies that the following linear system  
	\begin{align*}
	\begin{pmatrix}
	1 &0 &v_1'\\
	\vec j_2 \cdot \vec j_1 & \vec j_2 \cdot \vec i_1 &v_2'\\
	\vec j_e \cdot \vec j_1 & \vec j_e \cdot \vec i_1 &v_e'
	\end{pmatrix}
	\begin{pmatrix}
	+v_1'\\
	-\eta_1\\
	-1
	\end{pmatrix}
	= 
	\begin{pmatrix}
	0\\
	0\\
	0
	\end{pmatrix}
	\end{align*}
	holds. This implies that the multiplication matrix is rank deficient and thereby 
	\begin{align}
	(\vec j_2 \cdot \vec i_1)(\vec j_e \cdot \vec j_1) v_1' - (\vec j_2 \cdot \vec j_1)(\vec j_e \cdot \vec i_1)v_1' +  (\vec j_e \cdot \vec i_1) v_2'  + (\vec j_1 \cdot \vec i_2) v_e' = 0
	\end{align}   
	Above we use the relation $\vec j_2 \cdot \vec i_1 = - \vec j_1 \cdot \vec i_2$. It remains to show that $v_1'$ has the right multiplication coefficient, namely $\vec j_2 \cdot \vec i_s$. But
	\begin{align*}
	(\vec j_2 \cdot \vec i_1)\underbrace{(\vec j_e \cdot \vec j_1)}_{(\vec i_e \cdot \vec i_1)} - (\vec j_2 \cdot \vec j_1)\underbrace{(\vec j_e \cdot \vec i_1)}_{-(\vec i_e \cdot \vec j_1)} = \vec j_2 \cdot \big(\underbrace{(\vec i_e \cdot \vec i_1) \vec i_1 + (\vec i_e \cdot \vec j_1) \vec j_1}_{\vec i_e} \big) = \vec j_2 \cdot \vec i_e
	\end{align*}
	which proves (\ref{rigidPlane_1}). Secondly, to show that $(\ref{primaryVertexCond2H1}) \Rightarrow (\ref{rigidPlane_2})$, notice
	\begin{equation}
	\begin{split}
	\vec j_1 \cdot (\eta_1 \vec i_1  - v_1' \vec j_1) = \vec j_1 \cdot (\eta_e \vec i_e - v_e' \vec i_e) \quad  \Rightarrow \quad  (\vec j_1 \cdot \vec j_e) v_e' - (\vec j_1 \cdot \vec i_e) \eta_e - v_1'  &= 0 \\
	\vec j_2 \cdot (\eta_2 \vec i_2  - v_2' \vec j_2) = \vec j_2 \cdot (\eta_e \vec i_e - v_e' \vec j_e) \quad  \Rightarrow \quad  (\vec j_2 \cdot \vec j_e) v_e' - (\vec j_2 \cdot \vec i_e) \eta_e - v_2'  &= 0 
	\end{split}
	\end{equation}
	Similar argument on rank deficiency of coefficient matrix shows that 
	\begin{align*}
	(\vec j_2 \cdot \vec j_e) v_1' - (\vec j_e \cdot \vec j_1)v_2' +  (\vec j_1 \cdot \vec i_e)(\vec j_2 \cdot \vec j_e)\eta_e - (\vec j_1 \cdot \vec j_e)(\vec j_2 \cdot \vec i_e) \eta_e = 0 
	\end{align*}
	holds. But
	\begin{align}
	(\vec j_1 \cdot \vec i_e)(\vec j_2 \cdot \vec j_e) - (\vec j_1 \cdot \vec j_e)(\vec j_2 \cdot \vec i_e) = \vec j_1 \cdot \big((\vec i_2 \cdot \vec i_e) \vec i_e + (\vec i_2 \cdot \vec j_e) \vec j_e \big) = \vec j_1 \cdot \vec i_2
	\end{align}
	Compare with (\ref{rigidPlane_2}), this finish the desired result.
	Finally, it remains to show the reverse equality, i.e. $(\ref{rigidPlane_1}) ~\&~ (\ref{rigidPlane_2}) \Rightarrow (\ref{primaryVertexCond2H1})$.  Starting with (\ref{rigidPlane_1})
	\begin{equation}
	\label{vderivRe}
	\begin{split}
	(\vec j_1 \cdot \vec i_2) v_e' = -(\vec j_2 \cdot \vec i_e) v_1' - (\vec j_e \cdot \vec i_1) v_2' = 
	\vec j_e \cdot (v_1' \vec i_2 - v_2' \vec i_1)
	\end{split}
	\end{equation}
	Similarly from (\ref{rigidPlane_2}), 
	\begin{equation}
	\label{torsionRe}
	\begin{split}
	(\vec j_1 \cdot \vec i_2) \eta_e = -(\vec j_2 \cdot \vec j_e) v_1' + (\vec j_e \cdot \vec j_1) v_2' = 
	\vec i_e \cdot (v_2' \vec i_1 - v_1' \vec i_2)
	\end{split}
	\end{equation}
	This then implies that (assuming $\vec j_1 \cdot \vec i_2 \not = 0$)
	\begin{align}
	\label{vecM}
	\eta_e \vec i_e - v_e' \vec j_e =\frac{v_2'}{(\vec j_1 \cdot \vec i_2)} \vec i_1 - \frac{v_1'}{(\vec j_1 \cdot \vec i_2)} \vec i_2
	\end{align}
	holds for all $e \sim \vt$. Since the right hand side of (\ref{vecM}) is independent of $e$, then the result in (\ref{primaryVertexCond2H1}) is clear.
\end{proof}

\bibliographystyle{abbrv}
\bibliography{additional_beams,bk_bibl}

\end{document}